\documentclass[12pt]{article}
\newcommand{\acks}[1]{\section*{Acknowledgments}#1}
\usepackage{blindtext}
\usepackage[round]{natbib}
\usepackage{amsthm}
\usepackage{amssymb}
\usepackage[colorlinks=true,
            linkcolor=blue,
            citecolor=blue,
            urlcolor=blue]{hyperref}

\usepackage[linesnumbered,ruled,vlined]{algorithm2e}
\SetKwInput{KwInput}{Input}              \SetKwInput{KwOutput}{Output}

\usepackage{array}
\usepackage[caption=false,font=normalsize,labelfont=sf,textfont=sf]{subfig}
\usepackage{textcomp}
\usepackage{stfloats}
\usepackage{url}
\usepackage{verbatim}
\usepackage{bm}
\usepackage{booktabs}
\usepackage{bbm}
\usepackage{mathrsfs}
\usepackage{amsfonts}
\usepackage{amsmath}
\usepackage{graphicx}

\usepackage{placeins}
\usepackage{rotating}

\usepackage{xcolor}

\usepackage{multirow}
\usepackage{lscape}

\usepackage[noend]{algpseudocode}

\DeclareMathOperator*{\argmin}{arg\,min}

\makeatletter
\def\BState{\State\hskip-\ALG@thistlm}
\makeatother

\usepackage{enumitem}


\newtheorem{theorem}{Theorem}

\newtheorem{lemma}{Lemma}
\newtheorem{proposition}{Proposition}
\newtheorem{corollary}{Corollary}
\newtheorem{remark}{Remark}

\usepackage{lastpage}

\usepackage{lastpage}

\usepackage{lineno}
\usepackage{setspace}
\usepackage[margin=0.7in]{geometry}
\title{BOOOM: Loss-Function-Agnostic Black-Box Optimization over Orthonormal Manifolds for Machine Learning and Statistical Inference}

\author{
Beomchang Kim$^{1}$ \and
Subhrajyoty Roy$^{2}$ \and
Priyam Das$^{1}$ \\[0.5em]
$^{1}$Department of Biostatistics, Virginia Commonwealth University, \\
Richmond, VA, USA \\
$^{2}$Department of Statistics and Data Science, Washington University in St.\ Louis, \\
St.\ Louis, MO, USA
}

\date{}
       
\begin{document}
\maketitle
\begin{abstract}
Optimization over the Stiefel manifold $\mathrm{St}(p,d)$, the set of $p \times d$ column-orthonormal matrices, is fundamental in statistics, machine learning, and scientific computing, yet remains challenging in the presence of non-convex, non-smooth, or black-box objectives. Existing methods largely rely on either convex relaxations or gradient-based Riemannian optimization, limiting applicability in derivative-free and highly multimodal settings. We propose \textsc{BOOOM} (Black-box Optimization Over Orthonormal Manifolds), a general-purpose framework for loss-function-agnostic optimization on $\mathrm{St}(p,d)$. The key idea is a global Givens rotation-based parametrization that maps the manifold to an unconstrained Euclidean angle space while preserving feasibility exactly. Building on this representation, BOOOM employs a structured, parallelizable, derivative-free search based on Recursive Modified Pattern Search, enabling systematic exploration through plane-wise rotations without requiring gradient information and facilitating escape from poor local optima. We establish a unified theoretical framework showing equivalence between angle-space and manifold optimization, transfer of stationarity, and global convergence in probability under mild conditions. Empirical results across diverse problems, including heterogeneous quadratic optimization, low-rank and sparse matrix decomposition, independent component analysis, and orthogonal joint diagonalization, among other widely studied settings, demonstrate strong performance relative to state-of-the-art methods, particularly in non-smooth and highly multimodal regimes. We further illustrate its practical utility through a novel supervised PCA formulation applied to metabolomics data in colorectal cancer.
\end{abstract}


\clearpage
\section{Introduction}\label{sec:intro}

We consider optimization problems of the form
\begin{equation}
    \min_{Q \in \mathrm{St}(p, d)} f(Q),
    \label{eq:stiefel_problem}
\end{equation}
\noindent where 
\begin{equation}
    \mathrm{St}(p,d)=\{Q\in\mathbb{R}^{p\times d}: Q^\top Q = I_d\}
    \label{eqn:steifel-manifold}
\end{equation}
\noindent is the set of $p \times d$ matrices with orthonormal columns. Optimization over the Stiefel manifold $\mathrm{St}(p, d)$ is a central problem in modern machine learning, statistics, and scientific computing. Its importance stems from the pervasive role of orthogonality constraints in modeling, identifiability, and dimensionality reduction. Some classic and prominent examples include independent component analysis (ICA) \citep{hyvarinen2000independent, hyvarinen1999fast, bell1995information, Amari1995}, orthogonal joint diagonalization (OJD) \citep{cardoso1996jacobi, pham2001joint}, orthogonal factor rotations such as varimax \citep{kairser1958varimax, jennrich2001simple}, principal component analysis (PCA) and its variants \citep{jolliffe2002principal}, as well as low-rank matrix estimation and robust PCA formulations \citep{candes2011robust, netrapalli2014nonconvex, roy2024robust}. In scientific computing, similar structures arise in Rayleigh-Ritz type eigenvalue problems and electronic structure calculations, including Kohn-Sham density functional theory \citep{kohn1965self, Yang2009kssolv, Jiao2022kssolv, luo2025direct} as well as subspace methods in numerical linear algebra \citep{saad2011numerical}. These problems typically exhibit highly non-convex landscapes characterized by multiple local minima, rotational invariances, and strong coupling across variables.

To accommodate this diverse range of applications under a single optimization framework, the objective function $f:\mathrm{St}(p,d) \to \mathbb{R}$ is assumed to be completely general: it may be non-convex, non-smooth, discontinuous, or available only through function evaluations (black-box setting), with no access to gradient or higher-order information. The black-box setting has also emerged in various contexts in modern machine learning, deep learning, and computing scenarios. Consider the following motivating examples:
\begin{enumerate}
    \item Recent literature has shown that it is possible to construct an adversarial example to a deep neural network model (including the modern multi-modal models such as Vision transformers) by means of rotation of the original image~\citep{athalye2018synthesizing, jiang2026diversifying}. Typically, producing these attacks requires performing a projected gradient descent (PGD)~\citep{bryniarski2021evading}, hence requiring access to the model parameters. However, for a closed-source model, the adversary will maximize the loss under orthogonal perturbation of the input using only functional evaluations (API access) of the model.
    \item One of the fundamental limitations of classical recurrent neural network (RNN) architectures is the vanishing and exploding gradient problem, which makes training unstable, particularly for long sequences. Orthogonal neural networks, where the weight matrices are constrained to be orthogonal, have been proposed as an effective remedy since orthogonality preserves the norm of the propagated signal and improves gradient flow~\citep{vorontsov2017, huang2018orthogonal, li2019orthogonal}. Similar orthogonality constraints have also shown promising performance in modern transformer-based architectures~\citep{fei2022vit}. Despite these advantages, such networks have not seen widespread adoption in practice because enforcing orthogonality during training is notoriously difficult, often requiring architecture-specific surrogate parameterizations on which backpropagation can be performed. An objective-agnostic optimization framework that relies only on black-box function evaluations offers an alternative approach, enabling the training of orthogonality-constrained networks without explicit backpropagation and using only forward passes of the model.
    \item In scientific quantum computing, transformations between qubits are represented by orthogonal (more precisely, unitary) matrices~\citep{nielsen2010quantum}. Designing algorithms for quantum machine learning, therefore, often requires optimization over orthogonality-constrained parameter spaces. In practice, however, the instability and noise inherent in quantum systems make it difficult to evaluate the objective function and its gradients accurately. In many cases, only a noisy Monte Carlo estimate of the objective is available through repeated simulation or measurement, which makes gradient-based optimization challenging and motivates the use of black-box optimization methods. 
\end{enumerate}

As an illustration of the first example, consider a pretrained deep neural network for image classification (e.g., a ResNet~\citep{he2016deep}) accessed purely as a black-box, where only forward predictions are available and no information about gradients or internal architecture is exposed. Let $g(x)=(p_1(x),\dots,p_C(x))$ denote the class probabilities produced by the network for an input $x\in\mathbb{R}^p$. Let $\mathcal{R}(U;x)$ denote the transformed input obtained by applying the orthogonal parameter $U\in\mathrm{St}(p,k)$ (e.g., via a rendering or viewpoint transformation). We define
\[
f(U):=p_{y^\star}(\mathcal{R}(U;x))-\max_{1\le j\le C} p_j(\mathcal{R}(U;x)), \qquad U\in\mathrm{St}(p,k),
\]
where $y^\star$ denotes the true class. This objective is always nonpositive and satisfies $-1 \le f(U) \le 0$, as it is the difference between two probabilities. It measures the gap between the true-class confidence and the largest predicted confidence. Minimizing $f(U)$ yields adversarial (worst-case) views where the true class is strongly suppressed, while maximizing $f(U)$ seeks one of the best views; the largest possible value is $0$, attained when the true class is the most probable label. Since only evaluations of $f(U)$ are available and gradients are inaccessible, this forms a genuine black-box optimization problem over the Stiefel manifold. \textsc{BOOOM} addresses this setting through derivative-free search using structured Givens rotations, which we further detail in later sections. Figure~\ref{fig:resnet_blackbox} illustrates how such orthogonal transformations can substantially alter model predictions.
\begin{figure}[!h]
\centering
\includegraphics[width=0.95\textwidth]{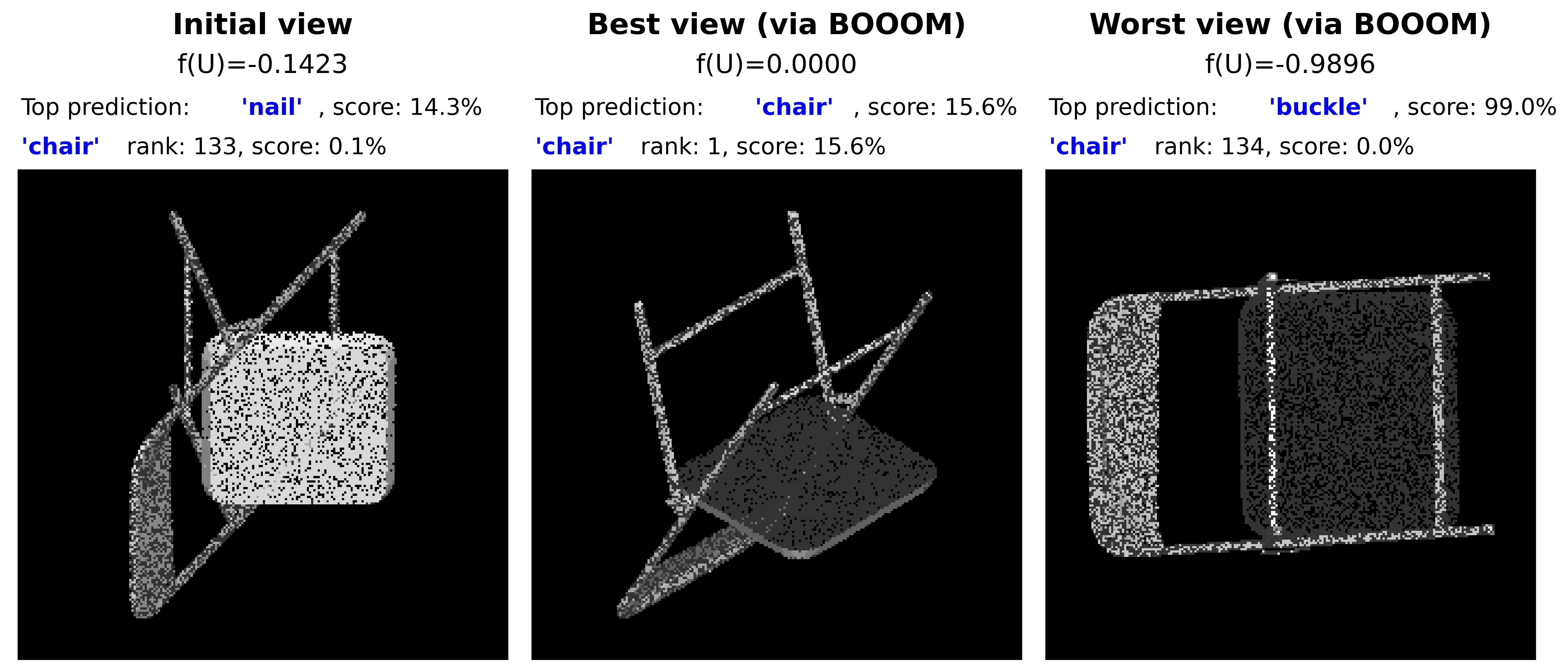}
\caption{
Black-box optimization over orthogonal transformations in a pretrained image classifier (ResNet). 
The objective $f(U)$ is evaluated using only forward passes of the network, without access to gradients or model internals. Starting from an initial view (left), \textsc{BOOOM} explores $U \in \mathrm{St}(p,k)$ to identify transformations that minimize $f(U)$ (worst-case view, right) and maximize $f(U)$ (one of the best-case views, center, where $f(U)\approx 0$). The figure demonstrates that structure-preserving orthogonal transformations can induce highly non-linear and non-convex changes in model predictions, underscoring the need for derivative-free black-box optimization on orthogonality-constrained domains.
}
\label{fig:resnet_blackbox}
\end{figure}

\subsection{Challenges and Existing Literature}\label{sec:lit-review}

Despite this ubiquity, optimization over $\mathrm{St}(p,d)$ remains fundamentally challenging. The feasible set is nonconvex, and even simple quadratic objectives can exhibit highly multimodal landscapes with numerous local minima and saddle points \citep{Edelman1998}. This is because many problems defined on $\mathrm{St}(p,d)$ exhibit substantial symmetry, such as sign, permutation, or rotational invariances, which lead to objective landscapes with flat regions and multiple equivalent optima. These difficulties are heightened in modern applications, where objective functions are often non-smooth, discontinuous, noisy, or defined implicitly through simulation or composite pipelines. In such black-box settings, derivative information may be unavailable, unreliable, or computationally prohibitive to obtain \citep{Conn2009,Rios2013}. Furthermore, the geometry of the Stiefel manifold introduces additional complications: iterates must remain orthonormal at all times, requiring projections, retractions, or structured parametrizations, each of which introduces computational overhead or distorts the search dynamics.

A substantial body of work has addressed these challenges through Riemannian optimization, which generalizes classical first- and second-order methods to manifold settings \citep{AbsilMahonySepulchre2008,Boumal2023}. Algorithms such as Riemannian gradient descent, conjugate gradient, and trust-region methods \citep{boumal2014manopt,Edelman1998} exploit local geometric structure and are effective when smoothness and derivative information are available; however, they are fundamentally local in nature and often sensitive to initialization, particularly in highly non-convex settings. In particular, they are designed to follow descent directions and therefore lack intrinsic mechanisms for escaping local minima or exploring distant regions of the search space. Alternative strategies based on parametrizations, including those built on Givens rotations or other structured factorizations, have also been explored in contexts such as orthogonal matrix optimization, joint diagonalization, and factor rotations \citep{cardoso1996jacobi, pham2001joint, jennrich2001simple, Edelman1998}. While such parametrizations can ensure feasibility and provide interpretable updates, existing methods typically rely on gradient information or problem-specific structure, and similarly do not incorporate systematic exploration mechanisms to traverse multiple basins of attraction. As a result, they are not well-suited to problems where the objective is non-differentiable.

An alternative line of work focuses on problem-specific algorithms or convex relaxations. For instance, semidefinite programming relaxations have been developed for certain classes of quadratic optimization problems \citep{burer2003nonlinear,gilman2025semidefinite}, while convex surrogates are widely used in robust PCA and related formulations \citep{candes2011robust, guo2017godec+}. Although effective in tailored settings, these approaches are not general-purpose and often fail to extend to heterogeneous objectives or large-scale problems. In parallel, stochastic metaheuristics such as Genetic Algorithms (GA) and Simulated Annealing (SA) provide mechanisms for global exploration \citep{Fraser1957,kirkpatrick1983optimization}, but typically suffer from poor scalability and limited theoretical guarantees in high-dimensional spaces \citep{Geris2012}. Moreover, extensions of such metaheuristic frameworks to orthogonality-constrained settings, such as optimization over the Stiefel manifold, remain relatively underexplored to the best of our knowledge, as maintaining feasibility and constructing effective search operators under nonconvex matrix constraints is challenging.

Taken together, these landscapes reveal a gap in literature: there is currently no unified optimization framework for the Stiefel manifold that is simultaneously objective-function-agnostic, strictly feasibility-preserving, capable of global exploration beyond local descent, and supported by rigorous theoretical guarantees. We address this gap below by proposing a novel optimization framework, called ``BOOOM''.

\subsection{Main contributions}\label{sec:booom-contributions}

In this work, we introduce \textsc{BOOOM} (Black-box Optimization Over Orthonormal Manifolds), a general-purpose framework for derivative-free optimization over $\mathrm{St}(p,d)$ that directly addresses the challenges outlined above. The central idea is to reparameterize the manifold using a complete sequence of Givens rotations. Building on classical factorization results for orthogonal matrices \citep{Hurwitz1963,jiang2022givens}, we show that any $Q \in\mathrm{St}(p,d)$ can be expressed as a finite product of planar rotations applied to a fixed orthonormal base. This induces a smooth, surjective mapping from a Euclidean angle space to the manifold, thereby transforming the constrained optimization problem into an unconstrained search in $\mathbb{R}^{\binom{p}{2}}$ while preserving feasibility exactly at every iterate. As a consequence, the optimization over the space of orthonormal matrices can be carried out entirely in angle space without ever violating the constraints. Unlike projection- or retraction-based methods, usage of this reparameterization eliminates the need for constraint enforcement, all the while maintaining the full expressivity of the original problem, ensuring that no feasible solutions are excluded.

On top of this parametrization, \textsc{BOOOM} employs a derivative-free optimization strategy based on Recursive Modified Pattern Search (RMPS) \citep{kim2026smart, Das2023MsiCOR, Das2023RMPSH, Das2022, Das2021}, a class of direct search methods rooted in classical pattern search and coordinate polling schemes~\citep{torczon1997convergence,Kolda2003}. At each iteration, the algorithm explores candidate solutions by applying planar rotations corresponding to coordinate-wise perturbations in angle space, resulting in small Givens rotations applied to the current evaluations. It systematically explores all rotation planes, and naturally supports parallel computation. Additionally, the algorithm uses an adaptive step-size refinement that allows for a balance between local and global exploration.

We explore the performance of the BOOOM algorithm both theoretically and empirically. The key theoretical results include demonstrating the existence of the angle-space parametrization, the equivalence between the lack of descent in angle space and Riemannian stationarity on the Stiefel manifold, and characterizing the convergence behavior at both local and global scales. We also evaluate \textsc{BOOOM} across a broad spectrum of problems spanning machine learning, statistics, and scientific computing. These include classical multimodal benchmark functions adapted to orthogonality constraints, heterogeneous quadratic optimization, low-rank plus sparse matrix decomposition, independent component analysis, orthogonal joint diagonalization, and Kohn--Sham Rayleigh--Ritz eigenvalue problems. Across these diverse settings, \textsc{BOOOM} consistently achieves competitive or superior objective values compared to state-of-the-art Riemannian optimization methods \citep{boumal2014manopt} and problem-specific baselines, particularly in non-smooth and highly multimodal regimes where traditional approaches struggle.

In summary, the main contributions of this work are:
\begin{itemize}
    \item Proposing a novel, parallelizable, general-purpose, objective-function-agnostic optimization framework \textsc{BOOOM} for orthogonality-constrained problems operating solely via function evaluations.
    \item A unified theoretical framework establishing stationarity, global convergence in probability, and per-run efficiency guarantees.
    \item Extensive empirical validation across diverse nonconvex, non-smooth, and application-driven problem settings, spanning across multiple fields.
    \item Demonstration of BOOOM as a flexible optimization engine for novel objective formulations, exemplified via a supervised PCA framework for identifying key metabolites governing colorectal cancer status.
\end{itemize}

The remainder of the paper is organized as follows. In Section~\ref{sec:prelim}, we introduce the proposed BOOOM framework, including the underlying Givens rotation-based parametrization and the RMPS-based optimization strategy. Section~\ref{sec:theory} develops the theoretical properties of the method, establishing connections between angle-space optimization and Riemannian stationarity, along with global convergence guarantees. In Sections~\ref{sec:class_bench} and~\ref{sec:bench_exps}, we present extensive empirical evaluations across classical benchmark functions and a diverse set of application-driven problems on the Stiefel manifold. Finally, in Section~\ref{sec:case_study}, we illustrate the practical utility of BOOOM through a real-data application to metabolomics in colorectal cancer, followed by discussion, concluding remarks, and directions for future research in Section~\ref{sec:discussion}.

\section{Preliminaries and Algorithm Description}\label{sec:prelim}

To facilitate a coherent development of the proposed methodology, we first discuss several key preliminary concepts that constitute its foundation. 

\subsection{Global optimization and the exploration--exploitation trade-off}

Global optimization methods typically address optimization of a non-convex objective function with multiple local minima by combining a broad search with local refinements. In contrast, purely local methods usually follow descent directions and may become trapped in suboptimal regions. A central framework in this context is the \emph{exploration vs.\ exploitation} trade-off, widely studied in reinforcement learning and search theory \citep{BergerTal2014, Zhang2023}. In numerical optimization, this trade-off can be conceptually expressed as a randomized iterative procedure,
\begin{equation*}
    \text{Next step} = \begin{cases}
        \text{Exploit current solution by localized search, } & \text{ with probability } \gamma,\\
        \text{Explore a new solution globally, }  & \text{ with probability } (1-\gamma),
    \end{cases}
\end{equation*}
\noindent where $\gamma \in (0, 1)$ controls the emphasis on local descent versus global search. As an example, gradient-based methods such as gradient descent~\citep{ruder2016overview} pick $\gamma = 1$ to ensure faster convergence, matrix manifold optimizations including Riemannian gradient and trust-region algorithms \citep{AbsilMahonySepulchre2008, boumal2014manopt, Boumal2023} can be used to obtain local solutions. In contrast, global optimization methods, including metaheuristics such as Genetic Algorithms (GA; \citealp{Bethke1980}) and Simulated Annealing (SA; \citealp{kirkpatrick1983optimization}), in principle, dynamically balance exploration and exploitation, favoring exploration ($\gamma$ closer to $0$) to escape local optima and shifting toward local refinement ($\gamma$ closer to $1$) within promising regions. 
Despite this flexibility, metaheuristic methods often suffer from poor scalability in high dimensions due to the exponential growth of the search space~\citep{Geris2012} combined with a deterministic exploration strategy. A randomized strategy that adaptively balances the choice of $\gamma$ based on the quality of the current solution would be able to borrow strengths from both of these paradigms.

\subsection{Fermi's principle, Pattern Search, and RMPS}\label{sec:fermi-principle}

One of the earliest derivative-free strategies for black-box optimization is \emph{Fermi’s principle}~\citep{Fermi1952}, which performs coordinate-wise exploration of the search space. In an $N$-dimensional unconstrained setting, this principle evaluates the objective at $2N$ axis-aligned candidate points of the form
\begin{equation*}
    x \pm s e_m, \ m = 1, 2, \dots, N,
\end{equation*}
\noindent where $s>0$ is a step size and $\{e_m\}_{m=1}^N$ are the canonical basis vectors. At each iteration, the best candidate replaces the current iterate. The step size $s$ controls the scale of exploration: large values promote global movement, while small values enable local refinement. Figure~\ref{fig:Fermi's_Principle} illustrates this coordinate-wise exploration for a fixed step size. When combined with a decreasing sequence of step sizes, this strategy yields convergence to local optima without requiring derivative information.

\begin{figure}[htbp]
	\centering
	\includegraphics[width=0.95\textwidth]{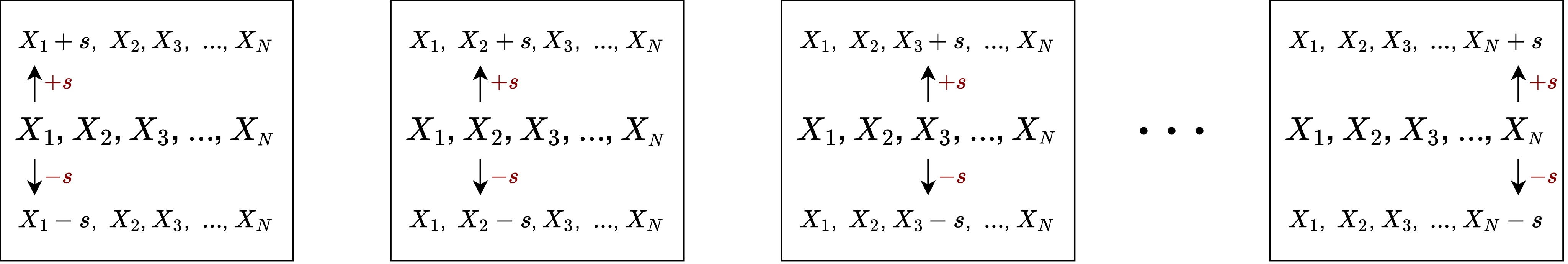}
	\caption{Fermi's principle: possible $2N$ coordinate-wise movements from a current point in $\mathbb{R}^N$ with fixed step size $s$.}
	\label{fig:Fermi's_Principle}
\end{figure}

Fermi’s principle forms the basis of Pattern Search (PS) methods and related direct search algorithms \citep{torczon1997convergence, Kolda2003}. These methods iteratively explore structured sets of candidate points and adapt the step size based on improvement, providing a robust framework for optimizing non-differentiable or discontinuous functions. While PS methods incorporate a degree of exploration through coordinate polling, they remain fundamentally local: once the step size becomes small, the search is confined to a neighborhood of the current solution, and there is no inherent mechanism to escape local minima.

To address this limitation, \cite{Das2023RMPSH} proposed the \emph{Recursive Modified Pattern Search (RMPS)} principle, which extends classical PS in two key ways. First, RMPS employs an adaptive step-size reduction scheme that ensures progressive refinement while maintaining stability. Second, it introduces a restart mechanism, whereby the algorithm is periodically reinitialized (typically from the best solution found so far) with a large step size, enabling renewed exploration of the search space. This combination allows RMPS to balance local convergence with global exploration in a principled manner. RMPS has demonstrated strong empirical performance across a variety of non-convex optimization problems and constrained domains, including spheres, simplices, and multi-simplex structures \citep{kim2026smart, Das2022, Das2021}, often outperforming classical metaheuristic black-box methods such as GA and SA. RMPS has been shown to outperform context-specific conventional algorithms, such as the Expectation–Maximization (EM) algorithm, in mixture Markov clustering problems \citep{Das2023MsiCOR}. It has also been successfully applied in statistical and applied settings such as mixture modeling, finance, and biomedical applications \citep{tan2020, Das2017b, Das2023}. These properties make RMPS a natural foundation for developing scalable, derivative-free optimization methods in more structured spaces, which we leverage in constructing the optimization method over $\mathrm{St}(p,d)$.

\subsection{Exploration on the Stiefel manifold via Givens rotations}\label{sec:exploration-givens}

While RMPS is naturally formulated in Euclidean spaces, directly applying coordinate-wise perturbations to \eqref{eq:stiefel_problem} is nontrivial due to the orthogonality constraint. Naive updates of the form $Q \mapsto Q \pm s E$ (where $E$ denotes a canonical basis matrix) generally violate the constraint $Q^\top Q = I_d$, necessitating projections or retractions that can be computationally expensive and may distort the underlying search geometry. To address this challenge, we construct a parametrization of the Stiefel manifold that enables structured, coordinate-wise exploration while preserving feasibility exactly. The key idea is to reinterpret Euclidean coordinate-wise search as plane-wise rotations on the manifold. Specifically, given a current iterate $Q \in \mathrm{St}(p,d)$, we generate candidate solutions by applying Givens rotations $R_{i,j}(\theta) \in \mathbb{R}^{p \times p}$, which rotate the rows of $Q$ within the two-dimensional subspace spanned by indices $(i,j)$ with $1 \le i < j \le p$. A rotation angle $\theta > 0$ corresponds to a counterclockwise rotation in the $(i,j)$-plane, while $\theta < 0$ induces a clockwise rotation. These transformations preserve orthonormality by construction, ensuring that all candidates remain on $\mathrm{St}(p,d)$ without the need for projection. Formally, each Givens rotation has the form
\begin{equation}
    R_{i,j}(\theta)=
    \begin{bmatrix}
    1 & \cdots & 0 & \cdots & 0 & \cdots & 0\\
    \vdots & \ddots & \vdots &  & \vdots &  & \vdots\\
    0 & \cdots & \cos\theta & \cdots & -\sin\theta & \cdots & 0\\
    \vdots &  & \vdots & \ddots & \vdots &  & \vdots\\
    0 & \cdots & \sin\theta & \cdots & \cos\theta & \cdots & 0\\
    \vdots &  & \vdots &  & \vdots & \ddots & \vdots\\
    0 & \cdots & 0 & \cdots & 0 & \cdots & 1
    \end{bmatrix}_{p\times p}.
    \label{eqn:givens-rotation}
\end{equation}
\noindent Applying $R_{i,j}(\theta)$ to $Q$ updates only the $i$-th and $j$-th rows while leaving all other rows unchanged, thereby preserving the orthonormal column structure. Since there are $\binom{p}{2}$ distinct index pairs, this construction yields $2\binom{p}{2}$ candidate moves at each iteration, corresponding to rotations by $\pm \theta$ along each coordinate direction in the induced parameter space.

Figure~\ref{fig:Fermi's_Principle_Stiefel_manifold} illustrates the proposed extension of Fermi’s principle to the Stiefel manifold. In contrast to Euclidean coordinate perturbations, which act independently along axes, performing these structured rotations via the Givens rotation operator as in~\eqref{eqn:givens-rotation} ensures that the iterates remain entirely within the feasible set. This establishes a direct analogue of coordinate-wise exploration in a curved, constrained space.

\begin{figure}[htbp]
	\centering
	\includegraphics[width=0.9\textwidth]{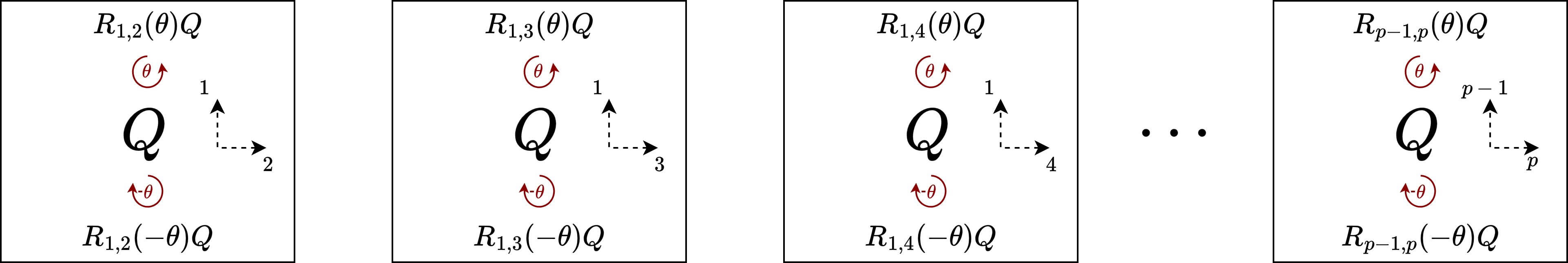}
	\caption{Fermi's principle on $\mathrm{St}(p,d)$: possible $2\binom{p}{2}$ plane-wise rotational moves from a current matrix $Q_{p\times d}$ using a fixed angle $\theta$. Positive $\theta$ corresponds to counterclockwise rotation, while negative $\theta$ corresponds to clockwise rotation.}
	\label{fig:Fermi's_Principle_Stiefel_manifold}
\end{figure}

\subsection{Overview of BOOOM}\label{sec:algo-overview}

Building on the Givens rotation-based exploration introduced in the previous subsection, the proposed framework BOOOM (Black-box Optimization Over Orthonormal Manifolds) applies an RMPS-style search over the induced parameter space while maintaining feasibility on $\mathrm{St}(p,d)$ at every step. The algorithm proceeds in \emph{runs}, each consisting of multiple iterations with progressively refined exploration. Within each run, BOOOM begins with a relatively large rotational step size $s$, enabling broad exploration of the search space. At subsequent iterations, the step size is geometrically decreased to facilitate local refinement. Once a run terminates, the next run is initialized at the best solution obtained so far, with the step size reset to a large value. This restart mechanism promotes renewed exploration and helps the algorithm escape local minima. BOOOM terminates when successive runs yield negligible improvement in the objective value.

\noindent\textbf{BOOOM parameters:} The BOOOM algorithm is governed by several parameters. In addition to the standard optimization parameters, such as the maximum number of iterations (max\_iter), the maximum number of runs (max\_runs), within-run convergence threshold $\tau_1 > 0$ and between-runs progress threshold $\tau_2 > 0$, BOOOM also uses an initial step size $s_{\textit{initial}} > 0$, step-size decay rate $\rho > 1$, step-size threshold $\phi > 0$. These parameters allow the user to balance coarse exploration (large $s_{\textit{initial}}$, small $\rho$) with fine local refinement (small $\phi$, tighter $\tau_1$ and $\tau_2$).

\vspace{0.2cm}
\noindent\textbf{Exploratory movements:} At the start of the $R$-th run, BOOOM initializaes at $Q^{(0)} = \hat{Q}^{(R-1)}$, the best solution obtained from the previous run. It starts with a step size $s^{(0)} = s_{\textit{initial}}$; a typical value is chosen at $\pi$ to ensure exploration of a broad neighbourhood around the current solution, increasing the likelihood of escaping a local minima. For the first run $R = 1$, the algorithm starts from a randomly generated orthonormal matrix. At each iteration $h$ within a run $R$, BOOOM constructs candidate solutions by applying Givens rotations to the current iterate $Q^{(h-1)}$. For each pair $(i,j)$ with $1 \le i < j \le p$, two candidates are generated using rotations with angles $\pm s$, producing a total of $p(p-1)$ candidates:
\begin{equation*}
    Q^{(h)}(i,j,\pm s) = R_{i,j}(\pm s)\, Q^{(h-1)}.    
\end{equation*}
\noindent The objective function is evaluated at all candidates, and the best-performing one is selected as the next iterate if it improves upon the current solution. The step size $s^{(h)}$ is updated adaptively: it is reduced to $s^{(h)}/\rho$ if no sufficient improvement is observed, or retained otherwise.

\vspace{0.2cm}
\noindent\textbf{Convergence criterion:} BOOOM monitors the improvement both across iterations within a run, as well as between two solutions obtained from two successive runs. If the change in objective value satisfies $|f(Q^{(h-1)}) - f(Q^{(h)})| < \tau_1$, the step size is reduced. Iterations continue until the step size falls below the threshold $\phi$, at which point the run terminates and outputs $\widehat{Q}^{(R)}$ as the current solution. At the end of each run, BOOOM compares the solutions at two successive runs by an adequate convergence criteria, e.g., the algorithm may terminate if 
\begin{equation*}
    |f(\widehat{Q}^{(R)}) - f(\widehat{Q}^{(R-1)})| < \tau_2.
\end{equation*}

\vspace{0.2cm}
\noindent \textbf{Parallelization:} For a fixed step size, the $2\binom{p}{2} = p(p-1)$ candidate evaluations at each iteration are independent, since each candidate depends only on the current iterate $Q^{(h-1)}$. As a result, these evaluations can be carried out in parallel across multiple CPU cores. Additionally, certain ``black-box'' objective functions allow scalable vectorized operations or processing of batch inputs, which can also be leveraged. This significantly reduces computational time of the proposed algorithm even when $p$ is large, allowing BOOOM to efficiently explore a large candidate set at each iteration.

Let $\zeta:\{1,\ldots,\binom{p}{2}\} \to \{(i,j)\in \mathbb{N}^2 : i<j\}$ be the enumeration mapping that associates each coordinate with a unique pair $(i, j)$ for $1 \leq i < j \leq p$, i.e., $\zeta(1) = (1,2), \zeta(2) = (1,3), \ldots, \zeta(\binom{p}{2}) = (p-1,p)$. Equipped with this notation and based on the above discussions, the complete pseudocode of the BOOOM algorithm is presented in Algorithm~\ref{euclid}. A schematic overview of the algorithmic flow, including the parallel evaluation of candidate rotations and the within-run and across-run update logic, is illustrated in Figure~\ref{BOOOM_concept}.

\begin{algorithm}[htbp]
    \caption{BOOOM}\label{euclid}
    \KwInput{Initial guess for orthogonal matrix ${Q}_{p \times d}$}
    \KwOutput{$\widehat{Q}$; BOOOM optimized solution $p \times d$ orthonormal matrix}
    \textbf{Initialization:} $R \gets 1$\tcc*[r]{run index}
    $N \gets \binom{p}{2}$\;
    \If{$R = 1$}{
        ${Q}^{(0)} \gets \text{Initial guess}$
        \tcc*[r]{${Q}^{(h)}$ is the value of ${Q}$ at $h$-th iteration}
    }
    \Else{
        ${Q}^{(0)} \gets \widehat{{Q}}^{(R-1)}$
        \tcc*[r]{$\widehat{{Q}}^{(R)}$ is the value of ${Q}$ at $r$-th run}
    }
    $h \gets 1$\;
    $s^{(0)} \gets s_{initial}$\tcc*[r]{step size}
    \While{$h \leq max\_iter$ and $s^{(h)} > \phi$}{
        $F_1 \gets f({Q}^{(h-1)})$\;
        $s \gets s^{(h-1)}$\;
        \For{$k = 1$ \textbf{to} $2N$}{
            $s_k \gets (-1)^k s$\;
            $(i,j) \gets \zeta(\lceil \frac{k}{2}\rceil)$
            \tcc*[r]{$\lceil \cdot\rceil$ indicates ceiling function}
            $Q^{(h)}(i,j,s_k) \gets R_{i,j}(s_k){Q}^{(h-1)}$ \tcc*[r]{Apply Given's rotation}
        }
        $(i,j,s_k)_{best} \gets \arg\min \limits_{(i,j,s_k)} f(Q^{(h)}(i,j,s_k))$ {over} $k = 1,\ldots,2N$\;
        $\boldsymbol{Q}_{best}^{(h)} \gets Q^{(h)}(i,j,s_k)_{best}$\;
        ${Q}_{temp} \gets Q_{{best}}^{(h)}$\;
        $\boldsymbol{Q}^{(h)} \gets {Q}^{(h-1)}$\;
        $F_2 \gets f({Q}_{best}^{(h)})$\;
        \If{$F_2 < F_1$}{
            ${Q}^{(h)} \gets {Q}_{temp}$\;
        }
        \If {$h > 1$ \textbf{and} $\vert F_1-F_2\vert < \tau_1$ \textbf{and} $s>\phi$}{
            $s \gets s / \rho$\;
        }
        $s^{(h)} \gets s$\;
        $h \gets h + 1$\;
    }
    $\widehat{{Q}}^{(R)} \gets {Q}^{(h)}$\;
    \If{$\vert f(\widehat{{Q}}^{(R)}) - f(\widehat{{Q}}^{(R-1)})\vert < \tau_2$}{
        \Return{$\widehat{{Q}} = \widehat{{Q}}^{(R)}$}
    }
    \Else{
        $R \gets R + 1$\;
        Go to step 3\;
    }
\end{algorithm}

\begin{figure}[htbp]
	\centering
	\includegraphics[width=\textwidth]{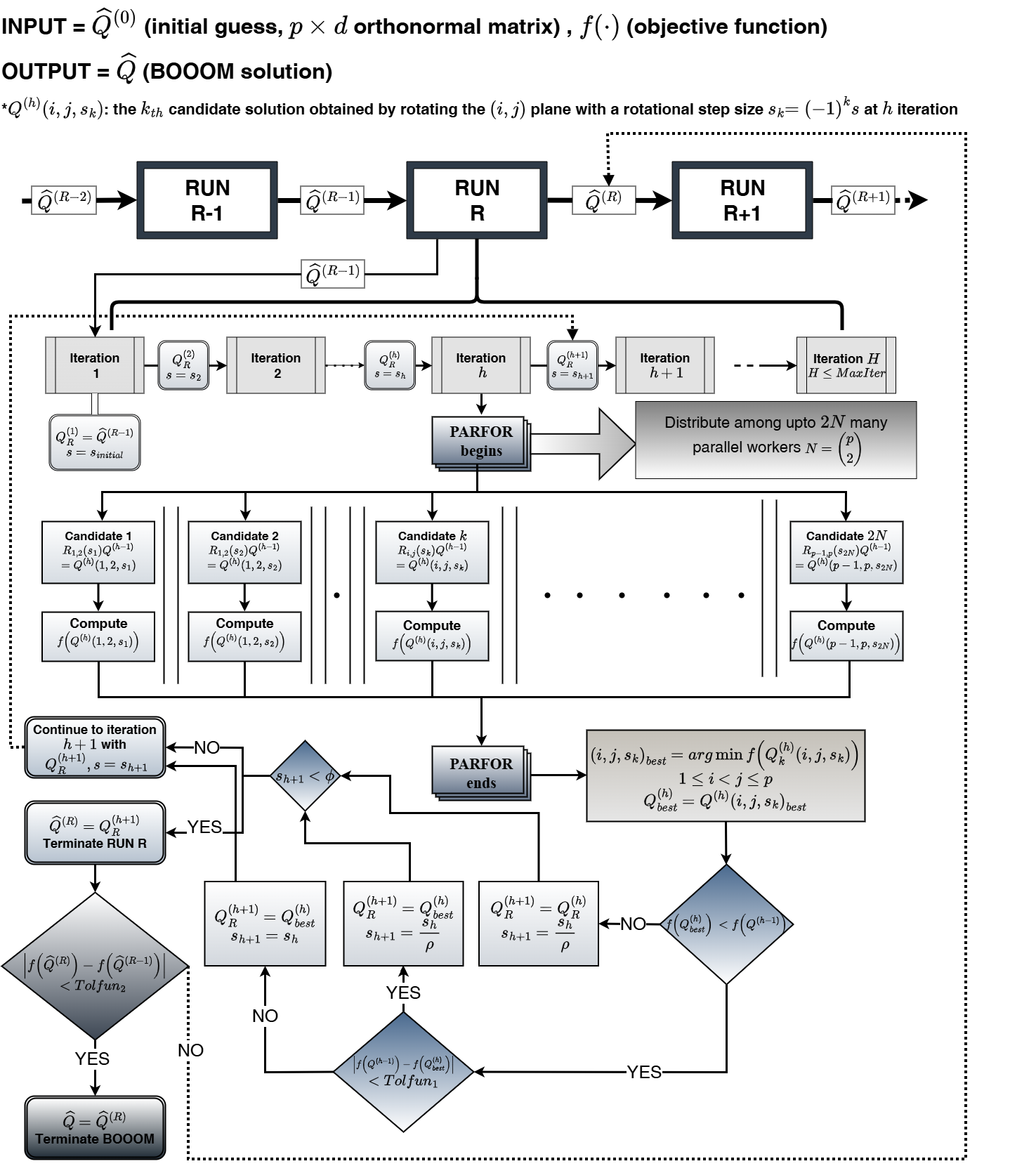}
	\caption{BOOOM flowchart}
	\label{BOOOM_concept}
\end{figure}

\section{Theoretical Properties}\label{sec:theory}

In this section, we establish the theoretical foundations of BOOOM by linking its algorithmic design to the geometry and optimization structure of the Stiefel manifold. Our analysis proceeds in two stages. First in Section~\ref{subsec:BOOOM_param}, we develop a theoretical understanding of the Givens rotation-based parametrization in the context of constrained optimization over $\mathrm{St}(p,d)$. Second in Section~\ref{subsec:theory-booom}, we analyze the convergence behavior of BOOOM under this parametrization. We show that the lack of descent in the angle space implies stationarity, which transfers to Riemannian stationarity on the manifold via a submersion argument. This connection enables us to establish global convergence in probability under mild conditions with full-support restarts, along with local optimality guarantees and a per-run complexity bound under convexity of the objective function. Together, these results establish a coherent link between the parametrization, the algorithmic updates, and the resulting convergence behavior, and provide a theoretical justification for BOOOM as a method for black-box optimization over orthogonality-constrained parameter spaces.

\subsection{Notations}\label{sec:notations}

The theoretical analysis of BOOOM relies on several notational constructs; we take some time to introduce them to aid the reader. As already indicated before in~\eqref{eqn:steifel-manifold}, let $\text{St}(p,d)$ be the Stiefel manifold, i.e., the set of $p \times d$ matrices with orthonormal columns. Let $\text{SO}(p)$ be the special orthogonal group of order $p$, consisting on only those matrices $Q$ from $\text{St}(p,p)$ such that their determinant is equal to $1$. Let $R_{i,j}(\theta)$ denote the Givens rotation for the $(i,j)$-th plane as in~\eqref{eqn:givens-rotation}. For a Riemannian manifold $\mathcal{M}$, we use $T_x\mathcal{M}$ to indicate the tangent space of $\mathcal{M}$ at $x \in \mathcal{M}$, and $\langle \cdot, \cdot \rangle$ to indicate the canonical Euclidean inner product on the tangent space. For the BOOOM algorithm, $Q^{(R,h)}$ and $\hat{Q}^{(R)}$ respectively denote the current solution at the $h$-th iteration of the $R$-th run and at the end of $R$-th run. When there is no ambiguity in the choice of the run $R$, we drop the superscript $R$ and denote $Q^{(h)}$ as the solution at $h$-th iteration. Similarly, $s^{(R, h)}$ and $s^{(h)}$ denote the step size at iteration $h$ for run $R$. For functions $f : A \mapsto B$, we use the notation $\mathcal{C}^k(A)$ to indicate the class of functions that are $k$-times differentiable in $A$, and we drop the set $A$ from parentheses when it is understood from the context. The readers are also referred to~\cite{Edelman1998} for different concepts related to Riemannian geometry and their implications.

\subsection{Parametrization of BOOOM}\label{subsec:BOOOM_param}

A central component of BOOOM is the ability to convert optimization over the constrained manifold $\mathrm{St}(p,d)$ into an unconstrained search over a Euclidean parameter space. This is achieved by expressing any orthonormal matrix as a product of planar Givens rotations applied to a fixed base point. Such a representation is particularly appealing for black-box optimization, as it preserves feasibility exactly while enabling coordinate-wise exploration in a continuous parameter space. We begin by establishing that any point on the Stiefel manifold can be generated through a finite sequence of Givens rotations. This extends the classical Hurwitz theorem (Lemma~1 of \citealp{jiang2022givens}, following \citealp{Hurwitz1963}) for orthogonal matrices to the rectangular Stiefel setting and provides the algebraic foundation for our parametrization.

\begin{lemma}[Extended Givens Decomposition on the Stiefel Manifold]\label{thm:stiefel-givens}
    Let $Q_0 \in \mathrm{St}(p,d)$ be any fixed column-orthonormal matrix. Then, for any $Q \in \mathrm{St}(p,d)$, there exists an angle vector $\theta = (\theta_{i,j}) \in \mathbb{R}^{\binom{p}{2}}$ such that
    \[
    Q = \left[ \prod_{(i,j)} R_{i,j}(\theta_{i,j}) \right] Q_0,
    \]
    where each $R_{i,j}(\theta_{i,j}) \in \mathbb{R}^{p \times p}$ is a Givens rotation acting on the $(i,j)$-plane, and the product is taken in any fixed order over $1 \le i < j \le p$. In other words, any $Q \in \mathrm{St}(p,d)$ can be obtained by applying a product of Givens rotations to any fixed base $Q_0 \in \mathrm{St}(p,d)$.
\end{lemma}

Lemma~\ref{thm:stiefel-givens} shows that the Stiefel manifold admits a constructive representation via planar rotations, which establishes a surjective map to transform any $Q_0 \in \mathrm{St}(p,d)$ to any other $Q \in \mathrm{St}(p,d)$. However, for optimization purposes, it is not sufficient. We require the mapping to be sufficiently smooth to allow us to reformulate constrained optimization problems in an unconstrained domain.

\begin{proposition}[Givens Parameterization of the Stiefel Manifold]\label{thm:givens_param}
    Let $p \geq d$ and fix any $Q_0\in\mathrm{St}(p,d)$. Then, there exists a smooth (i.e., $\mathcal{C}^\infty$), surjective mapping
    \begin{equation}
        \Phi:\Theta\to \mathrm{St}(p,d),\qquad 
        \Phi(\theta)=\Big[\prod_{1\le i<j\le p} R_{i,j}(\theta_{i,j})\Big]\,Q_0,
        \label{eqn:givens-phi}
    \end{equation}
    \noindent where $\Theta=\mathbb{R}^{\binom{p}{2}}$ (or $(-\pi,\pi]^{\binom{p}{2}}$) and $R_{i,j}(\cdot)$ is the Givens rotation as in~\eqref{eqn:givens-rotation}, and the product is taken in any fixed order over all pairs $1\le i<j\le p$. Consequently, for any continuous $f: \mathrm{St}(p,d) \to \mathbb{R}$,
    \begin{equation*}
        \min_{Q\in\mathrm{St}(p,d)} f(Q) \;=\; \min_{\theta\in\Theta} f(\Phi(\theta)).        
    \end{equation*}
\end{proposition}
\noindent Proposition~\ref{thm:givens_param} is fundamental for BOOOM: it shows that optimizing a function over the Stiefel manifold is equivalent to optimizing its pullback over the angle space. From an algorithmic perspective, this equivalence allows BOOOM to operate entirely in $\mathbb{R}^{\binom{p}{2}}$ using derivative-free search strategies, while implicitly maintaining feasibility on the manifold at every iteration.

\begin{remark}[Non-uniqueness and domain]
One should be careful about the fact that the map $\Phi$ is not injective, even if the angles are restricted to a fundamental domain such as $(-\pi,\pi]^{\binom{p}{2}}$ or $[0,2\pi)^{\binom{p}{2}}$. This can be attributed to several reasons: (i) Factorization redundancy: A given $U \in \mathrm{SO}(p)$ generally admits many ordered products of Givens rotations (beyond $2\pi$-periodicity), e.g., $R_{12}(\alpha)R_{13}(\pi/2)R_{23}(\beta) = R_{12}(\alpha+\delta)R_{13}(\pi/2)R_{23}(\beta - \delta)$ for any $\delta \in (-\pi, \pi]$. (ii) Stabilizer redundancy: Rotations acting only on the orthogonal complement of the span of the columns of $Q_0$ leave $Q_0$ unchanged, introducing additional non-uniqueness. This is also consistent with the fact that $\mathrm{St}(p,d)$ can be identified as the homogeneous space $\mathrm{SO}(p)/\mathrm{SO}(p-d)$.
\end{remark}

Fortunately, this non-uniqueness of the parametrization does not affect the optimization procedure. In fact, by introducing redundancy in the representation of the search space, it becomes beneficial in practice by providing multiple pathways to reach high-quality solutions. The following corollary of Proposition~\ref{thm:givens_param} connects the RMPS updates with this parametrization.

\begin{corollary}[RMPS iterates lie in the image of $\Phi$]
\label{cor:rmps-in-image}
    Fix $Q^{(0)} \in\mathrm{St}(p,d)$ and consider RMPS updates defined by pre-multiplication with planar Givens rotations, 
    \begin{equation*}
        Q^{(t+1)} = R_{i_t j_t}(\theta_t) Q^{(t)}, \ t = 1, 2, \dots, 
    \end{equation*}
    \noindent where eahc $R_{i_tj_t}(\theta_t) \in \mathbb{R}^{p \times p}$ is a Givens rotation acting on the $(i_t,j_t)$-th plane, as defined in~\eqref{eqn:givens-rotation}. Then, for every finite $t \geq 1$, there exists an angle vector $\vartheta_t \in\Theta\subset\mathbb{R}^{\binom{p}{2}}$ such that
    \begin{equation}
        Q^{(t)}=\Phi(\vartheta_t)=\Big[\prod_{1\le i<j\le p} R_{i,j}(\vartheta_{i,j})\Big]\,Q^{(0)},
        \label{eqn:cor-rmps-image}
    \end{equation}
    \noindent where the product is taken in any fixed canonical order over the $\binom{p}{2}$ pairs.
\end{corollary}

Note that, Corollary~\ref{cor:rmps-in-image} is trivially satisfied when $t \leq \binom{p}{2}$. However, the result ensures that the representation in~\eqref{eqn:cor-rmps-image} holds even when $T>\binom{p}{2}$ and planes $(i,j)$ are revisited multiple times. The key consequence of Corollary~\ref{cor:rmps-in-image} is that all iterates generated by RMPS (i.e., the key optimization strategy underlying BOOOM) mechanism, constructed through successive Givens rotations, remain within the image of the mapping $\Phi$. In other words, the BOOOM algorithm performs a structured search in the angle space, maintains the exact feasiblity at every iteration, while ensuring a coordinate-wise exploration. 

Let $N = \binom{p}{2}$. Equipped with an initial solution $Q_0 \in \mathrm{St}(p,d)$, the BOOOM parametrization of the Stiefel manifold is given by $\Phi: \mathbb{R}^N \to \mathrm{St}(p, d)$ as illustrated in~\eqref{eqn:givens-phi} of Proposition~\ref{thm:givens_param}. Now, consider a coordinate-wise wrapping map onto a fundamental box, namely,
\begin{equation*}
    \mathscr{M}:\mathbb{R}^{N}\to(-\pi,\pi]^N,
\end{equation*}
\noindent where each coordinate is reduced module $2\pi$. Under this, one can consider the wrapped parametrization 
\begin{equation}
    \widetilde{\Phi}=\Phi\circ \mathscr{M},\qquad 
    g(\varphi)=f(\widetilde{\Phi}(\varphi)).
    \label{eqn:wrapped-parametrization}
\end{equation}
\noindent Because each planar rotation is $2\pi$-periodic in its angle, optimizing $g=f\circ\widetilde{\Phi}$ over $\mathbb{R}^N$ is equivalent to optimizing over any fundamental box (or, equivalently, over the torus $\mathbb{T}^N=\mathbb{R}^N/(2\pi\mathbb{Z})^N$). In particular, the wrapped parametrization is \emph{exact} in the optimization sense:
\[
\min_{Q\in\mathrm{St}(p,d)} f(Q)
=
\min_{\theta\in\mathbb{R}^{N}} f(\Phi(\theta))
=
\min_{\varphi\in\mathbb{R}^{N}} g(\varphi).
\]
\noindent Additionally, the smoothness properties of $f$ and $g$ remain transferrable under this wrapped parametrization. The smoothness and differential properties of the wrapped parametrization, including the submersion property required to transfer stationarity from angle space to the manifold, are summarized in Appendix~\ref{subsec:auxiliary}.

From theoretical standpoint, this enables us to investigate the behaviour of the BOOOM algorithm by constraining the analysis in two ways: (i) Angles may be taken modulo $2\pi$, and hence $\Theta$ can be taken as the compact domain $\Theta_0 = [-\pi, \pi]^{N}$. (ii) If two rotations share no index, they commute (e.g., in $p=4$, $R_{12}(\alpha)R_{34}(\beta)=R_{34}(\beta)R_{12}(\alpha)$). As a result, the existence of $\vartheta$ in the fundamental box does not depend on the update order or the number of repeats of the rotation operation on the same plane.

\subsection{Convergence analysis of BOOOM}\label{subsec:theory-booom}

Building on the parametrization developed above, we now analyze the convergence properties of BOOOM under the wrapped angle representation. First, we show that repeated unsuccessful RMPS polls imply stationarity of the objective function, in turn, demonstrating a local convergence of BOOOM. Next, we show that the random restart mechanism in BOOOM can ensure that the iterations get arbitrarily close to the global optimum, and in combination of local convergence, this produces a global convergence in probability. 

Letting $N = \binom{p}{2}$, consider $\Pi: \{ 1, 2, \dots, N\} \to \{ (i, j): 1 \leq i < j \leq p\}$ be the indexing map used in $\Phi$, so that coordinate $m$ controls rotation in plane $\Pi(m)=(i,j)$. At an iterate $\nu \in \Theta_0$ and step size $s>0$, RMPS polls the $2N$ candidates $\nu \pm s e_m$ for $m=1,\ldots,N$, evaluating $g(\mathscr{M}(\nu\pm s e_m))$. In matrix terms, writing $Q=\widetilde{\Phi}(\nu)$, these correspond to updates of the form $Q\mapsto R_{ij}(\pm s)Q$ for $\Pi(m)=(i,j)$. An \emph{unsuccessful poll at scale $s$} means that none of these candidates decreases $g$. Equipped with these notions, we present our first result regarding the local convergence behavior of BOOOM algorithm.

\begin{theorem}[No descent implies stationarity]
\label{thm:booom-stationarity-angles}
    Let $g: \mathbb{R}^N \to\mathbb{R}$ be differentiable on $\mathbb{R}^N$. Fix $\nu \in \mathbb{R}^N$ and let $\delta_k = s / \rho^k$ with $s > 0, \rho > 1$. If, for every $k\in\mathbb{N}$ and each $i\in\{1,\ldots,N \}$, $g(\nu)\le g(\nu\pm \delta_k e_i)$ then $\nabla g(\nu)=\mathbf{0}$. Moreover, consider the wrapped parameterization as in~\eqref{eqn:wrapped-parametrization} with $g=f\circ\widetilde{\Phi}$. Assume $\nu$ is not on a wrapping boundary so that $D\mathscr{M}(\nu)=I$, and that $D\Phi(\mathscr{M}(\nu)):\mathbb{R}^N\to T_Q\mathrm{St}(p,d)$ has full rank (i.e., is surjective). Then, $\nabla g(\nu)=\bm{0}$ also implies vanishing Riemannian gradient of $f$ on $\mathrm{St}(p, d)$ at $Q = \widetilde{\Phi}(\nu)$, i.e., $\operatorname{grad} f(Q)=\bm{0}$. 
\end{theorem}

Theorem~\ref{thm:booom-stationarity-angles} establishes that persistent lack of descent in the RMPS polling directions forces stationarity of the pullback objective $g$ in the angle space. Under some additional regularity conditions, it ensures that when BOOOM stops within a run due to lack of descent, it also indicates that a local first-order optimality condition for $f$ is satisfied on the Stiefel manifold. 

\begin{remark}[Wrapping seams and regular points]
    Theorem~\ref{thm:booom-stationarity-angles} makes two key assumptions about the wrapping map to ensure stationarity of the objective function $f$. For the first assumption, note that the Lebesgue measure of the wrapping boundary is zero, and hence, with random restarts over the entire continuous domain, the probability of landing exactly on a seam is zero. Likewise, the rank deficiency of $D(\Phi \circ \mathscr{M})$ also occurs on a measure-zero set of specific angles governed by an algebraic identity; see~\cite{Edelman1998} for details. These facts justify the regularity assumptions used in Theorem~\ref{thm:booom-stationarity-angles}, and later on in Theorem~\ref{thm:booom-global-prob}.
\end{remark}

We next analyze the global exploration behavior of RMPS under restarts. The following result shows that, with full-support initialization, the algorithm eventually reaches a neighborhood of a global minimizer in angle space.

\begin{lemma}[Open-ball reachability of RMPS in angle space]\label{thm:booom-open-ball}
    Let $g$ be continuous on a compact, convex set $\Theta_0 \subset\mathbb{R}^N$ (here $\Theta_0 = [-\pi,\pi]^N$ under wrapping). Let $\nu^\star \in \argmin_{\nu\in \Theta_0} g(\nu)$ and fix $\delta>0$. Consider RMPS with polling along the canonical Euclidean basis vectors $\{\pm e_i\}_{i=1}^N$, step sizes $s^{(h)}\downarrow0$ via $s^{(h+1)}=s^{(h)}/\rho$ with $\rho > 1$, and between two successive runs, RMPS restarts from $\nu^{(r,0)}$ drawn i.i.d. from a distribution $\mu$ with full support on $\Theta_0$. Then, with probability $1$, there exists a run $r$ and iteration $h$ such that $\nu^{(r,h)} \in B_{\delta}(\nu^\star)$. 
\end{lemma}


The restart mechanism in Lemma~\ref{thm:booom-open-ball} serves primarily as a technical device to establish global reachability. By performing i.i.d. restarts with full support, BOOOM iterations can reach arbitrarily close to the true solution of the constrained optimization problem. However, the i.i.d. restarts are not representative of practical implementations. In particular, the result does not capture the exploratory behavior of the RMPS principle within a single run or under a finite restart budget. A more realistic perspective on practical performance and implementation strategies is discussed below in Remark~\ref{rem:finite-restarts}.

\begin{remark}[Finite restarts and practical exploration]\label{rem:finite-restarts}
Lemma~\ref{thm:booom-open-ball} guarantees eventual entry into any prescribed open ball under \emph{infinitely many} restarts drawn from a distribution with full support, and this result is later used to prove global convergence of BOOOM. In practice, however, computational budgets limit the number of restarts. Two clarifications are helpful. First, the proof in Lemma~\ref{thm:booom-open-ball} uses only one ingredient from the pattern-search machinery: once the step size is sufficiently small, the mesh $\nu^{(0)}+s\mathbb{Z}^N$ becomes fine enough to intersect any open ball in the compact domain. It does \emph{not} formalize the cumulative exploratory effect of the larger step sizes used earlier in a run. Second, despite this deliberately conservative analysis, empirical evidence shows that even a \emph{finite} number of restarts (e.g., 2–10), together with the inherently exploratory polling of pattern search, is often sufficient to reach high-quality solutions across benchmark functions on diverse domains (simplex, sphere, hyperrectangle), where RMPS has repeatedly outperformed state-of-the-art global optimizers, including the state-of-the-art algorithms such as GA and SA~\citep{kim2026smart, Das2023MsiCOR, Das2022, Das2023RMPSH, Das2021}. A practical refinement that consistently improves performance is to initialize each new run (from the second run onward) at the best solution found by the previous run, while resetting the step size to a large value. This preserves the ability to escape local minima (via large exploratory moves at the start of each run) yet exploits information accumulated so far. Extensive numerical experiments for RMPS indicate that this strategy converges in a few runs while maintaining strong global search capabilities. Motivated by these observations, BOOOM adopts this finite-restart, warm-start scheme. Developing sharper theory that captures the exploration power of pattern-search polling under finite restart budgets, without relying on asymptotically many restarts, remains an interesting direction for future work.
\end{remark}

Before turning to global guarantees, we recall a standard second-order sufficient condition (SOSC) on Riemannian manifolds: a positive-definite Riemannian Hessian at a critical point implies local minimality and local strong geodesic convexity (see, e.g.,
\citet[Sec.~5]{AbsilMahonySepulchre2008}, \citet[Ch.~5]{Lee2018} and \citet[Sec.~7]{Boumal2023}).

\begin{lemma}[SOSC implies local minimality and local strong convexity]
\label{thm:sosc-local}
    Assume $f$ is $\mathcal{C}^2$ in a neighborhood of $Q^\dagger\in\mathrm{St}(p,d)$, $\operatorname{grad} f(Q^\dagger)=\bm{0}$, and there exists $\mu>0$ such that the Riemannian Hessian satisfies
    \begin{equation*}
        \big\langle \operatorname{Hess} f(Q^\dagger)[\xi],\,\xi\big\rangle \;\ge\; \mu\,\|\xi\|^2 \qquad\text{for all } \xi\in T_{Q^\dagger}\mathrm{St}(p,d).       
    \end{equation*}
    \noindent Then there is a normal neighborhood $\mathcal V$ of $Q^\dagger$ such that
    \begin{enumerate}
        \item[(i)] $Q^\dagger$ is the unique minimizer of $f$ over $\mathcal V$; and;
        \item[(ii)] along every geodesic $\gamma$ contained in $\mathcal V$,
        \begin{equation*}
            f(\gamma(t)) \;\ge\; f(Q^\dagger) \;+\; \tfrac{\mu}{4}\,t^2,
        \end{equation*}
        \noindent for all sufficiently small $t > 0$, so that $f$ is (locally) strongly geodesically convex on $\mathcal{V}$.
    \end{enumerate}    
\end{lemma}


Combining the reachability result in Lemma~\ref{thm:booom-open-ball} with the local structure provided by Lemma~\ref{thm:sosc-local}, we obtain the following almost sure global convergence guarantee for BOOOM under full-support restarts.

\begin{theorem}[Global convergence of BOOOM]\label{thm:booom-global-prob}
    Assume that the objective function $f : \mathrm{St}(p,d)\to\mathbb{R}$ as in~\eqref{eq:stiefel_problem} satisfies the following property: (i) $f \in \mathcal{C}^0$, i.e., it is continuous; (ii) $f$ has at least one global minimizer $Q^\dagger$ and $f \in \mathcal{C}^2$ in a neighborhood of $Q^\dagger$, $\operatorname{grad} f(Q^\dagger)=\bm{0}$, and the SOSC in Lemma~\ref{thm:sosc-local} holds at $Q^\dagger$; (iii) there exists $\nu^\dagger \in \Theta_0 = [-\pi, \pi]^N$ with $\widetilde{\Phi}(\nu^\dagger)=Q^\dagger$ such that $\nu^\dagger$ is not on a wrapping boundary and $D\Phi(\mathscr{M}(\nu^\dagger))$ has full rank. Suppose, we run the BOOOM algorithm on $g = f \circ \tilde{\Phi}$ over a compact set $\Theta_0 \subset \mathbb{R}^N$ with step sizes $s^{(h)} \downarrow 0$ in each iteration by the geometric rate (as in Lemma~\ref{thm:booom-open-ball}) and i.i.d. full-support restarts on $\Theta_0$ according to a common probability measure $\mu$. Then, with probability $1$ under $\mu$, the geodesic distance 
    \begin{equation*}
        \operatorname{dist}\big(\widetilde{\Phi}(\nu^{(r, h)}),\,\operatorname{argmin}_{\mathrm{St}(p,d)} f\big) \to 0,
    \end{equation*}
    \noindent as either the number of runs $r \to \infty$ or the number of iterations within a run $h \to \infty$.
\end{theorem}

As a complement to the global-in-probability guarantee established above, we conclude with a per-run complexity bound for RMPS on the pulled-back objective. Let $k$ denote the number of step-size reductions within a run. When $g$ is convex and $L$-smooth on the compact angle hyperrectangle $\Theta_0 = [-\pi, \pi]^N$, the suboptimality at the iterates where the step size is reduced decays at $O(1/(k+1))$ rate. This provides a baseline efficiency guarantee for a single run and makes explicit how the reduction factor $\rho$ and the dimension $N$ enter the complexity of the algorithm.

\begin{theorem}[Sublinear complexity for convex function]\label{thm:booom-rate}
Let $g$ be convex and $\mathcal{C}^1$ on a compact convex set $\Theta_0 \subset \mathbb{R}^N$ (e.g., $\Theta_0 = [-\pi, \pi]^N$), with its gradient satisfying a Lipschitz condition with Lipschitz constant $L$. Consider a single RMPS run (without any restart) with initial step size $s_0$ and geometric reductions $s_k = s_0/\rho^k$ having $\rho > 1$, performed \emph{only} after an unsuccessful poll. Let $\nu_k$ denote the iterate at the (unsuccessful) poll that triggers the $k$-th reduction of the step size to $s_k$, and let $\nu^\star\in\argmin_{\Theta_0} g$. Then, for all $k\ge 0$,
\begin{equation*}
    g(\nu_k)-g(\nu^\star)\ \le\ \frac{LN s_0^2}{8(\rho^2 - 1)(k+1)}.
\end{equation*}
\end{theorem}

This result shows that BOOOM, when restricted to a single run and applied to a convex smooth pulled-back objective, achieves a sublinear $O(1/(k+1))$ decrease in optimality. This provides a baseline efficiency guarantee for the derivative-free polling mechanism and makes explicit how the smoothness constant, the reduction factor, and the dimension enter the bound. It is important, however, to interpret this rate with care: the index \(k\) counts only the unsuccessful polls that trigger a decrease in step size, not the total number of iterations or function evaluations. 

A crude bound on the total number of iterations can be obtained by noting that at step size $s$, BOOOM searches on the mesh $\nu_{\text{start}} + s \mathbb{Z}^N$ within the region $\Theta_0 = [-\pi, \pi]^N$. As a result, for a fixed step size $s$, it performs at most $(2\pi/s)^N$ steps. Therefore, if one wishes to obtain a solution $\hat{\nu}$ such that $g(\hat{\nu}) \in (g(\nu^\star), g(\nu^\star) + \delta)$, Theorem~\ref{thm:booom-rate} implies at most $O(1/\delta)$-many step size reductions, and hence the total number of iteration becomes bounded by the order of
\begin{equation*}
    \left(\frac{2\pi}{s_0}\right)^N \left( 1 + \rho^{-N} + \dots + \rho^{-N O(1/\delta)} \right) = O\left( \left(\frac{2\pi}{s_0}\right)^N \rho^{-N / \delta} \right).
\end{equation*}
\noindent However, for practical purposes, the number of iterations spent at a fixed step size can depend strongly on the objective function and on the local geometry of the search landscape 

Thus, Theorem~\ref{thm:booom-rate} should be viewed primarily as a per-run progress guarantee at the level of scale refinement, rather than as a full computational complexity bound in the usual first-order sense. Even with this qualification, the result is still informative. It shows that BOOOM retains a provable efficiency guarantee while operating as a zero-order method based only on function evaluations. Moreover, this per-run local descent guarantee complements the earlier restart-based reachability results: Theorems~\ref{thm:booom-stationarity-angles} and~\ref{thm:booom-rate} quantify efficiency and convergence within a single run, while Theorem~\ref{thm:booom-global-prob} address global exploration across runs. Taken together, these results highlight the dual role of BOOOM as a method that combines structured local descent with broad global search.

\section{Evaluation on Classical Benchmark Functions}\label{sec:class_bench}
To assess the performance of BOOOM, we adapt four classical benchmark functions: Ackley, Griewank, Rosenbrock, and Rastrigin, to the Stiefel manifold $\mathrm{St}(p,p)$, thereby embedding standard optimization landscapes into an orthogonality-constrained setting. We refer to \cite{surjanovic2013virtual} for the original benchmark definitions. For each orthogonal matrix $O \in \mathrm{St}(p,p)$, we form two vectors: the diagonal entries $x_D=\operatorname{diag}(O)$ and the vector $x_{OD}$ of all off-diagonal entries of $O$. The modified benchmark objective is defined as the sum of two standard benchmark functions, one applied to the diagonal part and the other to the off-diagonal part. For the Ackley, Griewank, and Rastrigin cases, we use the transformed inputs $10(x_D-\mathbf{1}_p)$ and $10x_{OD}$, so that the identity matrix corresponds to the global minimizer. For the Rosenbrock case, we apply the function to $x_D$ and $x_{OD}+\mathbf{1}_{p(p-1)}$, respectively, which again yields its minimum at the identity matrix. The scaling by $10$ in the Ackley, Griewank, and Rastrigin settings increases the oscillatory behavior of the resulting objective and produces more challenging multimodal landscapes on the Stiefel manifold. All experiments are implemented in MATLAB and executed on a high-performance computing environment equipped with AMD EPYC 9534 processors (20 cores) and 256 GB RAM. We compare BOOOM and its parallel version, BOOOM-Parallel, with a range of established optimization methods, including MATLAB's \texttt{fmincon} solvers (Active-set, Interior-point, and Sequential Quadratic Programming (SQP)), as well as Riemannian optimization algorithms implemented via the \texttt{Manopt} toolbox \citep{boumal2014manopt}. The latter include Riemannian Barzilai-Borwein (RBB), Riemannian Gradient Descent (RGD), Riemannian Conjugate Gradient (RCG), and Riemannian Trust-Region (RTR) methods.

\begin{sidewaystable}[p]
\centering
\resizebox{\columnwidth}{!}{%
\begin{tabular}{l|l|ccc|ccc|ccc|ccc}
\hline
\multicolumn{1}{c|}{\multirow{2}{*}{Functions}} & \multicolumn{1}{c|}{\multirow{2}{*}{Methods}} & \multicolumn{3}{c|}{$p=10$} & \multicolumn{3}{c|}{$p=20$} & \multicolumn{3}{c|}{$p=50$} & \multicolumn{3}{c}{$p=100$} \\ \cline{3-14} 
\multicolumn{1}{c|}{} & \multicolumn{1}{c|}{} & min value & s.e. of values & mean time (s.e.) & min value & s.e. of values & mean time (s.e.) & min value & s.e. of values & mean time (s.e.) & min value & s.e. of values & mean time (s.e.) \\ \hline
 & BOOOM & \textbf{1.60E+01} & 2.01E-01 & 10.35 (3.833) & \textbf{1.34E+01} & 1.60E-01 & 2295.42 (209.333) & \textbf{1.74E+01} & 1.38E-01 & 18000.18 (0.044) & \textbf{1.91E+01} & 3.23E-02 & 18001.30 (0.314) \\
 & BOOOM-Parallel & \textbf{1.60E+01} & 2.01E-01 & 792.88 (322.553) & \textbf{1.34E+01} & 1.61E-01 & 3448.26 (151.948) & \textbf{1.74E+01} & 1.34E-01 & 18000.94 (0.260) & \textbf{1.90E+01} & 5.87E-02 & 18002.03 (0.464) \\
 & Active-Set (fmincon) & 2.75E+01 & 9.04E-02 & 29.42 (1.756) & 2.53E+01 & 4.85E-02 & 137.81 (6.695) & — & — & — & — & — & — \\
 & Interior-Point (fmincon) & 2.76E+01 & 1.07E-01 & 15.89 (0.745) & 2.63E+01 & 5.02E-02 & 17.76 (0.866) & — & — & — & — & — & — \\
\multicolumn{1}{c|}{Ackley} & SQP (fmincon) & 2.76E+01 & 8.59E-02 & 23.05 (1.513) & 2.53E+01 & 4.50E-02 & 112.25 (3.656) & — & — & —  & — & — & — \\
 & Riemannian Barzilai–Borwein (RBB) & \textbf{1.88E+01} & 9.22E-01 & 643.88 (43.018) & \textbf{1.58E+01} & 9.74E-01 & 2974.16 (84.085) & \textbf{4.63E+00} & 2.34E+00 & 18007.36 (1.661) & 1.95E+01 & 3.15E-01 & 18133.38 (31.741) \\
 & Riemannian Conjugate Gradient (RCG) & 2.76E+01 & 8.20E-02 & 121.36 (8.564) & 2.53E+01 & 4.47E-02 & 1111.36 (127.927) & 2.30E+01 & 2.50E-02 & 16870.36 (792.598) & 2.19E+01 & 1.88E-02 & 18312.86 (48.711) \\
 & Riemannian Gradient Descent (RGD) & 2.73E+01 & 1.22E-01 & 475.20 (63.466) & 2.53E+01 & 6.27E-02 & 2996.57 (21.068) & 2.30E+01 & 3.35E-02 & 18013.79 (1.486) & 2.24E+01 & 2.54E-02 & 18103.39 (19.820) \\
 & Riemannian Trust-Region (RTR) & 2.77E+01 & 7.68E-02 & 1533.21 (16.257) & 2.53E+01 & 4.73E-02 & 3603.85 (0.472) & 2.30E+01 & 2.49E-02 & 18347.79 (97.217) & 2.21E+01 & 4.54E-02 & 19010.00 (219.932) \\ \hline
 & BOOOM & 3.76E-01 & 5.09E-02 & 23.33 (5.119) & 2.78E-01 & 1.47E-02 & 2661.47 (72.994) & 1.38E-01 & 1.03E-02 & 18000.18 (0.029) & \textbf{1.83E-01} & 1.28E-02 & 18001.17   (0.229) \\
 & BOOOM-Parallel & 3.76E-01 & 5.09E-02 & 2376.10 (475.550) & 2.78E-01 & 1.46E-02 & 3542.53 (57.682) & 1.38E-01 & 1.03E-02 & 3837.64 (36.230) & \textbf{1.18E-01} & 8.89E-03 & 12843.28 (163.015) \\
 & Active-Set (fmincon) & \textbf{2.86E-01} & 1.28E-01 & 21.97 (1.976) & \textbf{1.31E-01} & 4.34E-02 & 119.67 (4.359) & — & — & — & — & — & —  \\
 & Interior-Point (fmincon) & 5.27E-01 & 1.63E-01 & 14.34 (0.725) & 1.97E+00 & 8.17E-02 & 18.10 (0.782) & — & — & — & — & — & — \\
\multicolumn{1}{c|}{Griewank} & SQP (fmincon) & 3.98E-01 & 1.06E-01 & 22.16 (1.139) & 1.36E-01 & 1.32E-01 & 130.44 (8.369) & — & — & —  & — & — & — \\
 & Riemannian Barzilai–Borwein (RBB) & \textbf{3.34E-01} & 8.50E-02 & 635.74 (33.518) & 1.94E-01 & 7.13E-02 & 2631.25 (207.175) & \textbf{1.17E-01} & 9.03E-02 & 18013.50 (2.409) & 1.85E+00 & 2.28E-01 & 18134.86 (22.344) \\
 & Riemannian Conjugate Gradient (RCG) & 4.63E-01 & 6.18E-02 & 412.46 (65.008) & \textbf{1.31E-01} & 3.32E-02 & 2947.76 (72.807) & 1.23E-01 & 6.85E-03 & 18008.85 (2.048) & 4.96E-01 & 8.62E-02 & 18109.59 (23.740) \\
 & Riemannian Gradient Descent (RGD) & 4.63E-01 & 1.01E-01 & 670.17 (6.837) & \textbf{1.33E-01} & 9.58E-02 & 3015.87 (24.686) & 2.43E-01 & 4.32E-02 & 18009.40 (2.157) & 5.82E+00 & 2.89E-02 & 18196.16 (53.924) \\
 & Riemannian Trust-Region (RTR) & 4.87E-01 & 5.14E-02 & 1016.76 (240.076) & \textbf{1.31E-01} & 3.42E-02 & 3119.28 (232.775) & \textbf{1.05E-01} & 8.17E-03 & 18709.43 (126.843) & 2.72E+00 & 1.08E-01 & 18715.37 (248.280) \\ \hline
 & BOOOM & 1.57E+02 & 6.62E+01 & 8.91 (1.872) & \textbf{5.55E+02} & 1.26E+02 & 475.65 (52.463) & \textbf{1.88E+03} & 1.45E+02 & 7600.97 (23.482) & \textbf{5.78E+03} & 3.12E+02 & 18001.15 (0.227) \\
 & BOOOM-Parallel & 1.57E+02 & 6.62E+01 & 1350.89 (433.402) & \textbf{5.55E+02} & 1.26E+02 & 3541.81 (58.475) & \textbf{1.88E+03} & 1.45E+02 & 3650.86 (8.098) & \textbf{5.09E+03} & 2.51E+02 & 11401.89 (269.719) \\
 & Active-Set (fmincon) & 1.57E+02 & 1.21E+02 & 27.63 (1.444) & 3.50E+03 & 1.24E+02 & 106.08 (4.899) & — & — & — & — & — & — \\
 & Interior-Point (fmincon) & 1.59E+02 & 1.11E+02 & 12.46 (0.619) & 3.30E+03 & 1.39E+02 & 14.38 (0.696) & — & — & — & — & — & — \\
\multicolumn{1}{c|}{Rosenbrock} & SQP (fmincon) & \textbf{1.12E+02} & 1.12E+02 & 18.97 (0.990) & 4.59E+03 & 1.23E+02 & 82.67 (1.897) & — & — & — & — & — & — \\
 & Riemannian Barzilai–Borwein (RBB) & \textbf{4.00E+00} & 6.10E+01 & 622.87 (31.153) & \textbf{2.72E+02} & 3.38E+01 & 3072.75 (27.941) & \textbf{2.72E+02} & 5.78E+01 & 18010.57 (1.739) & 3.99E+04 & 1.21E+02 & 18218.73 (30.518) \\
 & Riemannian Conjugate Gradient (RCG) & 1.57E+02 & 1.04E+02 & 65.49 (4.727) & 1.38E+03 & 8.20E+01 & 352.17 (48.601) & 4.20E+03 & 1.47E+02 & 2417.64 (109.888) & 8.71E+03 & 2.03E+02 & 18127.00 (22.580) \\
 & Riemannian Gradient Descent (RGD) & 1.57E+02 & 9.89E+01 & 174.94 (18.280) & 1.38E+03 & 7.41E+01 & 1312.91 (291.009) & 4.23E+03 & 1.26E+02 & 6455.01 (813.411) & 9.24E+03 & 1.91E+02 & 18139.75 (28.323) \\
 & Riemannian Trust-Region (RTR) & 1.57E+02 & 9.97E+01 & 2922.01 (269.490) & 1.38E+03 & 7.62E+01 & 3602.58 (0.669) & 4.14E+03 & 1.42E+02 & 18018.85 (3.741) & 9.21E+03 & 1.77E+02 & 18593.46 (146.960) \\ \hline
 & BOOOM & \textbf{8.18E+02} & 1.93E+01 & 3.13 (0.245) & \textbf{2.48E+03} & 3.11E+01 & 1529.40 (261.354) & \textbf{1.05E+04} & 4.80E+01 & 18000.19 (0.029) & 4.13E+04 & 2.35E+02 & 18001.72 (0.342) \\
 & BOOOM-Parallel & \textbf{8.18E+02} & 1.93E+01 & 162.24 (18.226) & \textbf{2.48E+03} & 3.13E+01 & 3259.25 (283.010) & \textbf{1.05E+04} & 4.80E+01 & 6202.70 (140.371) & \textbf{3.11E+04} & 9.53E+01 & 18000.10 (0.020) \\
 & Active-Set (fmincon) & 2.19E+03 & 6.14E+01 & 31.35 (1.612) & 5.29E+03 & 6.28E+01 & 102.59 (5.366) & — & — & — & — & — & — \\
 & Interior-Point (fmincon) & 1.51E+03 & 1.16E+02 & 14.11 (0.713) & 5.33E+03 & 6.38E+01 & 13.59 (0.751) & — & — & — & — & — & — \\
\multicolumn{1}{c|}{Rastrigin} & SQP (fmincon) & 1.45E+03 & 1.03E+02 & 22.47 (1.142) & 5.40E+03 & 6.95E+01 & 87.97 (2.391) & — & — & — & — & — & —  \\
 & Riemannian Barzilai–Borwein (RBB) & \textbf{8.79E+02} & 5.93E+01 & 687.90 (3.321) & \textbf{2.16E+03} & 9.11E+01 & 2674.90 (35.362) & \textbf{1.05E+04} & 1.94E+02 & 17939.73 (36.979) & 1.03E+05 & 2.03E+03 & 18135.56 (23.731) \\
 & Riemannian Conjugate Gradient (RCG) & 2.22E+03 & 5.78E+01 & 45.42 (2.185) & 4.73E+03 & 5.94E+01 & 394.49 (27.104) & 1.45E+04 & 6.25E+01 & 5222.35 (327.458) & \textbf{3.87E+04} & 1.17E+02 & 18208.91 (43.448) \\
 & Riemannian Gradient Descent (RGD) & 2.02E+03 & 5.98E+01 & 103.34 (20.990) & 4.85E+03 & 5.21E+01 & 712.37 (156.209) & 1.46E+04 & 5.49E+01 & 9400.78 (923.116) & 3.98E+04 & 9.11E+01 & 18311.54 (35.113) \\
 & Riemannian Trust-Region (RTR) & 2.17E+03 & 5.82E+01 & 1362.76 (6.899) & 4.77E+03 & 5.52E+01 & 3602.85 (0.660) & 1.45E+04 & 5.88E+01 & 18018.23 (3.778) & 3.99E+04 & 4.58E+01 & 18259.58 (66.729) \\ \hline
\end{tabular}%
}
\caption{Performance comparison on the modified Ackley, Griewank, Rosenbrock, and Rastrigin functions over the Stiefel manifold $\mathrm{St}(p,p)$ for $p \in \{10,20,50,100\}$. For each method and dimension, we report the best objective value over 10 independent runs, the standard error of the final objective values, and the mean runtime in seconds with standard error in parentheses. BOOOM and BOOOM-Parallel are compared with MATLAB \texttt{fmincon} solvers (Active-set, Interior-point, and SQP) and Riemannian optimization methods from \texttt{Manopt} (RBB, RCG, RGD, and RTR). All methods use the same random orthogonal initializations within each setting. The \texttt{fmincon} solvers are omitted for $p=50,100$ because they often stalled for extended periods in higher dimensions. The two smallest objective values in each setting are highlighted in bold.}
\label{tab:benchmark_study}
\end{sidewaystable}

\begin{figure}[!htbp]
	\centering
	\includegraphics[width=0.95\textwidth]{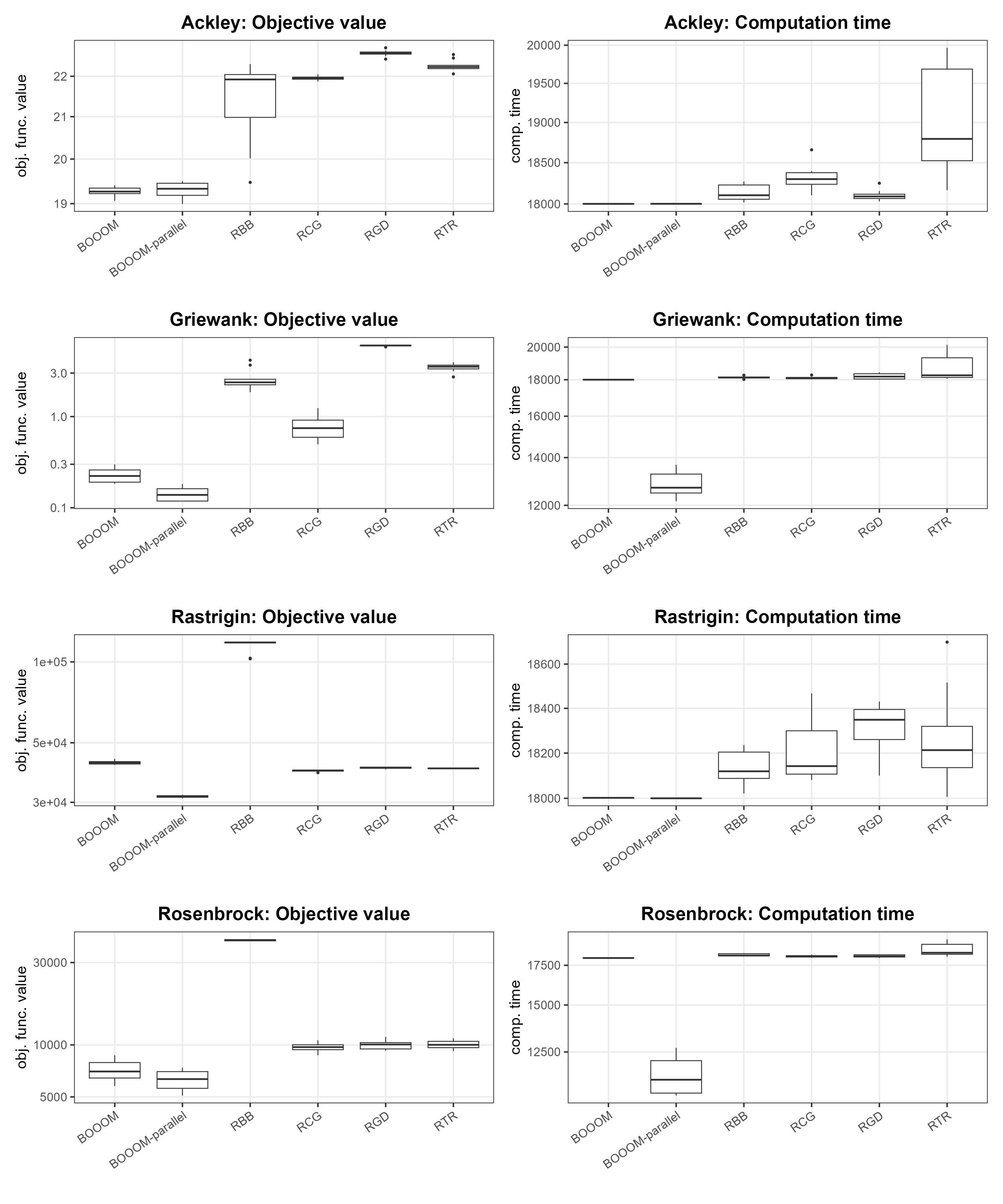}
	\caption{Comparison of BOOOM, BOOOM-parallel, and four Riemannian optimization baselines: Riemannian Barzilai–Borwein (RBB), Riemannian Conjugate Gradient (RCG), Riemannian Gradient Descent (RGD). and Riemannian Trust Region (RTR) on the modified Ackley, Griewank, Rastrigin, and Rosenbrock functions in dimension $p=100$. Left panels show objective values; right panels show corresponding computation times across 10 runs. Each run was terminated after a maximum of 5 hours.}
	\label{fig:Benchmark_boxplots}
\end{figure}

Table~\ref{tab:benchmark_study} reports performance across four problem sizes $p \in \{10, 20, 50, 100\}$. For each configuration, all methods are initialized from the same set of randomly generated orthogonal matrices and repeated over 10 independent runs. A maximum runtime of one hour for $p=10,20$ and five hours for $p=50,100$ is imposed for BOOOM and the Riemannian optimization routines. Such time limits cannot be enforced for the \texttt{fmincon} solvers; moreover, in higher-dimensional settings ($p=50,100$), these methods frequently stalled for extended periods and are therefore excluded from those scenarios. Table~\ref{tab:benchmark_study} indicates that BOOOM is consistently among the strongest performers across all four benchmark families and all problem dimensions. For Ackley, BOOOM and BOOOM-Parallel achieve the best objective values throughout, outperforming the \texttt{fmincon} solvers in the small-dimensional cases and generally improving upon the Riemannian baselines in the larger problems. For Griewank, several baselines remain competitive at $p=10$ and $p=20$, but BOOOM-Parallel gives the best result at $p=100$ while requiring substantially less time than methods that exhaust the five-hour budget.
The advantage of BOOOM is more evident on the more challenging Rosenbrock and Rastrigin functions. On Rosenbrock, BOOOM and BOOOM-Parallel attain the best values from $p=20$ onward, with BOOOM-Parallel giving the strongest result at $p=100$. On Rastrigin, BOOOM and BOOOM-Parallel dominate at $p=10,20,50$, and BOOOM-Parallel again attains the best value at $p=100$. Overall, these results show that BOOOM is robust across diverse non-convex landscapes, while BOOOM-Parallel is particularly effective in higher dimensions, where it frequently achieves better solutions under the same runtime budget or comparable solutions in markedly less time. Among the competing baselines, the Riemannian methods, specifically Riemannian Barzilai–Borwein (RBB), remains the most competitive alternative.

Figure~\ref{fig:Benchmark_boxplots} further illustrates the comparative behavior of the methods at $p=100$ in terms of both solution quality and computational cost. BOOOM and BOOOM-Parallel consistently achieve lower objective values with substantially reduced variability across runs relative to the Riemannian baselines, particularly for the Ackley, Rosenbrock, and Rastrigin functions. Among the competing methods, RBB and RCG remain the strongest alternatives, but exhibit noticeably higher variability and are more prone to suboptimal local minima, especially on the highly multimodal Rastrigin landscape. In terms of computational efficiency, BOOOM-Parallel demonstrates a clear advantage in several settings, notably for the Griewank and Rosenbrock functions, where it attains competitive or superior objective values in substantially less time than methods that typically run until the imposed time limit. In contrast, most Riemannian methods, particularly RTR, frequently exhaust the full computational budget without corresponding gains in solution quality. Overall, the boxplots reinforce that BOOOM provides stable and high-quality solutions, while its parallel implementation is particularly beneficial in high-dimensional regimes.

\section{Benchmark Experiments}\label{sec:bench_exps}
We evaluate BOOOM through a comprehensive suite of benchmark experiments designed to assess its performance across a wide spectrum of orthogonality-constrained optimization problems. Rather than focusing on a single application domain, we consider six representative problem classes spanning signal processing, matrix decomposition, blind source separation, statistical learning, and electronic structure computation. Table~\ref{tab:simulation_overview} summarizes the benchmark settings, highlighting their objective functions, sources of non-convexity, and competing baseline methods. These problems exhibit diverse structural challenges, including multi-modal landscapes, rotational invariances, and coupled low-rank and sparse structures, thereby providing a rigorous testbed for evaluating BOOOM as a general-purpose black-box optimizer on the Stiefel manifold.

\begin{table}[h]
\centering
\footnotesize
\begin{tabular}{p{3.0cm} p{4.2cm} p{3.9cm} p{2.5cm}}
\toprule
Problem & Objective & Nonconvex structure & Baselines \\
\midrule

Heterogeneous quadratic maximization (Section \ref{subsec:HQF}) &
$\sum_{i=1}^{d} q_i^\top M_i q_i$ &
Heterogeneous quadratic terms producing multiple local optima &
SDP relaxation \\

\addlinespace[4pt]

Low-rank + sparse decomposition (Section \ref{subsec:LRSD}) &
$\|XQQ^\top\|_* + \lambda\|X(I-QQ^\top)\|_1$ &
Structured nonconvexity arising from simultaneous low-rank and sparse components &
AccAltProj, GoDec+, LRSD-TNNSR \\

\addlinespace[4pt]

Independent component analysis (Section \ref{subsec:ICA}) &
Log-cosh contrast maximization under orthogonality constraints &
Non-Gaussian likelihood surface with scale and sign ambiguity &
FastICA, Infomax, Picard \\

\addlinespace[4pt]

Varimax factor rotation (Section \ref{subsec:Varimax}) &
Fourth-moment dispersion criterion applied to rotated loadings &
Multiple rotation-equivalent local optima due to orthogonal rotations &
MATLAB \texttt{rotatefactors} \\

\addlinespace[4pt]

Orthogonal joint diagonalization (Section \ref{subsec:OJD}) &
$\sum_k \|\mathrm{offdiag}(W^\top C_k W)\|_F^2$ &
Approximate simultaneous diagonalization of multiple matrices &
Jacobi AJD, Riemannian GD, Riemannian TR \\

\addlinespace[4pt]

Reduced Kohn--Sham Rayleigh--Ritz (Section \ref{subsec:KS}) &
$\mathrm{tr}(Q^\top H_{\mathrm{red}}Q)$ &
Eigenvalue optimization with orthogonality constraints &
Riemannian CG \\

\bottomrule
\end{tabular}
\caption{Overview of simulation benchmarks used to evaluate BOOOM.
All problems are formulated as optimization tasks on the Stiefel manifold.
Objective functions and baseline methods are summarized here using acronyms 
for brevity; full descriptions are provided in the corresponding subsections.}
\label{tab:simulation_overview}
\end{table}

\subsection{Maximization of Sum of Heterogeneous Quadratic Forms}\label{subsec:HQF}
The maximization of sums of heterogeneous quadratic forms under the Stiefel manifold arises in signal processing, which deal with multi-dataset subspace estimation problems. Heteroscedastic probabilistic PCA technique (HePPCAT) incorporates heteroscedastic noise and likelihood-based estimation to address low-rank approximation \citep{hong2021heppcat}. However, the maximization of heterogeneous quadratic forms is non-convex and non-trivial, and traditional gradient-based optimization methods do not guarantee global optimality.

Recently, \citet{gilman2025semidefinite} developed a semidefinite programming (SDP) relaxation to solve the original problem of maximizing sums of heterogeneous quadratic forms under Stiefel manifold constraints. The SDP relaxation convexifies the original non-convex objective and provides an efficiently computable upper bound on the true optimum. It enables global analysis through convex optimization when each optimal $X_i=q_iq_i^ \top$ is rank-one where $q_i$ is $i_{th}$ orthornormal basis vector in equation \eqref{eqn:Heterogeneous_Quadratic}. We evaluate the performance of BOOOM and SDP method on the maximization problem
\begin{align}
\max_{Q\in \text{St}(p,d)}\sum_{i=1}^{d} \textbf{q}_i^\top\textbf{M}_i\textbf{q}_i 
\label{eqn:Heterogeneous_Quadratic}
\end{align} 
where $\textbf{M}_1,\ldots,\textbf{M}_d \succeq 0 \text{ for } d<p,$ $\textbf{Q}=[\textbf{q}_1 \cdots \textbf{q}_d] \in \ \mathbb{R} ^{p\times d}$, $\text{St}(p,d) = \{\textbf{Q} \in \mathbb{R} ^{p\times d} : \textbf{Q}^\top\textbf{Q}= \textbf{I}_d  \}$.

To assess performance under different structure conditions on the matrices $M_i$,  we generate sets of $d$ positive semidefinite matrices following three patterns:
\begin{enumerate}[nosep]
\item Random pattern: Each $M_i$ is generated as $A_i^\top A_i$ where $A_i$ has iid Gaussian entries. This represents unstructured and indepently diagonalizable matrices. 
\item Toeplitz pattern: Each $M_i$ has entries $m_{jk} = \rho_i^{\left| j-k\right|}$ where $\rho_i\in(0,1)$. This pattern produces matrices that are close-to-jointly diagonalizable.
\item Block diagonal pattern: Each $M_i$ consists of 5 diagonal blocks with random within-block correlation matrices. This pattern represents matrices that are structurally jointly block diagonalizable.
\end{enumerate}
\begin{figure}[!h]
	\centering
	\includegraphics[width=\textwidth]{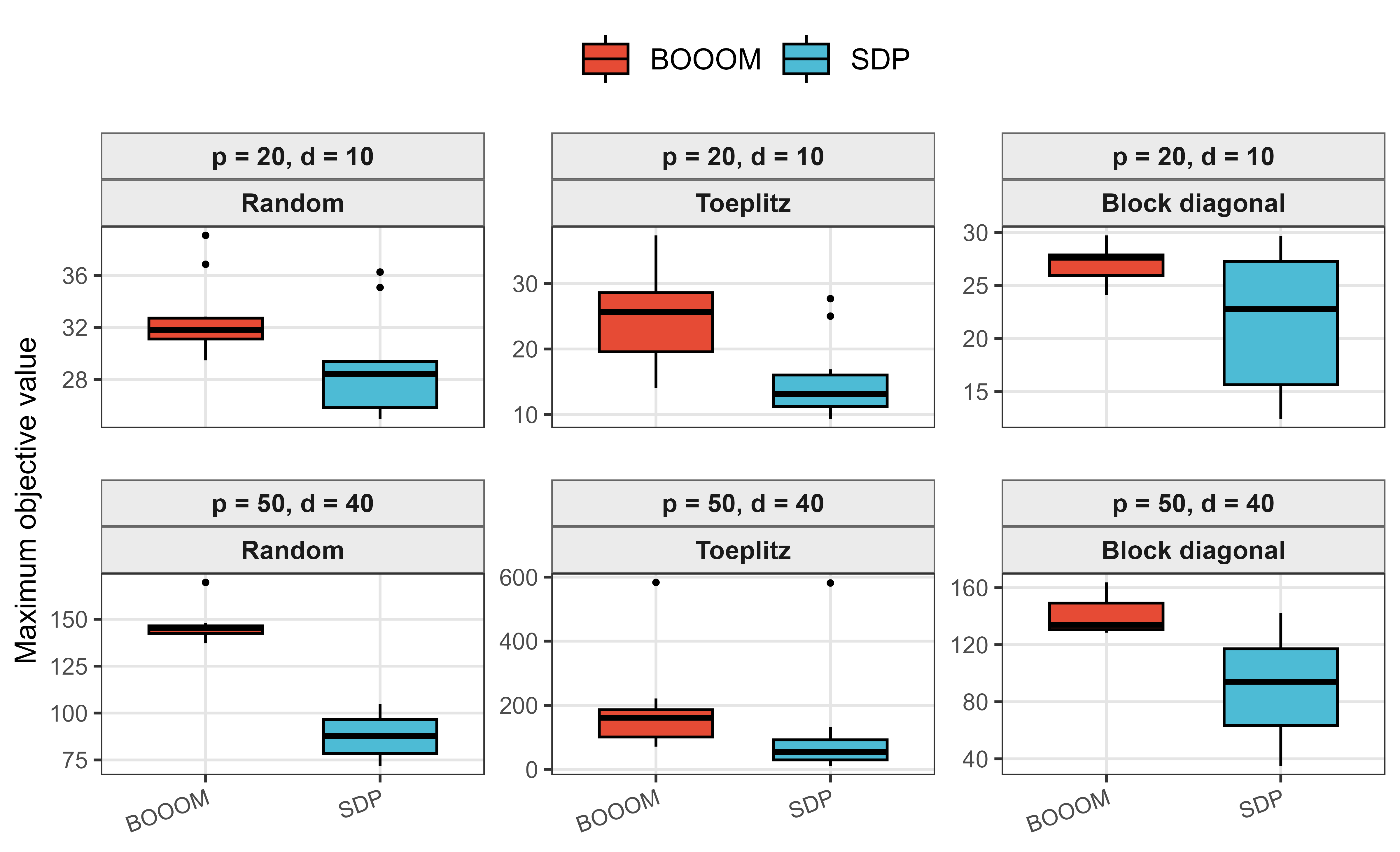}
	\caption{Maximization of heterogeneous quadratic forms: Maximum value comparison between BOOOM and SDP for Random, Toeplitz, Block diagonal pattern on $M_i$. Upper panel: size of $M_i \text{ is } (20 \times 20)$ and size of $Q \text{ is } (20 \times 10)$. Lower panel: size of $M_i \text{ is }(50 \times 50)$ and size of $Q \text{ is } (50 \times 40)$. Each boxplot summarizes $10$ Monte Carlo replicates.}
	\label{fig:HQFmaximization}
\end{figure}
We consider two problems of dimension of $Q_{p\times d}$ with $(p, d) = (20,10), \; (50,40)$. For each setting, we conduct 10 experiments and compute the maximum value of the objective function. Figure~\ref{fig:HQFmaximization} compares the maximum objective values achieved by BOOOM and the SDP-based approach under three structural patterns for $M_i$ (Random, Toeplitz, and Block diagonal) and two problem sizes (top: $(p, d) = (20, 10)$; bottom: $(p, d) = (50, 40)$). Across all matrix structures and dimensions, BOOOM consistently attains higher objective values than the SDP-based method. In the lower-dimensional setting (top panel), BOOOM achieves higher median objective values across all three patterns. In the higher-dimensional setting (bottom panel), the advantage of BOOOM becomes more pronounced, particularly under the Random and Block diagonal structures, where the gap in central tendency between BOOOM and SDP is substantially larger than in the smaller-dimensional case. Overall, BOOOM not only improves the attainable maximum objective values but also demonstrates strong robustness, especially under the Random and Block diagonal patterns. Although variability is higher under the Toeplitz structure, BOOOM maintains superior overall performance across replicates. The numerical summaries corresponding to these boxplots are provided in Appendix Table~\ref{tab:HQFmaximization}.

\subsection{Low rank and Sparse Matrix Decomposition}\label{subsec:LRSD}
Decomposing a matrix into low-rank and sparse components aims to separate dominant low-dimensional structure from sparse, large-magnitude deviations. This paradigm has been widely used in applications such as video surveillance, face recognition, and latent semantic indexing \citep{candes2011robust}. A classical formulation is Robust Principal Component Analysis (RPCA), given by
\begin{equation}
    \min \ \mathrm{rank}(L) + \lambda \|S\|_0, \quad \text{such that } X = L + S,
    \label{eqn:robust-pca-main}
\end{equation}
which is non-convex and NP-hard \citep{hu2020moving}. A widely used convex relaxation is Principal Component Pursuit (PCP),
\begin{equation}
    \min \ \|L\|_* + \lambda \|S\|_1, \quad \text{such that } X = L + S,
    \label{eqn:robust-pca-approx}
\end{equation}
where the nuclear norm and $\ell_1$-norm serve as convex surrogates for rank and sparsity, respectively \citep{candes2011robust}. Numerous algorithms have been developed to efficiently solve these formulations. For example, Accelerated Alternating Projections (AccAltProj) \citep{cai2019accelerated} directly targets the non-convex RPCA problem using alternating projections with provable recovery guarantees. GoDec+ \citep{guo2017godec+} combines randomized low-rank approximation with a maximum correntropy criterion to enhance robustness under heavy-tailed noise. Low-rank and sparse matrix decomposition via the truncated nuclear norm and a sparse regularizer (a.k.a. LRSD-RNNSR; \citealp[]{xue2019low}) employs truncated nuclear norm regularization alongside sparse penalties to better approximate rank while promoting sparsity.

To connect low-rank structure with orthogonality constraints, we consider a factorized representation of the low-rank component via an orthonormal basis $Q \in \mathrm{St}(p,d)$, where $QQ^\top$ acts as a projection onto a $d$-dimensional subspace. Under this parameterization, the low-rank component can be expressed as $XQQ^\top$, while the residual $X(I - QQ^\top)$ captures sparse deviations. This leads to the following optimization problem:
\begin{equation}
    \min_{Q \in \mathrm{St}(p,d)} \ \|XQQ^\top\|_* + \lambda \|X(I - QQ^\top)\|_1,
    \label{eqn:robust-pca-obj}
\end{equation}
where $X \in \mathbb{R}^{n \times p}$ is the data matrix, $\|\cdot\|_*$ denotes the nuclear norm, and $\|\cdot\|_1$ is the element-wise $\ell_1$-norm.

We compare the matrix decomposition performance of BOOOM applied to the reparameterized PCP objective \eqref{eqn:robust-pca-obj} against AccAltProj, GoDec+, and LRSD-TNNSR. For each method, performance is evaluated using the mean absolute error (MAE) between the true low-rank matrix $L$ and its estimate $\widehat{L}$ across varying data dimensions. The synthetic data matrix $X$ is generated as follows:

\noindent
\begin{enumerate}[nosep]
\item Generate $A \in \mathbb{R}^{n \times p}$ with $A_{ij} \overset{\text{iid}}{\sim} N(0,1)$ and compute its singular value decomposition $A = U_A D_A V_A^\top$.
\item Construct a rank-$d$ low-rank matrix $L = U_A[:,1\!:\!d] \, D_L \, V_A[:,1\!:\!d]^\top$, where 
\[
D_L = \mathrm{diag}\!\left(2, \, 1 + \frac{d-2}{d-1}, \, \dots, \, 1 + \frac{1}{d-1}\right).
\]
\item Generate a sparse matrix $S$ with entries $S_{ij} = M_{ij} B_{ij}$, where $M_{ij} \sim \mathrm{Cauchy}(0,1)$ and $B_{ij} \sim \mathrm{Bernoulli}(0.2)$.
\item Set $X = L + S$.
\end{enumerate}

\begin{figure}[!h]
	\centering
    \includegraphics[width=\textwidth]{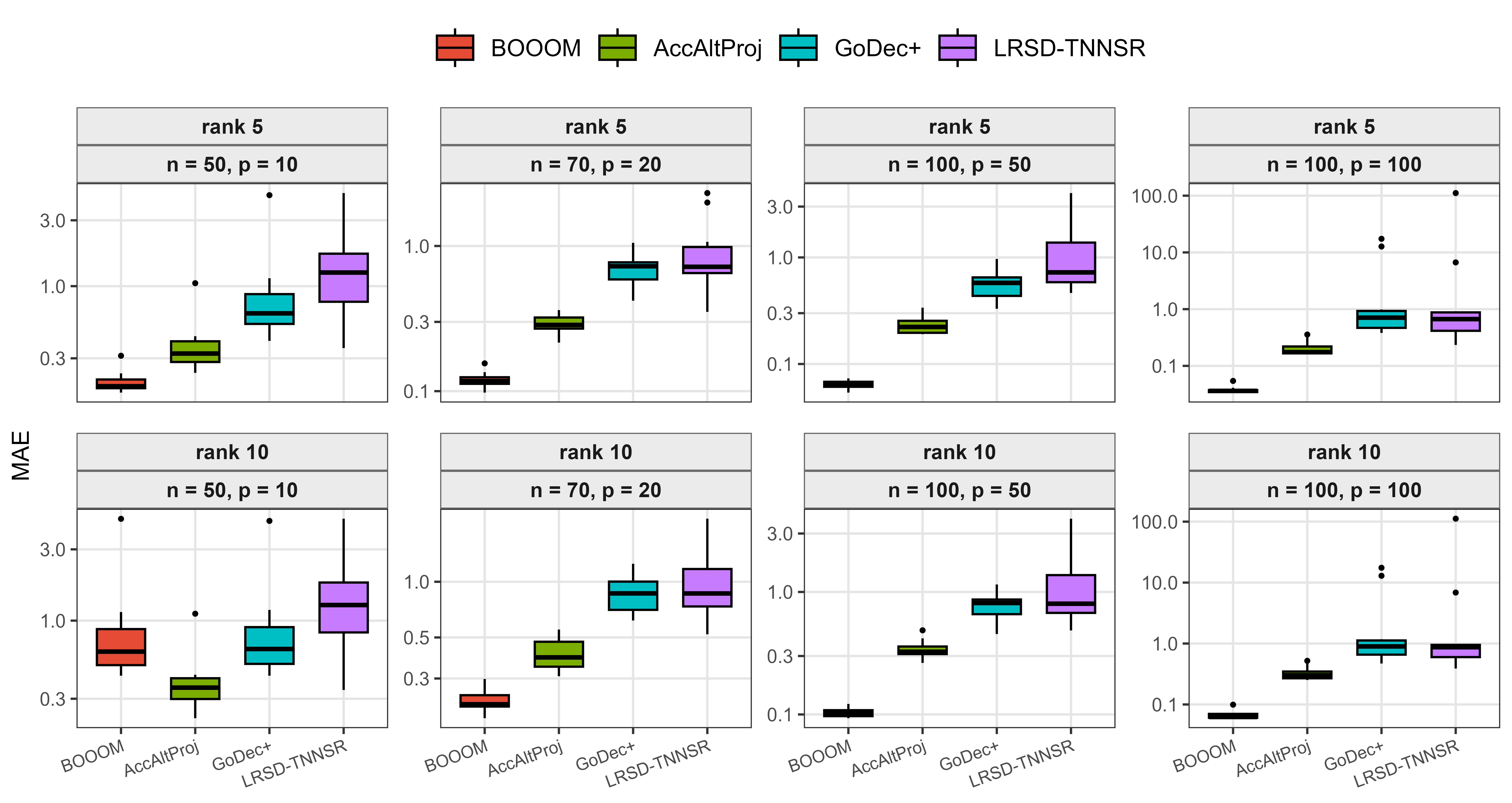}
	\caption{Low rank and Sparse Matrix Decomposition: Mean absolute error (MAE) comparison on low rank reconstruction $\hat{L}$ and true low rank matrix $L$ under synthetic data $X$ size is $(n\times p)$.  Upper panel: rank $d=5$ on $L$. Lower panel: rank $d=10$ on $L$. Each boxplot summarizes $10$ Monte Carlo replicates.}
	\label{fig: low-rank decomposition}
\end{figure}

We consider matrix dimensions $(n,p) \in \{(50,10), (70,20), (100,50), (100,100)\}$ with intrinsic rank $d \in \{5,10\}$. For the proposed formulation \eqref{eqn:robust-pca-obj}, the estimated low-rank component is given by $\widehat{L} = XQQ^\top$. Competing methods produce $(\widehat{L}, \widehat{S})$ via their respective procedures. Unless otherwise specified, competing methods are implemented using parameter settings recommended in their original references for synthetic data experiments. For \eqref{eqn:robust-pca-obj}, the regularization parameter is set to $\lambda = 1/\sqrt{\max(n,p)}$, following \citep{candes2011robust}. Each experiment is repeated over 10 independent replicates.

Figure~\ref{fig: low-rank decomposition} compares the MAE of BOOOM, AccAltProj, GoDec+, and LRSD-TNNSR across all simulated settings. BOOOM attains the lowest reconstruction error in the vast majority of scenarios, with particularly strong performance in the rank-$5$ setting across all matrix dimensions $(n,p)$. In the higher-rank setting ($d=10$), AccAltProj performs competitively for smaller matrices $(50,10)$, but its accuracy deteriorates as the problem size increases, whereas BOOOM maintains consistently low error. In addition, BOOOM exhibits reduced variability across replicates, as evidenced by the tighter interquartile ranges in the boxplots, indicating stable performance. Overall, these results demonstrate that BOOOM achieves accurate low-rank recovery while remaining robust to increases in data dimension and intrinsic rank. Numerical summaries corresponding to these boxplots are reported in Appendix Table~\ref{tab:Low_rank_decomposition}.

\subsection{Independent Component Analysis}\label{subsec:ICA}
Independent Component Analysis (ICA) aims to recover statistically independent latent sources from their linear mixtures. ICA has become a fundamental technique in signal processing, neuroscience, and blind source separation \citep{hyvarinen2000independent,cardoso1998blind}. Given an observed data matrix $X \in \mathbb{R}^{p\times n}$ generated from the mixing model
\begin{equation}
X = A S,
\end{equation}
where $A \in \mathbb{R}^{p\times p}$ is an unknown mixing matrix and $S \in \mathbb{R}^{p\times n}$ contains mutually independent source signals, the goal of ICA is to estimate an unmixing matrix $W$ such that the recovered signals $\hat{S}=WX$ approximate the true independent components. Many ICA algorithms exploit non-Gaussianity of the latent sources. A widely used contrast function is based on the log-cosh nonlinearity, which leads to the objective
\begin{equation}
\max_{W \in \mathrm{St}(p,p)} 
\frac{1}{n}\sum_{t=1}^{n} \sum_{i=1}^{p} 
\log \cosh\!\left(a_1\, w_i^\top x_t \right),
\label{eqn:ica_obj}
\end{equation}
where $w_i$ denotes the $i$th row of $W$, $x_t$ is the $t$-th observation vector, and $a_1>0$ controls the slope of the log-cosh nonlinearity; in our experiments we set $a_1=1$. After whitening the observations, the unmixing matrix is constrained to be orthogonal, so the optimization variable lies on the Stiefel manifold $W \in \mathrm{St}(p,p)$. 

Several algorithms have been developed to solve the ICA problem efficiently. FastICA uses a fixed-point iteration based on maximizing non-Gaussianity and has become one of the most widely used ICA algorithms \citep{hyvarinen1999fast}. Infomax ICA performs maximum likelihood estimation via a natural gradient approach \citep{bell1995information}. Picard accelerates ICA estimation by combining quasi-Newton updates with preconditioning to improve convergence speed \citep{ablin2018faster}. We compare BOOOM with these established ICA methods by optimizing the log-cosh objective \eqref{eqn:ica_obj}. The competing algorithms include FastICA, Infomax ICA (runica), and Picard. It is important to note that these algorithms are not explicitly designed to optimize the log-cosh objective \eqref{eqn:ica_obj} under an orthogonality constraint; rather, they are derived from different contrast functions or likelihood formulations tailored to ICA. To enable a unified comparison, we evaluate all methods using the same log-cosh objective \eqref{eqn:ica_obj} as an external criterion. For a fair comparison under the orthogonality constraint $W \in \mathrm{St}(p,p)$, we post-process each estimated unmixing matrix by projecting it onto the nearest orthogonal matrix (via an SVD/Procrustes step) before computing the log-cosh objective and the performance index Amari distance \citep{Amari1995}.

\begin{figure}[h]
	\centering
    \includegraphics[width=\textwidth]{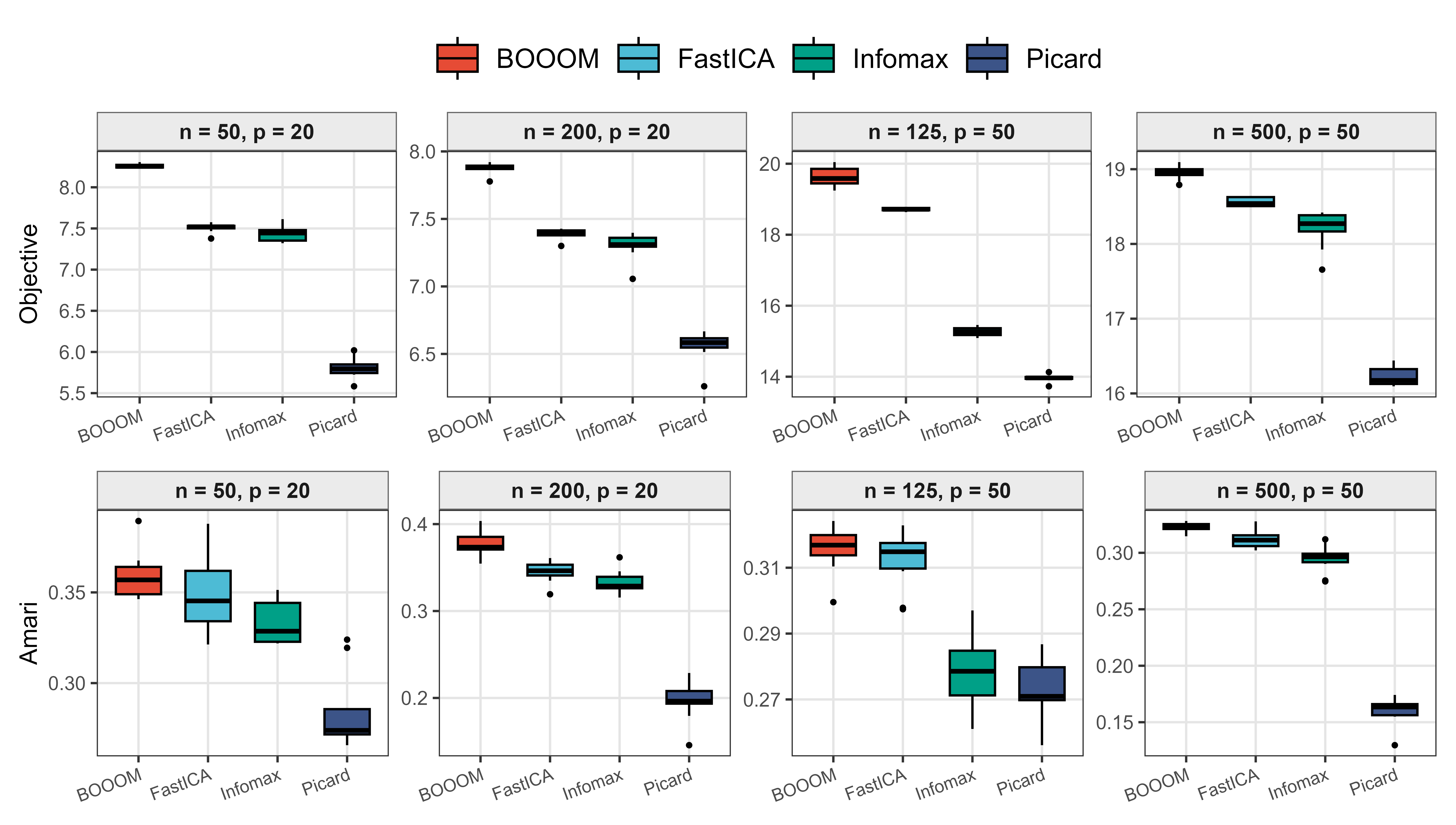}
	\caption{Independent Component Analysis: Performance comparison across four problem dimensions. Columns correspond to $(p,n) \in {(20,50), (20,200), (50,125), (50,500)}$. The top row shows the achieved values of the log-cosh objective \eqref{eqn:ica_obj} (higher is better), while the bottom row reports the Amari distance measuring source recovery accuracy (lower is better). Each boxplot summarizes $10$ Monte Carlo replicates.}
	\label{fig:ica_results}
\end{figure}

Synthetic data are generated according to the ICA model $X = AS$. First, latent sources $S \in \mathbb{R}^{p\times n}$ are generated with mutually independent components to ensure non-Gaussianity. For each component $i=1,\dots,p$, the source distribution cycles through three types: Laplace (super-Gaussian), Student-$t$ with three degrees of freedom, and Uniform on $[-\sqrt{3},\sqrt{3}]$ (sub-Gaussian). The generated sources are standardized to have zero mean and unit variance. Next, the mixing matrix $A \in \mathbb{R}^{p\times p}$ is constructed as $A = U \,\mathrm{diag}(s_1,\dots,s_p)\, V^\top$, where $U$ and $V$ are random orthogonal matrices obtained from QR decompositions of Gaussian matrices and the singular values $(s_1,\dots,s_p)$ follow a logarithmic spacing $s_i \in [1,10]$ to produce a moderately conditioned mixing system. The observed data are then formed as $X = AS$. Prior to ICA estimation, the observations are centered and whitened using the eigenvalue decomposition of the sample covariance matrix so that the transformed data have identity covariance. Under this whitening step, the unmixing matrix is orthogonal and the optimization variable lies on the Stiefel manifold $W \in \mathrm{St}(p,p)$.

We consider four dimensional configurations $(p,n) \in 
\{(20,50), (20,200), (50,125), (50,500)\}$. For each setting, $10$ Monte Carlo replicates are generated. Two performance metrics are evaluated. Figure~\ref{fig:ica_results} summarizes the results. Each column corresponds to one $(p,n)$ configuration. The top row displays the achieved objective values in \eqref{eqn:ica_obj}, while the bottom row reports the Amari distances across Monte Carlo replicates. It is important to emphasize that BOOOM directly optimizes the log-cosh objective in \eqref{eqn:ica_obj} under the orthogonality constraint, whereas the competing ICA methods are derived from different contrast functions or likelihood formulations and are not explicitly designed to optimize this objective. In particular, FastICA, Infomax (runica), and Picard employ their own estimation mechanisms, and the resulting unmixing matrices are not guaranteed to be orthogonal. Therefore, for a fair comparison, we project each estimated matrix onto the nearest orthogonal matrix on $\mathrm{St}(p,p)$ before evaluating the objective value and Amari distance. Across all scenarios, BOOOM consistently attains the highest objective values, reflecting its ability to directly optimize \eqref{eqn:ica_obj}. In contrast, Picard typically achieves the lowest Amari distances, followed by Infomax and FastICA, indicating more accurate source recovery under this metric. This discrepancy suggests that, for the considered simulation settings, maximization of the log-cosh objective and minimization of the Amari distance are not perfectly aligned. Overall, these results highlight that BOOOM successfully fulfills its primary goal of maximizing the target objective on the Stiefel manifold, while its Amari performance, although competitive, remains comparatively suboptimal relative to specialized ICA solvers that are tailored for source separation accuracy. The numerical summaries corresponding to these boxplots are provided in Appendix Table~\ref{tab:ica_results}.

\subsection{Varimax Factor Rotation}\label{subsec:Varimax}
Orthogonal factor rotation is a classical problem in exploratory factor analysis that seeks a rotated loading matrix with a simple and interpretable structure. Among orthogonal rotation criteria, the Varimax criterion is one of the most widely used methods for achieving sparse and interpretable factor loadings \citep{kairser1958varimax}. Given a loading matrix $A \in \mathbb{R}^{n\times p}$ obtained from an initial factor extraction procedure, the goal is to find an orthogonal rotation matrix $R$ such that the rotated loadings $B = AR$ exhibit a simple structure. The Varimax rotation maximizes the dispersion of squared loadings within each factor, encouraging each variable to load strongly on only a few factors. However, the resulting optimization problem is non-convex under the orthogonality constraint and may admit multiple local optima, making reliable optimization challenging. The corresponding objective function is

\begin{equation}
V(R) = \sum_{j=1}^{p} \left[
\frac{1}{n}\sum_{k=1}^{n} B_{kj}^{4}
-
\left(\frac{1}{n}\sum_{k=1}^{n} B_{kj}^{2}\right)^2
\right],
\qquad B = AR,
\label{eqn:varimax}
\end{equation}
subject to the orthogonality constraint $R \in \mathrm{St}(p,p)$. In our implementation, we minimize the negative Varimax objective $-V(R)$, thereby casting the problem as a minimization task on the Stiefel manifold.

We compare BOOOM with the classical Varimax rotation implemented in MATLAB's \texttt{rotatefactors} function, which computes the orthogonal rotation maximizing the Varimax criterion. The implementation follows standard iterative rotation procedures widely used in factor analysis \citep{jennrich2001simple}. Both methods estimate an orthogonal rotation matrix $R$ that maximizes the criterion in \eqref{eqn:varimax}. Synthetic loading matrices are generated to follow a simple structure. First, a true loading matrix $B_0 \in \mathbb{R}^{n\times p}$ is constructed such that each variable loads primarily on a single factor. For each row $k$, a dominant factor index $j$ is sampled uniformly from $\{1,\ldots,p\}$. The corresponding loading magnitude is drawn from $\text{Uniform}(0.8,1.2)$ with a random sign, i.e.,
$B_{0,kj} = s \cdot u$ where $u \sim \text{Uniform}(0.8,1.2)$ and $s = \mathrm{sign}(Z)$ with $Z \sim N(0,1)$. To produce a realistic loading structure, two additional factors are selected uniformly without replacement and assigned small Gaussian cross-loadings sampled from $N(0,0.05^2)$. All remaining entries are set to zero. The columns of $B_0$ are then normalized to unit $\ell_2$ norm to ensure comparable scaling across factors. Next, a random orthogonal rotation matrix $R_{\text{true}} \in \mathrm{St}(p,p)$ is generated via QR decomposition of a Gaussian matrix. The observed loading matrix is constructed as

\[
A = B_0 R_{\text{true}}^{\top}.
\]

Recovering the rotation matrix corresponds to maximizing the Varimax criterion applied to $A$. We consider eight dimensional configurations for the loading matrix $(n,p) =$ $(30,5)$, $(60,5)$, $(50,10)$, $(100,10)$, $(80,20)$, $(150,20)$, $(120,30)$, $(200,30)$. For each configuration, $10$ Monte Carlo replicates are conducted. 
\begin{figure}[h]
	\centering
    \includegraphics[width=\textwidth]{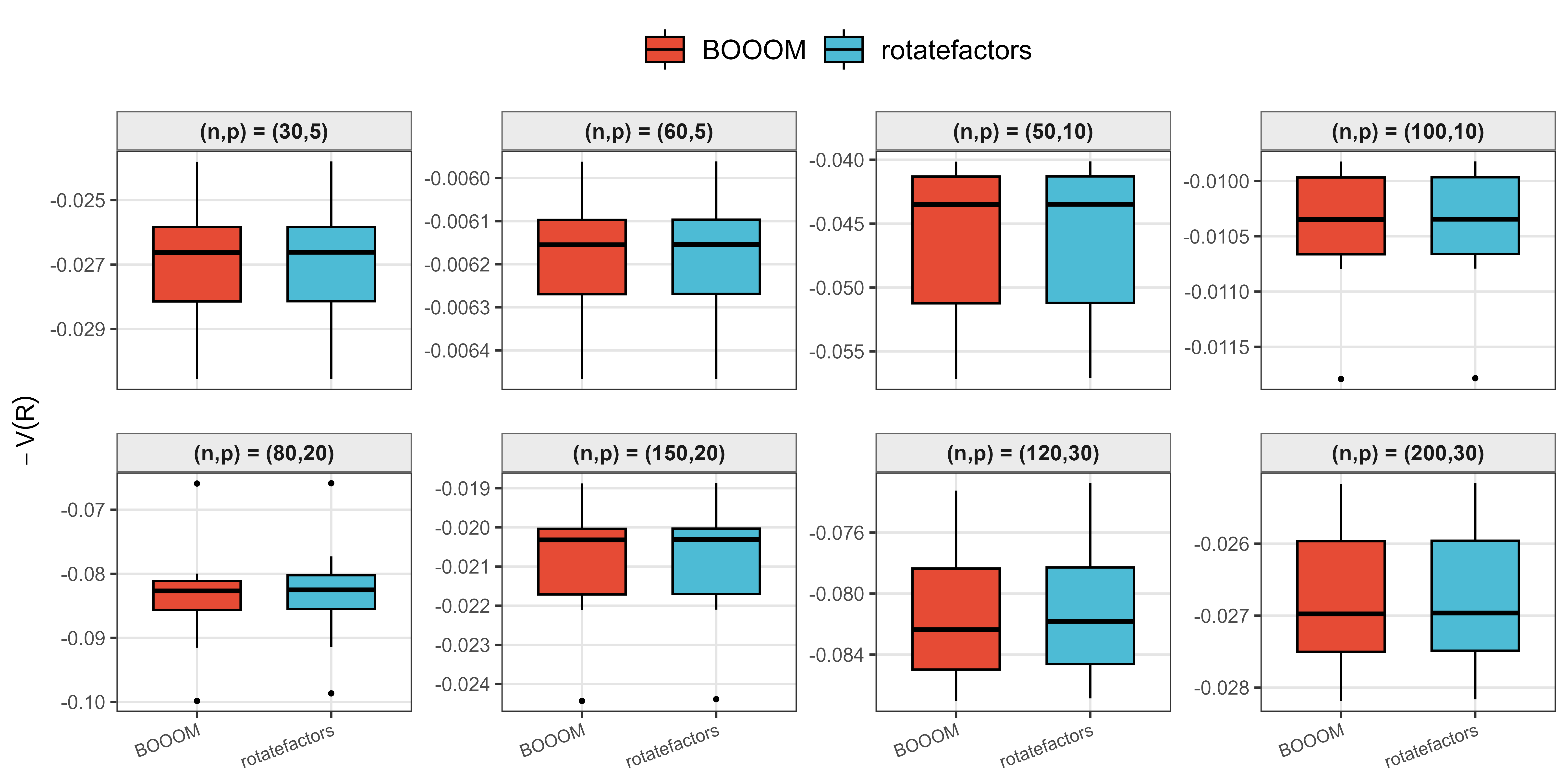}
	\caption{Varimax Factor Rotation: Comparison of optimization performance for the Varimax rotation problem across eight dimensional configurations $(n,p) \in \{(30,5),(60,5),(50,10),(100,10),(80,20),(150,20),(120,30),(200,30)\}$. The vertical axis reports the minimized objective $-V(R)$, where $V(R)$ is the classical Varimax criterion. Each boxplot summarizes $10$ Monte Carlo replicates for BOOOM and the MATLAB \texttt{rotatefactors} implementation.}
	\label{fig:varimax_results}
\end{figure}
The performance of each method is evaluated by the achieved value of the minimized objective $-V(R)$, along with computational time. The resulting objective values across Monte Carlo replicates are summarized in Figure~\ref{fig:varimax_results}. Across all dimensional configurations, BOOOM and the classical \texttt{rotatefactors} implementation achieve nearly identical objective values with strongly overlapping distributions. This behavior indicates that both methods are able to reliably optimize the Varimax criterion and recover rotations with comparable quality. The numerical summaries corresponding to these boxplots are provided in Appendix Table~\ref{tab:varimax_results}.

\subsection{Orthogonal Joint Diagonalization}\label{subsec:OJD}
Orthogonal joint diagonalization (AJD) is a fundamental problem arising in signal processing and blind source separation, where the objective is to find an orthogonal matrix that approximately diagonalizes a set of symmetric matrices simultaneously \citep{cardoso1996jacobi,pham2001joint}. This problem appears in applications such as independent component analysis, covariance estimation, and multi-set signal processing. However, the joint diagonalization objective is non-convex under orthogonality constraints, and optimization methods may converge to local optima. Let ${C_1,\ldots,C_m}$ be a collection of symmetric matrices in $\mathbb{R}^{p\times p}$. The goal is to find an orthogonal matrix $W$ such that each transformed matrix $W^\top C_k W$ becomes as diagonal as possible. A common formulation minimizes the sum of squared off-diagonal elements:
\begin{equation}
\min_{W \in \mathrm{St}(p,p)}
\sum_{k=1}^{m}
\left\| \operatorname{offdiag}\!\left(W^\top C_k W\right) \right\|_F^2
\label{eqn:AJD}
\end{equation}
where $\mathrm{offdiag}(\cdot)$ extracts the off-diagonal entries of a matrix and $|\cdot|_F$ denotes the Frobenius norm. The optimization is performed over the Stiefel manifold $\mathrm{St}(p,p)$ consisting of orthogonal matrices.

We compare BOOOM with three established approaches for orthogonal joint diagonalization: the classical Jacobi-based AJD algorithm \citep{cardoso1996jacobi}, Riemannian gradient descent on the Stiefel manifold, and the Riemannian trust-region method implemented in the Manopt toolbox \citep{boumal2014manopt}. Synthetic matrices are generated according to a standard joint diagonalization model. For each $k=1,\ldots,m$, the diagonal entries of $D_k$ are generated as
\[
d_k = \tilde d + 0.2\,\varepsilon_k,
\qquad
\tilde d_i = 0.5 + \frac{2.0-0.5}{p-1}(i-1), \quad i=1,\ldots,p,
\]
where $\varepsilon_k \sim N(0,I_p)$ and $D_k=\mathrm{diag}(d_k)$. The noiseless jointly diagonalizable matrices are then defined as
\[
C_k^{(0)} = W_{\text{true}} D_k W_{\text{true}}^\top .
\]
To create a realistic approximate joint diagonalization scenario, symmetric Gaussian noise is added:
\[
C_k = C_k^{(0)} + \sigma E_k ,
\]
where $E_k$ is a random symmetric matrix with entries sampled from a standard normal distribution and $\sigma = 0.1$ controls the noise level.

We consider four dimensional configurations defined by the matrix dimension $p$ and the number of matrices $m$:
$(p,m) \in \{(20,5), (20,10), (50,5), (50,10)\}$. For each configuration, $10$ Monte Carlo replicates are generated following the approximate joint diagonalization model described above. The performance of each method is evaluated using the achieved objective value in \eqref{eqn:AJD}. Figure~\ref{fig:AJD_results} summarizes the resulting objective values for the Jacobi AJD algorithm, Riemannian gradient descent, the Riemannian trust-region method, and BOOOM. Across all four configurations, BOOOM, Riemannian gradient descent, and the trust-region method achieve very similar objective values, with largely overlapping distributions across replicates, while BOOOM tends to obtain slightly lower values than Riemannian gradient descent in several settings. In contrast, the classical Jacobi-based AJD algorithm consistently attains noticeably larger objective values, indicating less effective minimization of the off-diagonal criterion. The results therefore suggest that BOOOM performs competitively with modern Riemannian optimization approaches while consistently improving upon the classical Jacobi AJD procedure. These findings highlight the robustness of BOOOM for non-convex optimization problems on the Stiefel manifold arising in approximate joint diagonalization. The numerical summaries corresponding to these boxplots are provided in Appendix Table~\ref{tab:AJD_results}.
\begin{figure}[h]
	\centering
    \includegraphics[width=0.7\textwidth]{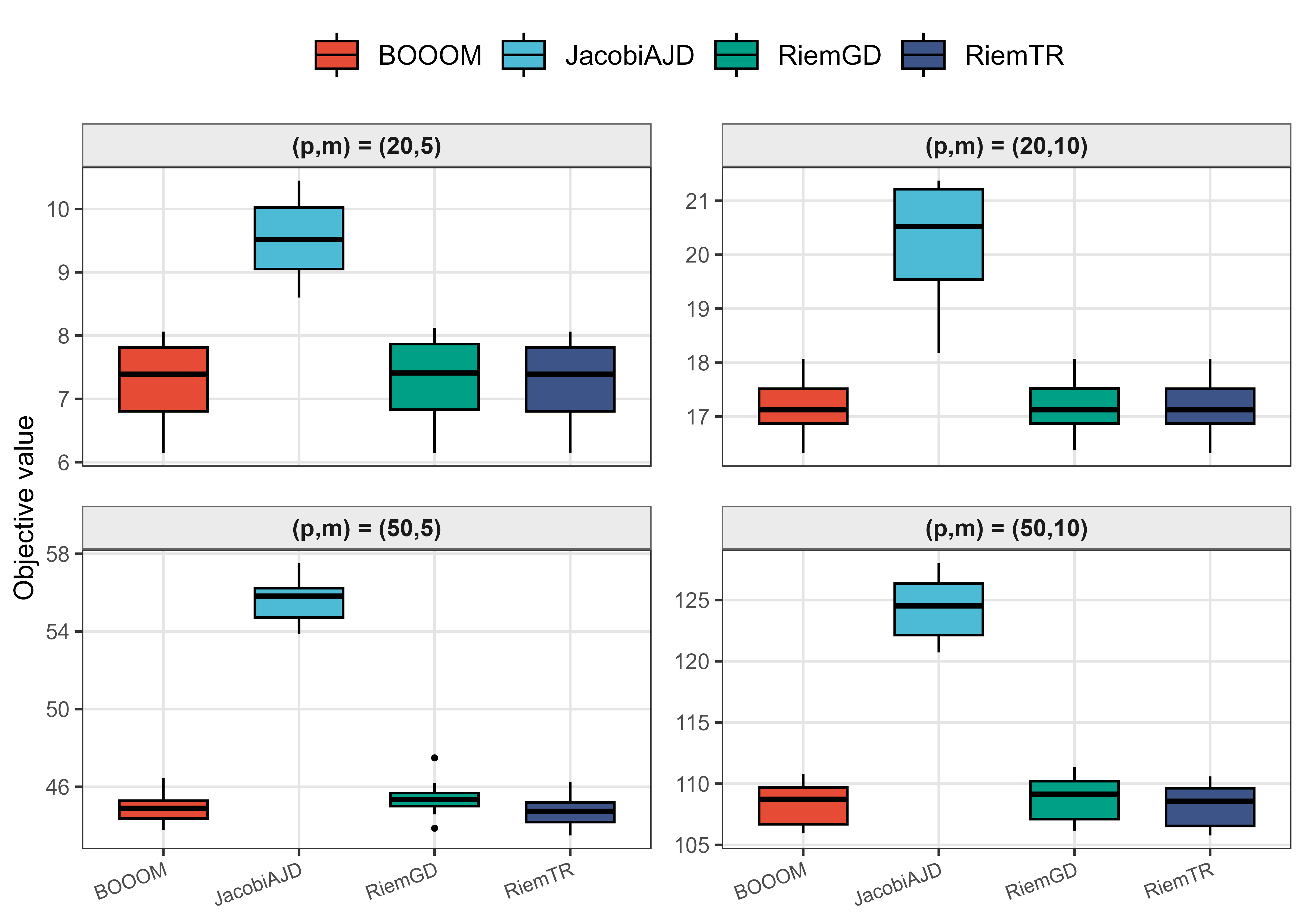}
	\caption{Orthogonal Joint Diagonalization: Comparison of optimization performance for the orthogonal joint diagonalization problem across four dimensional configurations $(p,m) \in \{(20,5),(20,10),(50,5),(50,10)\}$. The vertical axis reports the achieved objective value in \eqref{eqn:AJD}, corresponding to the sum of squared off-diagonal elements of the transformed matrices. Each boxplot summarizes $10$ Monte Carlo replicates for BOOOM, the Jacobi AJD algorithm, Riemannian gradient descent, and the Riemannian trust-region method.}
	\label{fig:AJD_results}
\end{figure}

\subsection{Reduced Kohn--Sham Rayleigh--Ritz Optimization}\label{subsec:KS}
Kohn--Sham density functional theory (DFT) is a fundamental computational framework for electronic structure calculations in quantum chemistry and condensed matter physics \citep{kohn1965self}. In practical implementations, solving the Kohn--Sham equations requires repeatedly computing the lowest eigenstates of a Hamiltonian operator. This task is typically performed using Rayleigh--Ritz optimization on orthogonality-constrained matrices \citep{goedecker1999linear,lin2012adaptive}. Consequently, efficient optimization on the Stiefel manifold is a key computational component in modern electronic structure solvers. Direct large-scale Kohn--Sham problems involve extremely high-dimensional Hamiltonian operators defined on plane-wave grids or large basis sets. To construct a reproducible and computationally tractable benchmark while retaining the essential optimization structure, we consider a reduced Kohn--Sham Rayleigh--Ritz formulation. Such reduced problems arise naturally in subspace iteration and adaptive basis methods, where the Hamiltonian is projected onto a smaller orthonormal basis \citep{lin2012adaptive}.

Let $H^\ast$ denote the frozen Kohn--Sham Hamiltonian obtained after a self-consistent field (SCF) calculation. Given an orthonormal basis matrix $B \in \mathbb{R}^{N_g \times p}$, the Hamiltonian projected onto the reduced subspace is

\[
H_{\mathrm{red}} = B^\top H^\ast B ,
\]
where $p$ denotes the reduced subspace dimension. The Rayleigh--Ritz problem then seeks an orthonormal matrix 
$Q \in \mathrm{St}(p,d)$ that minimizes

\begin{equation}
\min_{Q \in \mathrm{St}(p,d)}
\operatorname{tr}(Q^\top H_{\mathrm{red}} Q),
\label{eqn:ks_rr}
\end{equation}
where $d$ denotes the number of occupied orbitals. This formulation corresponds to computing the $d$ lowest eigenvectors of the reduced Hamiltonian while enforcing orthogonality constraints.

We compare BOOOM with a Riemannian conjugate gradient (RCG) method on the Stiefel manifold, implemented using the Manopt toolbox \citep{boumal2014manopt}. As a reference solution, the exact eigenvalue decomposition of the reduced Hamiltonian $H_{\mathrm{red}}$ is computed, which provides the optimal value of the Rayleigh--Ritz objective. To construct the benchmark problem, we first perform a Kohn--Sham self-consistent field (SCF) calculation for a hydrogen molecule (H$_2$) using the KSSOLV~2.0 package \citep{Jiao2022kssolv}. This calculation produces the converged Kohn--Sham Hamiltonian $H^\ast$ together with the corresponding set of occupied orbitals $U^\ast$. The Hamiltonian $H^\ast$ is then treated as a fixed operator in the subsequent experiments. Next, a reduced orthonormal basis $B \in \mathbb{R}^{N_g \times p}$ is constructed. Let $k_0$ denote the number of converged Kohn--Sham orbitals obtained from the SCF calculation. The first $k_0$ columns of $B$ are set equal to these orbitals $U^\ast$. The remaining $p-k_0$ columns are generated by sampling a Gaussian random matrix $Z \in \mathbb{R}^{N_g \times (p-k_0)}$ with independent $N(0,1)$ entries. To ensure orthogonality with the orbitals $U^\ast$, the sampled matrix is first projected onto the orthogonal complement of $\mathrm{span}(U^\ast)$,
\[
Z \leftarrow Z - U^\ast(U^{\ast\top} Z).
\]
A thin QR decomposition is then applied to the projected matrix,
\[
Z = QR,
\]
and the first $p-k_0$ columns of $Q$ are used to complete the basis $B$. The reduced Hamiltonian used in the Rayleigh--Ritz optimization is then defined as
\[
H_{\mathrm{red}} = B^\top H^\ast B .
\]
Given $H_{\mathrm{red}}$, the optimization problem seeks an orthonormal matrix 
$Q \in \mathbb{R}^{p \times d}$ satisfying $Q^\top Q = I_d$ that minimizes the 
Rayleigh--Ritz objective
\begin{equation}
\min_{Q \in \mathrm{St}(p,d)} 
\operatorname{tr}(Q^\top H_{\mathrm{red}} Q).
\label{eq:ks_rr_objective}
\end{equation}
The exact eigenvalue decomposition of $H_{\mathrm{red}}$ provides the optimal solution of \eqref{eq:ks_rr_objective} and serves as the ground-truth reference when computing objective value gaps and optimality residuals for the optimization methods. We consider reduced dimensions
\[
p \in \{20, 50, 80, 100\}, 
\qquad d = 2,
\]
and perform $10$ Monte Carlo replicates. 
\begin{figure}[!h]
	\centering
    \includegraphics[width=\textwidth]{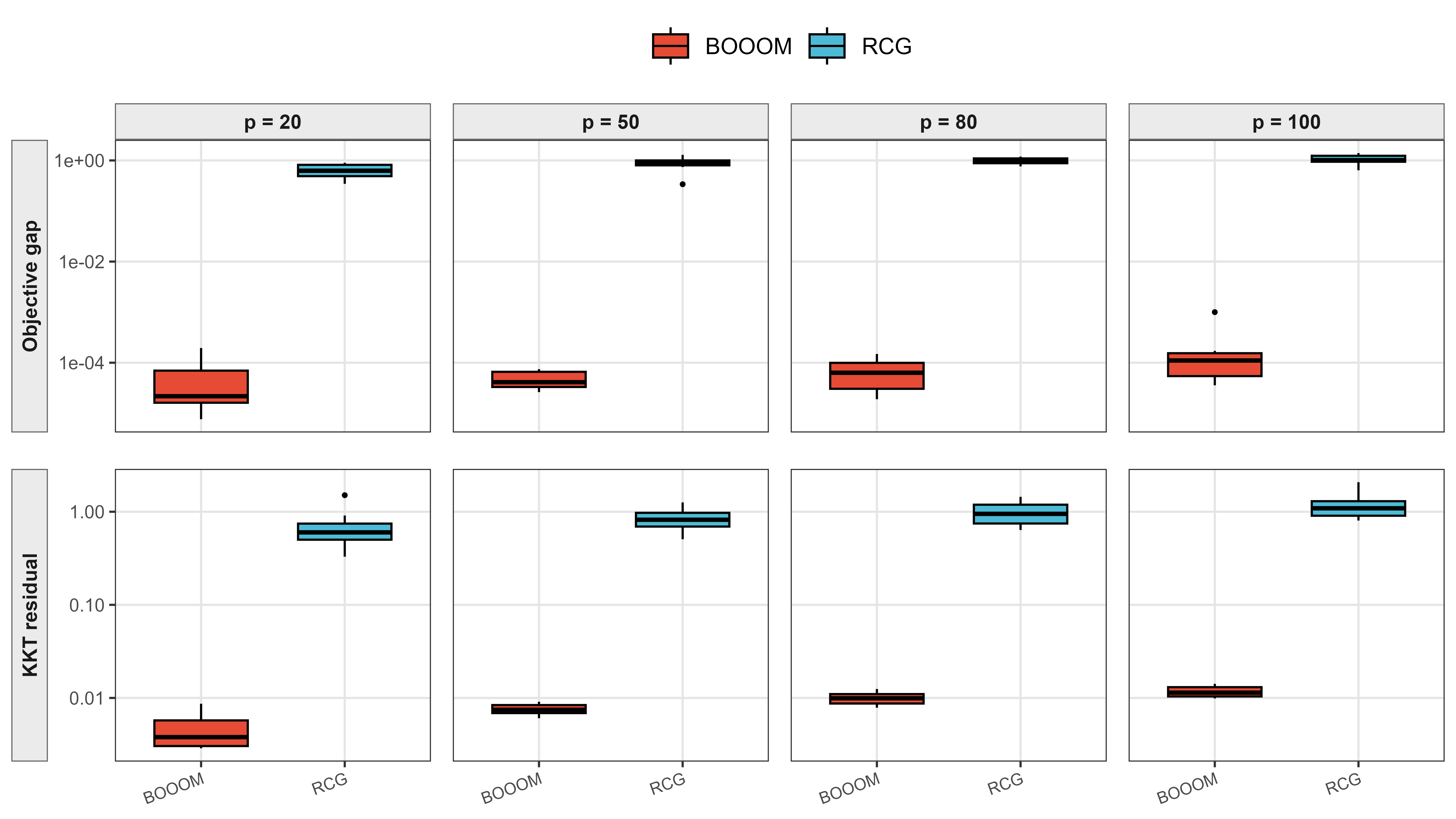}
	\caption{Reduced Kohn--Sham Rayleigh--Ritz optimization: Each column corresponds to a reduced subspace dimension $p \in \{20,50,80,100\}$ with $d=2$ extracted eigenvectors. The top row reports the objective value gap relative to the optimal solution obtained from the exact eigenvalue decomposition of $H_{\mathrm{red}}$, while the bottom row shows the KKT residual measuring violation of the first-order optimality conditions. Results are based on $10$ Monte Carlo replicates and the vertical axis is shown on a logarithmic scale.}
	\label{fig:ks_results}
\end{figure}
The parameter $d$ denotes the number of eigenvectors extracted in the Rayleigh--Ritz optimization. In electronic structure calculations this quantity corresponds to the number of Kohn--Sham orbitals being computed within the projected subspace. For the hydrogen molecule considered in our benchmark, the electronic system contains only two electrons, so the number of physically relevant orbitals is very small. Consequently we set $d=2$ in order to compute the lowest few eigenstates of the reduced Hamiltonian. Such small values of $d$ are typical in Rayleigh--Ritz subspace iterations used in electronic structure calculations, where the primary computational challenge lies in the large ambient dimension of the Hamiltonian rather than the number of eigenstates being extracted \citep{goedecker1999linear,lin2012adaptive}. Consequently, the optimization variable $Q \in \mathrm{St}(p,d)$ has a small column dimension $d$ while the reduced subspace dimension $p$ varies across different problem sizes. 

Performance is evaluated using two metrics that assess solution accuracy and optimality. The first metric is the objective value gap, defined as the difference between the obtained objective value and the optimal value given by the exact eigenvalue decomposition of $H_{\mathrm{red}}$. The second metric is the Karush--Kuhn--Tucker (KKT) residual, which measures the norm of the first-order optimality condition associated with the Stiefel-constrained optimization problem. Together, these metrics quantify how closely the computed solution approaches the optimal Rayleigh--Ritz solution and how well the orthogonality-constrained optimality conditions are satisfied. Figure~\ref{fig:ks_results} presents boxplots of the objective value gap and KKT residual across Monte Carlo replicates for BOOOM and RCG. These metrics quantify both solution accuracy and convergence quality relative to the exact eigen solution. Across all reduced problem sizes, BOOOM consistently achieves smaller objective gaps and lower optimality residuals while maintaining competitive computational time, demonstrating its effectiveness for orthogonality-constrained optimization problems arising in electronic structure calculations. The numerical summaries corresponding to these boxplots are provided in Appendix Table~\ref{tab:ks_results}.

\section{Applications and Modeling Flexibility of BOOOM}\label{sec:case_study}
While BOOOM is readily applicable to a wide range of existing optimization problems on the Stiefel manifold, its primary advantage lies in its modeling flexibility. In many statistical and machine learning problems, formulating objectives with orthogonality constraints often necessitates designing problem-specific optimization algorithms, creating a bottleneck that limits methodological development. BOOOM addresses this challenge by providing a general-purpose, black-box optimization framework for $\mathrm{St}(p,d)$. Once an objective is specified over the Stiefel manifold, BOOOM can be applied directly, regardless of smoothness or availability of derivatives, thereby decoupling model formulation from optimization. To illustrate this flexibility, we consider a novel supervised sparse principal component analysis formulation motivated by metabolite profiling for colorectal cancer. This example demonstrates how BOOOM can optimize a composite, non-convex objective combining reconstruction, sparsity, and discriminative structure without requiring a specialized solver.

\medskip
To improve interpretability in PCA, sparse PCA incorporates variable selection, while supervised extensions leverage class labels to enhance discriminative power \citep{feng2019supervised, shi2020supervised}. In many applications, however, one seeks a low-dimensional representation that simultaneously (i) reconstructs the data well, (ii) selects a subset of informative variables, and (iii) separates predefined groups. These competing objectives naturally lead to a composite formulation that balances reconstruction, sparsity, and discrimination. Motivated by this perspective, we consider the following supervised sparse PCA formulation:
\begin{equation}
\label{eqn:ACS_obj}
\min_{Q \in \mathrm{St}(p,d)}\;
\|X - XQQ^\top\|_F^{2}
+ \lambda_1 \|Q\|_{2,1}
+ \lambda_2 \mathcal{L}_{\mathrm{Fisher}}(XQ, Y).
\end{equation}
Here, $X \in \mathbb{R}^{n \times p}$ denotes the data matrix, $\|\cdot\|_{2,1}$ promotes row-wise sparsity in $Q$, and $\mathcal{L}_{\mathrm{Fisher}}$ is a discriminative loss. The $L_{2,1}$ norm is defined as
\[
\|Q\|_{2,1} = \sum_{i=1}^{p} \|Q_{i\cdot}\|_2,
\]
encouraging selection of a subset of variables. Let $\mathcal{I}_0 = \{i : Y_i = 0\}$ and $\mathcal{I}_1 = \{i : Y_i = 1\}$ denote the class index sets, with sizes $n_0$ and $n_1$. The projected class means are
\[
\mu_0 = \frac{1}{n_0} \sum_{i \in \mathcal{I}_0} (XQ)[i,:], 
\quad
\mu_1 = \frac{1}{n_1} \sum_{i \in \mathcal{I}_1} (XQ)[i,:],
\]
and the Fisher loss is given by
\[
\mathcal{L}_{\text{Fisher}}(XQ, Y)
= \sum_{i \in \mathcal{I}_0} \| (XQ)[i,:] - \mu_0 \|^2 
+ \sum_{i \in \mathcal{I}_1} \| (XQ)[i,:] - \mu_1 \|^2.
\]
The objective balances reconstruction accuracy, sparsity, and class separation, leading to an inherent trade-off. We therefore analyze the Pareto-optimal solutions \citep{collette2003multiobjective}, where no alternative simultaneously improves sparsity and classification performance.

We apply this model to a gut microbiome–metabolome dataset to identify metabolites associated with colorectal cancer while maintaining discrimination between healthy individuals and colorectal cancer patients. The original dataset consisted of 450 metabolites measured in 347 subjects. Rare metabolites with substantial missingness were removed, leaving 110 metabolites for analysis. These measurements were log-transformed, followed by quantile normalization to standardize the overall distribution of metabolite intensities across samples and improve comparability. The normalized data were then centered and scaled to unit variance to prevent metabolites with larger variance from dominating the reconstruction and discrimination objectives. To focus the analysis on the comparison between healthy individuals and colorectal cancer patients, subjects with a history of colorectal surgery or multiple polypoid adenomas with low-grade dysplasia were excluded. The final analytic sample included 277 subjects, including 127 healthy individuals and 150 colorectal cancer patients.

We optimize the objective in \eqref{eqn:ACS_obj} using BOOOM. The tuning parameters $(\lambda_1,\lambda_2)$ are selected over a logarithmically spaced grid 
to explore a broad range of sparsity-discrimination trade-offs. The resulting Pareto frontier, shown in Figure~\ref{fig: pareto_curves}, characterizes the relationship between sparsity and classification performance. For instance, $(\lambda_1,\lambda_2) = (10^6, 10^3)$ yields higher discriminative power with reduced sparsity, whereas smaller values of $(\lambda_1,\lambda_2)$ promote sparser solutions at the cost of slightly increased misclassification. The $L_{2,1}$ penalty induces row-wise sparsity in $Q$, effectively selecting a subset of metabolites that contribute to the learned low-dimensional representation. Consequently, the magnitude $\|Q[i,:]\|_2 \in [0,1]$ provides a natural, scale-normalized importance score for each metabolite. Figure~\ref{fig: importance_vars} displays the top 20 metabolites ranked by this measure under the tuning configuration with minimal misclassification error, where larger row norms indicate stronger contributions to $XQ$. 
\begin{figure}[!h]
	\centering
	\includegraphics[width=0.8\textwidth]{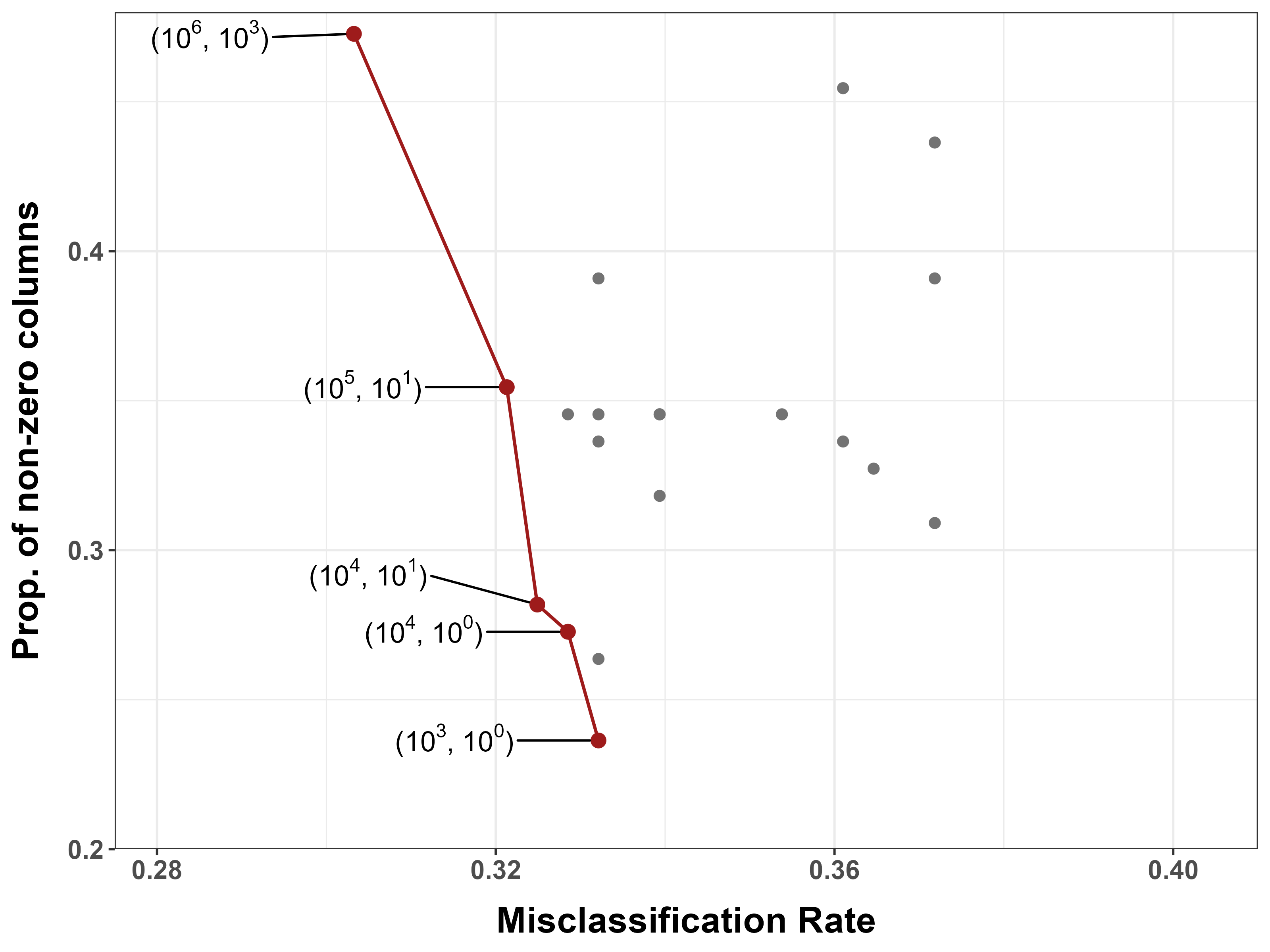}
	\caption{Pareto curves showing the relationship between sparsity (proportion of non-zero columns) and accuracy (misclassification rate). The red line marks the Pareto-optimal frontier obtained across combinations of $(\lambda_1$, $\lambda_2$) with selected tuning parameter values labeled at their corresponding points.}
	\label{fig: pareto_curves}
\end{figure}

Several metabolites among the top-ranked features have well-established associations with colorectal cancer (CRC). \textit{Citrulline} (rank 2) exhibits reduced circulating levels in CRC patients relative to healthy controls \citep{bednarz2020arginine}. \textit{Serum pipecolic acid} (rank 4) has been reported to be elevated in CRC patients \citep{hashim2020global}, while \textit{glycine} (rank 5) shows increased serum levels in CRC \citep{tevini2022changing}. Similarly, \textit{N\textsuperscript{1}-acetylspermidine} (rank 7) is significantly elevated in CRC patients compared to healthy individuals \citep{udo2020urinary}.
\begin{figure}[!h]
	\centering
	\includegraphics[width=0.8\textwidth]{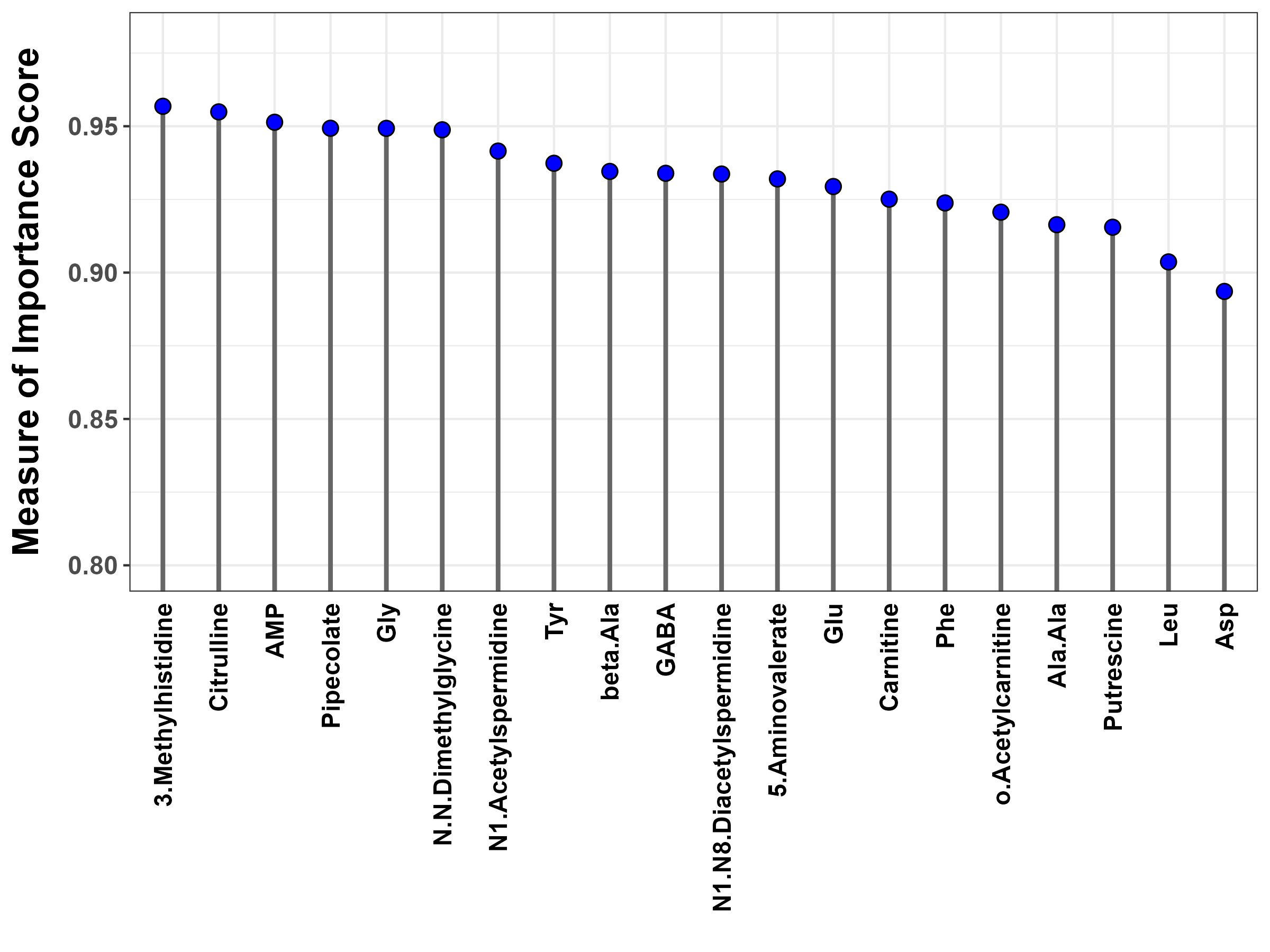}
	\caption{Top 20 metabolites ranked by their importance scores, defined by the row norms of the loading matrix $Q$. Higher scores indicate stronger contributions to the learned low-dimensional representation, highlighting metabolites most relevant to colorectal cancer.}
	\label{fig: importance_vars}
\end{figure}
Other highly ranked metabolites reflect broader alterations in CRC-associated metabolic pathways rather than acting as isolated biomarkers. \textit{3-methylhistidine} (rank 1), while not a standalone marker, has been linked to CRC through an increased methylhistidine-to-histidine ratio in serum \citep{avram2025metabolomic}. \textit{AMP} (rank 3) is not directly diagnostic, but elevated extracellular levels in the tumor microenvironment can enhance \textit{adenosine monophosphate deaminase (AMPase)} activity, which has potential relevance for CRC detection \citep{shi2022potential}. \textit{Dimethylglycine} (rank 6) is metabolically converted to \textit{sarcosine} and subsequently to \textit{glycine}, both of which are elevated in CRC \citep{bernasocchi2024subcellular}. Although \textit{tyrosine} and \textit{GABA} (ranks 8 and 10) are not standalone biomarkers, they exhibit CRC-specific shifts in microbiome associations, including weakened correlations with \textit{Synergistes sp.\ 3\_1\_syn} and \textit{Porphyromonas gingivalis}, respectively \citep{coker2022altered}. Finally, \textit{beta-alanine} (rank 9) metabolism has been identified as a significantly enriched pathway distinguishing CRC patients from healthy controls \citep{li2022metabolomic}.

Overall, this application demonstrates that BOOOM can serve as a unified optimization engine for complex Stiefel-constrained problems. By decoupling model design from optimization, it allows practitioners to formulate rich, application-driven objectives without developing bespoke algorithms. The empirical results further show that such flexibility does not come at the cost of performance or interpretability, with the learned representations recovering biologically meaningful signals in colorectal cancer.

\section{Discussion}\label{sec:discussion}
This work introduces \textsc{BOOOM}, a general-purpose framework for optimization over the Stiefel manifold that departs fundamentally from the dominant paradigm of local, derivative-driven methods. By combining a global Givens-based parametrization with a derivative-free exploration strategy rooted in RMPS, the method provides a unified mechanism for handling black-box, non-smooth, and highly multimodal objectives while preserving feasibility exactly. Across a diverse collection of benchmark problems, ranging from classical multimodal test functions to ICA, low-rank decomposition, and Kohn--Sham eigenvalue problems, BOOOM consistently demonstrates strong empirical performance, often outperforming both Riemannian optimization methods (e.g., RCG, RTR, RGD, RBB via \texttt{Manopt}) and general-purpose constrained solvers (e.g., \texttt{fmincon} variants).

At a conceptual level, the primary contribution is not merely algorithmic but structural: the angle-space reformulation decouples constraint handling from optimization, enabling a fully unconstrained search while maintaining exact orthogonality. This perspective allows BOOOM to ``jump across basins'' in a way that is fundamentally inaccessible to purely local methods. In particular, the structured rotational moves induced by Givens updates provide a deterministic yet globally explorative mechanism, contrasting sharply with both gradient-based descent (which is inherently local) and stochastic metaheuristics (which often lack structure and efficiency). The additional ability to evaluate candidate rotations in parallel further enhances the practical appeal of the method, especially in settings where objective evaluations are expensive.

Despite these strengths, BOOOM also exhibits several important limitations that merit careful discussion. The most immediate limitation is computational. Each BOOOM iteration evaluates $2\binom{p}{2}$ candidates, leading to an $O(p^2)$ per-iteration cost in terms of function evaluations. While this exhaustive polling strategy is central to the method’s strong exploratory capability, it becomes a bottleneck in high-dimensional settings. In particular, problems where large Stiefel manifolds arise, such as high-dimensional ICA (e.g., $p \gtrsim 500$), large-scale orthogonal joint diagonalization, or subspace optimization in electronic structure calculations (e.g., plane-wave Kohn--Sham problems with large basis sizes), may render BOOOM much slower than derivative-based alternatives. Parallelization partially mitigates this issue, as the candidate evaluations are embarrassingly parallel. However, even with parallel implementation, BOOOM remains computationally heavier than some of the local Riemannian methods or first-order algorithms that exploit gradient structure. BOOOM should therefore be viewed as complementary rather than universally superior: its advantages are most pronounced in black-box, non-smooth, or highly multimodal settings where local methods struggle.

A second limitation concerns the global convergence guarantees. The theoretical result establishing global convergence in probability relies on the open-ball reachability property under infinitely many full-support restarts. While this provides a clean and rigorous justification, it is inherently asymptotic and not directly reflective of practical implementations, where only a finite number of runs is feasible. In practice, BOOOM employs a small number of restarts (often fewer than 10), combined with warm-start initialization and large initial step sizes. Empirically, this strategy performs remarkably well and consistently outperforms both classical metaheuristics (e.g., Genetic Algorithms, Simulated Annealing) and local manifold optimizers in our experiments \citep{kim2026smart, Das2023MsiCOR, Das2023RMPSH}. Nevertheless, the theoretical framework does not yet fully capture this behavior. In particular, the current analysis treats restarts as the sole driver of global exploration, whereas a substantial portion of BOOOM’s practical effectiveness appears to arise from the intrinsic exploratory capacity of RMPS within a single run. Bridging this gap, by developing finite-restart or single-run global guarantees, remains an important open problem. In the present setting, understanding how structured rotational polling interacts with the geometry of the Stiefel manifold to enable effective exploration is an important direction for future work. 

Lastly, the Givens-based parametrization, while enabling unconstrained optimization, also introduces redundancy due to its many-to-one nature. Multiple angle configurations correspond to the same point on $\mathrm{St}(p,d)$, enlarging the effective search space and potentially slowing convergence. Although this redundancy can occasionally aid exploration, it contributes to the overall computational burden, suggesting that more parsimonious parametrizations or dimension-reduction strategies in angle space may improve efficiency.

These limitations point to several directions for future research. On the algorithmic side, improving scalability through randomized or block-coordinate polling, adaptive selection of rotation subsets, or hybrid schemes combining BOOOM’s global exploration with local gradient-based refinement is a natural priority. On the theoretical side, developing guarantees under finite restart budgets, characterizing single-run exploration behavior, and establishing non-asymptotic convergence results would significantly strengthen the framework. Extensions to other manifolds or constrained geometries may further broaden applicability.

In summary, BOOOM provides a unified and flexible framework for orthogonality-constrained optimization, particularly well-suited to black-box and highly non-convex problems. While it does not replace efficient local methods in smooth settings, it fills an important gap by enabling structured global exploration with exact feasibility, supported by both theoretical grounding and strong empirical performance.

\acks{Authors declare no conflicts of interest.  \vspace{0.3cm}

\noindent PD is partially supported by the National Institutes of Health/National Cancer Institute Cancer Center Support Grant P30 CA016059.}

\section*{Code and data}
Code for BOOOM, including demonstration scripts and the real dataset used in this article, is available on GitHub at \href{https://github.com/bkim12/BOOOM}{https://github.com/bkim12/BOOOM}. \vspace{0.3cm}

\appendix
\renewcommand{\thefigure}{A.\arabic{figure}} 
\renewcommand{\thetable}{A.\arabic{table}}   
\setcounter{figure}{0}
\setcounter{table}{0}

\clearpage
\section{Proofs}\label{app:theorem}


In the following, we present the proofs of all the technical results presented in the paper.

\subsection{Proof of Lemma~\ref{thm:stiefel-givens}}

\begin{proof}
Let $Q \in \mathrm{St}(p,d)$ and fix any base matrix $Q_0 \in \mathrm{St}(p,d)$. Since both $Q$ and $Q_0$ have orthonormal columns, they each admit orthogonal completions in $\mathrm{SO}(p)$. Let $\widetilde{Q}, \widetilde{Q}_0 \in \mathrm{SO}(p)$ be full orthogonal matrices such that
\begin{equation*}
    Q = \widetilde{Q}[:, 1\!:\!d], \qquad Q_0 = \widetilde{Q}_0[:, 1\!:\!d],    
\end{equation*}
\noindent i.e., $Q$ and $Q_0$ are the first $d$ columns of their respective completions. This is always possible via Gram-Schmidt, Householder QR, or any other orthogonalization procedure.

Now observe that $\mathrm{SO}(p)$ is a group under matrix multiplication, so for any fixed $\widetilde{Q}_0 \in \mathrm{SO}(p)$ and any $\widetilde{Q} \in \mathrm{SO}(p)$, there exists $\widetilde{Q}'$ such that
\begin{equation*}
    \widetilde{Q} = \widetilde{Q}' \cdot \widetilde{Q}_0.
\end{equation*}
\noindent By the classical Hurwitz theorem (Lemma~1 of \citealp{jiang2022givens}, following \citealp{Hurwitz1963}), every orthogonal matrix in $\mathrm{SO}(p)$ can be written as a product of Givens rotations, and hence we have the identity
\begin{equation*}
    \widetilde{Q}' = \prod_{(i,j)} R_{i,j}(\theta_{i,j})
    \quad \text{for some } \theta \in \mathbb{R}^{\binom{p}{2}}.    
\end{equation*}
\noindent Now, we have the chain of equality, 
\begin{equation*}
    Q = \widetilde{Q}[:, 1\!:\!d] = \left[ \prod_{(i,j)} R_{i,j}(\theta_{i,j}) \right] \widetilde{Q}_0[:, 1\!:\!d] = \left[ \prod_{(i,j)} R_{i,j}(\theta_{i,j}) \right] Q_0.
\end{equation*}
\noindent This completes the proof.
\end{proof}

\subsection{Proof of Proposition~\ref{thm:givens_param}}

\begin{proof}
First define $\Psi:\Theta\to\mathrm{SO}(p)$ by $\Psi(\theta)=\prod_{i<j}R_{i,j}(\theta_{i,j})$. Each factor $R_{i,j}(\cdot)$ is smooth in its angle and matrix multiplication is smooth, hence $\Psi$ is smooth. By the classical Hurwitz–Givens decomposition~\citep{Hurwitz1963}; Lemma~1 of~\cite{jiang2022givens}, every $U\in\mathrm{SO}(p)$ can be written as a product of planar Givens rotations, so $\Psi$ is surjective onto $\mathrm{SO}(p)$.

Next consider the left action of $\mathrm{SO}(p)$ on $\mathrm{St}(p,d)$, $A(U,Q)=UQ$. 
Fix any $Q_0 \in \mathrm{St}(p,d)$. 
To show that the orbit map $A_{Q_0}(U)=UQ_0$ is surjective, it suffices to prove: 
for every $Q \in \mathrm{St}(p,d)$ there exists $U \in \mathrm{SO}(p)$ with $UQ_0 = Q$. 
Indeed, given such a $Q$, choose orthogonal completions $\widetilde{Q}_0,\widetilde{Q} \in O(p)$ 
whose first $d$ columns equal $Q_0$ and $Q$, respectively; if $\det(\widetilde{Q}_0)$ or 
$\det(\widetilde{Q})$ is $-1$, flip the sign of a column outside the first $d$ so that 
$\widetilde{Q}_0,\widetilde{Q} \in \mathrm{SO}(p)$. Then 
$U := \widetilde{Q}\,\widetilde{Q}_0^\top \in \mathrm{SO}(p)$ satisfies $UQ_0 = Q$, 
establishing surjectivity of $A_{Q_0}$. Since $\Psi$ is surjective onto $\mathrm{SO}(p)$, 
the composition $\Phi = A_{Q_0} \circ \Psi$ is surjective onto $\mathrm{St}(p,d)$.

Composing the two maps yields $\Phi:=A_{Q_0}\circ\Psi$, i.e., $\Phi(\theta)=\Psi(\theta)Q_0=\big[\prod_{i<j}R_{i,j}(\theta_{i,j})\big]Q_0$, which is smooth as a composition of smooth maps and surjective since both $\Psi$ and $A_{Q_0}$ are surjective.

For the optimization equivalence, note that $\Phi(\theta)\in\mathrm{St}(p,d)$ for all $\theta$, so $\inf_{\theta} f(\Phi(\theta))\ge \inf_{Q\in\mathrm{St}(p,d)} f(Q)$. Conversely, because $f$ is continuous and $\mathrm{St}(p,d)$ is compact, a minimizer $Q^\star$ exists; by surjectivity of $\Phi$ there is $\theta^\star$ with $\Phi(\theta^\star)=Q^\star$, hence $\inf_{\theta} f(\Phi(\theta))\le f(\Phi(\theta^\star))=f(Q^\star)=\min_{Q} f(Q)$. The two inequalities give $\min_{\theta} f(\Phi(\theta))=\min_{Q} f(Q)$.
\end{proof}

\subsection{Proof of Corollary \ref{cor:rmps-in-image}}

\begin{proof}
Let $U_t:=\prod_{s=0}^{t-1} R_{i_s j_s}(\theta_s)$ for any $t \geq 1$. Each factor is orthogonal with determinant of $1$, hence $U_t \in \mathrm{SO}(p)$ and $Q^{(t)}=U_t Q^{(0)}$. By the surjectivity of $\Phi$ established in Proposition~\ref{thm:givens_param}, there exists $\vartheta_t \in\Theta$ such that $U_t=\prod_{i<j} R_{i,j}((\vartheta_{t})_{i,j} )$, leading to $Q_t=\Phi(\vartheta_t)$.
\end{proof}

\subsection{Proof of Theorem~\ref{thm:booom-stationarity-angles}}

\begin{proof}
We begin by proving the first part of the result. Pick any $i \in \{ 1, 2, \dots, N \}$, and consider the map $\tilde{g}_i : t \mapsto g(\nu + te_i)$. Clearly, $\tilde{g}_i \in \mathcal{C}^1(\mathbb{R})$. Since there is stationarity in the evaluations, from $g(\nu)\le g(\nu+\delta_k e_i)$ with $\delta_k\downarrow0$ we get
$\liminf_{t\downarrow0}\frac{g(\nu+t e_i)-g(\nu)}{t}\ge 0$. On the other hand, from $g(\nu)\le g(\nu-\delta_k e_i)$ we get
$\limsup_{t\uparrow0}\frac{g(\nu+t e_i)-g(\nu)}{t}\le 0$. Differentiability of $\tilde{g}_i$ then forces us to have
\begin{equation*}
    \partial_i g(\nu) = 0,
\end{equation*}
\noindent where $\partial_i$ denotes the partial differentiation operator with respect to $i$-th coordinate. Since, $i$ is arbitrary, we have $\nabla g(\nu)=\mathbf{0}$.

For the second part, we equip $\mathrm{St}(p,d)$ with the canonical embedded Riemannian metric $\langle \Xi,\Upsilon\rangle=\mathrm{tr}(\Xi^\top\Upsilon)$ on $T_Q\mathrm{St}(p,d)$ (see \citealp[Sec.~3.5]{AbsilMahonySepulchre2008}). Let $\mathrm{d}f(Q)$ be the differential operator of $f$ at $Q$. Then, by Riesz representation of $\mathrm{d}f(Q)$ under the metric on $T_Q\mathrm{St}(p,d)$ we obtain that
\begin{equation*}
    \mathrm{d}f(Q)[\Xi]=\langle \operatorname{grad} f(Q),\,\Xi\rangle\quad\text{for all }\Xi\in T_Q\mathrm{St}(p,d).
\end{equation*}
\noindent Since $\nu$ is not on a wrapping seam, $D\mathscr{M}(\nu)=I$, hence by the chain rule
\begin{equation*}
    \nabla g(\nu) \;=\; D(\,f\circ\Phi\circ\mathscr{M}\,)(\nu)^\top \;=\; D\widetilde{\Phi}(\nu)^\top\, (\text{grad} f(Q)),
\end{equation*}
\noindent Concretely, for any $v \in \mathbb{R}^N$, 
\begin{equation}\label{eq:chain-rule}
    \langle \nabla g(\nu),\,v\rangle_{\mathbb{R}^N}
    \;=\; \mathrm{d}f(Q)\big[D\widetilde{\Phi}(\nu)v\big]
    \;=\; \big\langle \operatorname{grad} f(Q),\, D\widetilde{\Phi}(\nu)v\big\rangle,
\end{equation}
\noindent If $\nabla g(\nu)=\bm{0}$ then the left-hand side of \eqref{eq:chain-rule} vanishes for every $v$, hence
\begin{equation*}
    \big\langle \operatorname{grad} f(Q),\, \Xi\big\rangle \;=\; 0
\qquad\text{for all }\Xi \in \mathrm{Im}\,D\widetilde{\Phi}(\nu).
\end{equation*}
Since $D\mathscr{M}(\nu)=I$, we have $D\widetilde{\Phi}(\nu)=D\Phi(\mathscr{M}(\nu))$. By the full-rank hypothesis, $\mathrm{Im}\,D\Phi(\mathscr{M}(\nu))=T_Q\mathrm{St}(p,d)$. Thus $\langle \operatorname{grad} f(Q),\,\Xi\rangle=0$ for every $\Xi\in T_Q\mathrm{St}(p,d)$, which, by positive definiteness of the inner product, implies $\operatorname{grad} f(Q)=\bm{0}$.
\end{proof}



\subsection{Proof of Lemma~\ref{thm:booom-open-ball}}
\begin{proof}
Without loss of generalized, we may shrink $\delta$ if needed so that $B_\delta(\nu^\star) \subset \Theta_0$. 

The proof follows from analyzing the behaviour of BOOOM algorithm under three different situations.

\begin{enumerate}
    \item[(i)] \emph{Mesh structure at a fixed step size:} Let us consider a phase of the BOOOM algorithm where the step size $s > 0$ is held fixed. During this time, RMPS polls the $2N$ axis-aligned neighbors $\nu\pm s e_i$ and, when an improvement is found, \emph{moves} by adding $\pm s e_i$ to the current iterate. Consequently, throughout such a phase every visited point lies on the affine lattice $\nu_{\text{start}} + s \mathbb{Z}^N$, where $\nu_{\mathrm{start}}\in \Theta_0$ is the point at which this phase began (the first iterate after the last step-size reduction or restart), and $\mathbb{Z}$ is the set of integers. This is the standard ``evolving mesh'' property of pattern search methods specialized to coordinate polling (cf.~\cite{torczon1997convergence}, Theorem~3.2; see also the survey \citealp{Kolda2003} for related mesh-based direct search schemes). Note that, since RMPS only moves in the direction where the objective value strictly improves, it visits each point on the lattice at most once. As a result, for a fixed step size $s > 0$, such a phase can run for at most $(M / s)^N$ iterations, where $M$ is the diameter of the compact set $\Theta_0$. 
    \item[(ii)] \emph{Behaviour of BOOOM inside a single run.} Now fix a run $r$ and consider BOOOM iterations. If the RMPS polls return no improvement, then the step size is decreased. As the step size $s^{(h)} \to 0$, there exists an index $h_0$ such that $s^{(h_0)} < \delta/(2\sqrt{N})$. If the BOOOM algorithm starts any phase with step size $s^{(h_0)}$, say at a starting point $\nu_{\mathrm{start}}\in\Theta_0$, then consider the vector with integral coordinates,
    \begin{equation*}
        z^\sharp \;:=\; \Big(\,\mathrm{round}\Big(\frac{\nu^\star_1-\nu_{\mathrm{start},1}}{s^{(h_0)}}\Big), \ldots, \mathrm{round}\Big(\frac{\nu^\star_N-\nu_{\mathrm{start},N}}{s^{(h_0)}}\Big)\,\Big)\in\mathbb{Z}^N.
    \end{equation*}
    \noindent Here, $\nu^\star_i$ is the $i$-th coordinate of $\nu^\star$, $\nu_{\mathrm{start},i}$ is the $i$-th coordinate of $\nu_{\mathrm{start}}$, and `$\mathrm{round}$' denotes rounding to the nearest integer (ties broken arbitrarily). By construction, 
    \begin{equation*}
        \big\|\nu_{\mathrm{start}} + s^{(h_0)} z^\sharp - \nu^\star \big\|_\infty \;\le\; \tfrac{s^{(h_0)}}{2}
        \quad\Rightarrow\quad
        \big\|\nu_{\mathrm{start}} + s^{(h_0)} z^\sharp - \nu^\star \big\|_2 \;\le\; \sqrt{N}\,\tfrac{s^{(h_0)}}{2} \;<\; \delta,
    \end{equation*}
    \noindent hence $\nu_{\mathrm{grid}} := \nu_{\mathrm{start}} + s^{(h_0)} z^\sharp \in B_\delta(\nu^\star)$. Thus, as soon as the current step size is smaller than $\delta/(2\sqrt{N})$, the corresponding lattice $\nu_{\mathrm{start}}+s^{(h_0)}\mathbb{Z}^N$ contains a point in $B_\delta(\nu^\star)$.
    \item[(iii)] \emph{Almost-sure eventual entry via restarts.} By design, each run of the algorithm restarts from an independent draw $\nu^{(r, 0)} \sim \mu$ on $\Theta_0$. Since $B_\delta(\nu^\star)$ is an open set contained in $\Theta_0$, and $\mu$ has support on the entire compact set $\Theta_0$, it follows that $\mu\big(B_\delta(\nu^\star)\big)>0$. As a result, the independent draws across runs yields
    \begin{equation*}
        \sum_{r=1}^\infty \mathbb{P}\!\left(\nu_0^{(r)}\in B_\delta(\nu^\star)\right)
        \;=\; \sum_{r=1}^\infty \mu\big(B_\delta(\nu^\star)\big)
        \;=\; \infty.
    \end{equation*}
    \noindent By the second Borel-Cantelli lemma~\citep{billingsley1995probability}, with probability one, there are many runs $r$ for which the restart point satisfies $\nu{(r, 0)}\in B_\delta(\nu^\star)$.
\end{enumerate}

Let us now carefully look at BOOOM iteration steps by combining the consequences obtained from the above three scenarios. If each run of the algorithm performs at most a fixed, finite number of iterations, then at the end of each such run, the algorithm must restart from an independent draw of the initial value $\nu^{(r,0)}$. This connects to scenario (iii) above, and in this case, we must have a run $r_0$ such that $\nu^{(r_0, 0)} \in B_\delta(\nu^\star)$, as there are infinitely many such choices. Suppose, on the other hand, there is a run $r$ such that the algorithm performs an $T_r$ number of iterations, such that $T_r$ is unbounded, i.e., $T_r \to \infty$. By scenario (i), such a run can spend at most $(M/s)^N$ iterations at step size $s$, and after that, it must decrease the step size. As a result, till it reaches the step size $s^{(h_0)}$ (as asserted in scenario (ii)), there can be at most
\begin{equation*}
    M^N s_{\text{initial}}^{-N} \left( 1 + \rho^{-N} + \rho^{-2N} + \dots + \rho^{-2h_0 N} \right) \leq M^N s_{\text{initial}}^{-N}\frac{1}{1-\rho^N},
\end{equation*}
\noindent iterations within the $r$-th run, since $\rho > 1$. Note that, this uniform bound on the total number of iterations is applicable for any number of iterations (potentially allowing $h \to \infty$ and $s^{(h)} \to 0$). As a result, either the algorithm needs to move on to the next $(r+1)$-th run before $M^N s_{\text{initial}}^{-N}\frac{1}{1-\rho^N}$ iterations, or must visit every point on $\nu_{\mathrm{start}} + s^{(h_0)}\mathbb{Z}^N$. The first case cannot occur as we can choose sufficiently large $r$ such that $T_r > M^N s_{\text{initial}}^{-N}\frac{1}{1-\rho^N}$. If the second case occurs, then by scenario (ii), we will visit $B_\delta(\nu^\ast)$ during one such iteration in run $r$.

\end{proof}

\subsection{Proof of Theorem  \ref{thm:booom-global-prob}}
\begin{proof}
Fix any $\epsilon > 0$. By Lemma~\ref{thm:booom-open-ball} and continuity of $\widetilde{\Phi}$ away from wrapping seams, with probability $1$, there exists a finite run $r_0$ and iteration $h_0$ such that $\nu^{(r_0, h_0)}$ enters the small ball $B_\epsilon(\nu^\dagger)$ in the angle space, on which $D\widetilde{\Phi}$ is of full rank. Hence, $\widetilde{\Phi}$ is a $\mathcal{C}^1$ submersion from $B_\epsilon(\nu^\dagger)$ onto $\mathcal{B}_\epsilon := \widetilde{\Phi}(B_\epsilon(\nu^\dagger)) \subset \mathcal{V}$, where $\mathcal{V}$ is the normal neighborhood from Lemma~\ref{thm:sosc-local}. Since $\widetilde{\Phi}$ is continuous, for all sufficiently small $\epsilon$, the objective function $f$ is $\mathcal{C}^2$ and locally strongly geodesically convex on $\mathcal{B}_\epsilon$. Since $g = f \circ \widetilde{\Phi}$, it is $\mathcal{C}^1$ and convex on $B_\epsilon(\nu^\dagger)$. As a result, by the continuity and convexity of $g$, for all sufficiently small $\epsilon > 0$, we have 
\begin{equation}
    \sup_{ \nu \in B_\epsilon(\nu^\dagger)} g(\nu) \leq \inf_{\nu \notin B_\epsilon(\nu^\dagger)} g(\nu) + \tau_1,
    \label{eqn:stationarity-bound}
\end{equation}
\noindent where $\tau_1$ is the progress threshold for within-run convergence.

Now consider the BOOOM iterations starting from $\nu^{(r_0, h_0)}$. Within this run, the RMPS stage cannot reduce $g$ sufficiently by going out of $B_\epsilon(\nu^\dagger)$, because of~\eqref{eqn:stationarity-bound}. Therefore, it must continue its search by decreasing the step sizes $s^{(h)} \downarrow 0$, and moving in the direction where the objective function is improved. This RMPS stage stops when it cannot reduce $g$ sufficiently by coordinate polling, at which point, invoking Theorem~\ref{thm:booom-stationarity-angles}, we obtain that $\nabla g = 0$, and $\operatorname{grad} f=0$ at the image point. Local strong geodesic convexity and the uniqueness in Lemma~\ref{thm:sosc-local} force the image iterate to lie at $Q^\dagger$. Since $Q^\dagger$ is the global minimizer, by restarting to another run, RMPS never yields a better value, effectively convering the iteration at $Q^\dagger$. Since $\epsilon > 0$ was arbitrary, the almost sure convergence is now immediate. 
\end{proof}

\subsection{Proof of Theorem~\ref{thm:booom-rate}}

\begin{proof}
Fix $k\ge 0$ and consider the (unsuccessful) poll at $\nu_k$ with step size $s_r$. By $L$-smoothness (the descent lemma; see \citealp[Chap.~2]{Nesterov2004} and, \citealp[Lem.~5.7]{Beck2017}) for each coordinate $i$,
\begin{equation*}
    g\!\left(\nu_k \pm s_k e_i\right)\ \le\ g\!\left(\nu_k\right)\ \pm\ s_k\,\partial_i g\!\left(\nu_k \right)\ +\ \frac{L}{2}\,s_k^2.    
\end{equation*}
\noindent Unsuccessful polling means $g(\nu_k \pm s_k e_i)\ge g(\nu_k)$ for both signs, hence
\begin{equation*}
    0\ \le\ g\!\left(\nu_k \pm s_k e_i\right)-g\!\left(\nu_k \right)\ \le\ \pm s_k\,\partial_i g\!\left(\nu_k \right)\ +\ \frac{L}{2}\,s_k^2,
\end{equation*}
\noindent which implies $|\partial_i g(\nu_k)|\ \le\ \tfrac{L}{2}\,s_k$ for all $i=1,\ldots,N$. Therefore
\begin{equation*}
    \big\|\nabla g(\nu_k)\big\|_2\ \le\ \frac{L}{2}\sqrt{N}\,s_k.
\end{equation*}
\noindent For convex $L$-smooth functions (see \citealp[Prop.~2.1.5]{Nesterov2004}),
\begin{equation*}
    g(\nu_k)-g(\nu^\star)\ \le\ \frac{1}{2L}\,\big\|\nabla g(\nu_k)\big\|_2^2\ \le\ \frac{L N}{8}\,s_k^2.    
\end{equation*}
\noindent Using $s_k=s_0/\rho^k$ and Bernoulli’s inequality, $\rho^{2k}=(1+(\rho^2-1))^{k}\ge 1+k(\rho^2-1)$, we obtain
\begin{equation*}
    s_k^2\ =\ \frac{s_0^2}{\rho^{2k}}\ \le\ \frac{s_0^2}{\,1+k(\rho^2-1)\,}\ \le\ \frac{s_0^2}{(\rho^2-1)(k+1)}.
\end{equation*}
\noindent Combining the last two inequalities yields
\begin{equation*}
    g(\nu_k)-g(\nu^\star)\ \le\ \frac{L N}{8}\,s_k^2\ \le\ \frac{L N\, s_0^2}{8(\rho^2-1)}\cdot \frac{1}{k+1},
\end{equation*}
\noindent which is the claim.
\end{proof}
\clearpage
\section{Additional Technical Details}\label{app:tech-details}

\subsection{Brief discussion on Riemannian manifold}\label{app:manifold-overview}

Before measuring distances, a space must support standard differential calculus. A smooth manifold $\mathcal{M}$ of dimension $n$ is a topological space equipped with an atlas of coordinate charts mapping local neighborhoods to $\mathbb{R}^n$, such that all transition maps between overlapping charts are smooth ($C^\infty$) diffeomorphisms. Because $\mathcal{M}$ is generally nonlinear, we cannot perform standard vector algebra directly on its elements (or points). Instead, we linearize the space locally.
\begin{enumerate}
    \item \textbf{The Tangent Space} ($T_p\mathcal{M}$): At any point $p \in \mathcal{M}$, the tangent space is an $n$-dimensional real vector space. Formally, it is the space of all point-derivations at $p$ (linear operators that satisfy the Leibniz rule on smooth functions defined near $p$). Geometrically, it contains all possible "velocity vectors" of smooth curves passing through $p$.
    \item \textbf{The Tangent Bundle} ($T\mathcal{M}$): The disjoint union of all tangent spaces across the manifold, forming a smooth manifold of dimension $2n$. A vector field $X$ is a smooth assignment of a tangent vector $X_p \in T_p\mathcal{M}$ to each point $p$.
\end{enumerate}

A smooth manifold becomes a Riemannian manifold $(\mathcal{M}, g)$ when equipped with a Riemannian metric $g$. The metric is a smoothly varying family of inner products. At each point $p \in \mathcal{M}$, $g_p$ is a symmetric, positive-definite bilinear form on the tangent space $g_p : T_p\mathcal{M} \times T_p\mathcal{M} \to \mathbb{R}$, also denoted by $\langle \cdot, \cdot\rangle_p$ (or just $\langle \cdot, \cdot\rangle$). For $v \in T_p\mathcal{M}$, we define its norm as $\|v\|_p = \sqrt{\langle v, v \rangle_p}$. For a smooth curve $\gamma: [a, b] \to \mathcal{M}$, its length is given by $L(\gamma) = \int_a^b \|\dot{\gamma}(t)\|_{\gamma(t)} dt$. A curve $\gamma(t)$ is a geodesic if its velocity vector field $\dot{\gamma}(t)$ is parallel transported along the curve itself. Locally, geodesics are length-minimizing curves.

For the specific case of the Stiefel manifold, we use the standard description of the Stiefel tangent space; see~\citealp[Sec.~3.5]{AbsilMahonySepulchre2008},
\begin{equation*}
    T_Q\mathrm{St}(p,d)\;=\;\Big\{\,Z\in\mathbb{R}^{p\times d}\,:\,Q^\top Z+Z^\top Q=0\,\Big\},    
\end{equation*}
\noindent and in particular $T_{Q}\mathrm{St}(p,d)$ is this subspace at the base point $Q \in \mathrm{St}(p,d)$. 

Let us consider a real-valued function $f: \mathcal{M} \to \mathbb{R}$. At a point $p \in \mathcal{M}$, the differential $\mathrm{d}f_p$ is a linear functional that maps a tangent vector $v \in T_p\mathcal{M}$ to a real number. Geometrically, it measures the rate of change of $f$ along the vector $v$. If we take a smooth curve $\gamma: (-\epsilon, \epsilon) \to \mathcal{M}$ such that $\gamma(0) = p$ and $\dot{\gamma}(0) = v$, the action of the differential is given by
$$
\mathrm{d}f_p(v) = \left. \frac{d}{dt} f(\gamma(t)) \right|_{t=0}
$$
\noindent The gradient, $\operatorname{grad} f$, is a vector field such that at a point $p \in \mathcal{M}$, it is the unique tangent vector that represents the differential $\mathrm{d}f_p$ via the metric. By the Riesz Representation Theorem, for all $v \in T_p\mathcal{M}$, the gradient satisfies
$$\langle \operatorname{grad} f(p), v \rangle_p = \mathrm{d}f_p(v)$$
\noindent In a local coordinate chart $(x^1, \dots, x^n)$, if the metric is given by the matrix $g_{ij}$ and its inverse is $g^{ij}$, the gradient is computed through the partial derivatives as 
$$
\operatorname{grad} f = \sum_{i,j} g^{ij} \frac{\partial f}{\partial x^j} \frac{\partial}{\partial x^i}.
$$

\subsection{Auxiliary mapping properties of the wrapped Givens parametrization}\label{subsec:auxiliary}

In this section, we take a closer look at the wrapped Givens parametrization of the Stiefel manifold $\mathrm{St}(p,d)$. Let $p \geq d$, $N := \binom{p}{2}$ and fix any $Q_0\in\mathrm{St}(p,d)$. Let us define the (redundant) Givens map
\begin{equation*}
    \Psi:\mathbb{R}^{N}\to \mathrm{SO}(p),\qquad 
    \Psi(\theta)=\prod_{1\le i<j\le p} R_{ij}(\theta_{ij}),
\end{equation*}
\noindent where the product is taken in any fixed order, and each $R_{ij}(\cdot)$ is the planar Givens rotation in the $(i,j)$-plane as given in~\eqref{eqn:givens-rotation}. The BOOOM parameterization from the initial estimate $Q_0$ is then given by 
\begin{equation*}
    \Phi:\mathbb{R}^{N}\to \mathrm{St}(p,d),\qquad 
    \Phi(\theta):=\Psi(\theta)\,Q_0.
\end{equation*}
\noindent We also introduce a \emph{wrapping map} to a compact fundamental box, where each coordinate is restricted modulo $2\pi$, yielding
\begin{equation*}
\mathscr{M}:\mathbb{R}^{N}\to \Theta_0 = [-\pi, \pi]^N,\quad
\big(\mathscr{M}(\varphi)\big)_k = \mathrm{wrap}_{(-\pi,\pi]}\,(\varphi_k),
\end{equation*}
\noindent i.e., $\mathrm{wrap}_{(-\pi,\pi]}(t)=t-2\pi\big\lfloor\frac{t+\pi}{2\pi}\big\rfloor$. Then, we define 
\begin{equation*}
    \widetilde{\Phi}:=\Phi\circ \mathscr{M}:\mathbb{R}^{N}\to \mathrm{St}(p,d),\qquad 
g(\varphi):=f\big(\widetilde{\Phi}(\varphi)\big)=f\big(\Phi(\mathscr{M}(\varphi))\big). 
\end{equation*}

Each planar rotation $R_{ij}(\cdot)$ is $C^\infty$ in its angle and matrix multiplication is smooth, so $\Psi$ and $\Phi$ are $C^\infty$. The wrapping map $\mathscr{M}$ is coordinate-wise modulo $2\pi$ onto $(-\pi,\pi]$; it is $C^\infty$ on the interior of each $2\pi$-periodic cell and only has jump discontinuities on a measure-zero union of coordinate hyperplanes (``wrapping seams''). Consequently, $g=f\circ\widetilde{\Phi}$ is continuous everywhere and $C^1$ at any $\varphi$ whose coordinates avoid the seams. In what follows, $D\Psi(\theta)$ denotes the \emph{differential} (Jacobian) of $\Psi$ at $\theta$, understood as the linear map $D\Psi(\theta):\mathbb{R}^{N}\to T_{\Psi(\theta)}\mathrm{SO}(p)$ defined by
\begin{equation*}
    D\Psi(\theta)[v]\;=\;\left.\frac{d}{dt}\right|_{t=0}\Psi(\theta+t v),\qquad v\in\mathbb{R}^{N}.    
\end{equation*}
\noindent Similarly, $D\Phi(\theta)=dA_{Q_0}|_{\Psi(\theta)}\circ D\Psi(\theta)$ is the differential of $\Phi$ at $\theta$, where the orbit map $A_{Q_0}:\mathrm{SO}(p)\to\mathrm{St}(p,d)$, $U\mapsto UQ_0$, is a submersion. The wrapping map is used for two reasons that also motivate our assumptions in the probabilistic arguments: it compactifies angles to $(-\pi,\pi]^N$, avoiding unbounded angular drift while preserving the image, and it localizes all non-smoothness to a measure-zero set of seams.

\clearpage
\section{Tables from benchmark experiments}\label{app:tables}
\renewcommand{\thefigure}{B.\arabic{figure}} 
\renewcommand{\thetable}{B.\arabic{table}}   
\setcounter{figure}{0}
\setcounter{table}{0}

\begin{table}[!htbp]
\centering
\resizebox{0.8\columnwidth}{!}{%
\begin{tabular}{c|c|cc|cc|cc}
\hline
\multirow{2}{*}{Dimension} & \multirow{2}{*}{Methods} & \multicolumn{2}{c|}{Random} & \multicolumn{2}{c|}{Toeplitz} & \multicolumn{2}{c}{Block Diagonal} \\ \cline{3-8} 
 &  & \multicolumn{1}{l}{median} & \multicolumn{1}{l|}{IQR} & \multicolumn{1}{l}{median} & \multicolumn{1}{l|}{IQR} & \multicolumn{1}{l}{median} & \multicolumn{1}{l}{IQR} \\ \hline
\multirow{2}{*}{$p=20, d=10$} & \multicolumn{1}{l|}{BOOOM} & \textbf{31.8135} & 1.7572 & \textbf{25.644} & 10.0385 & \textbf{27.6086} & 2.5414 \\
 & SDP & 28.4344 & 3.6470 & 13.1205 & 6.2606 & 22.7762 & 13.8479 \\ \hline
\multirow{2}{*}{$p=50, d=40$} & \multicolumn{1}{l|}{BOOOM} & \textbf{145.3178} & 5.0502 & \textbf{160.8581} & 94.1405 & \textbf{133.8839} & 22.9438 \\
 & SDP & 87.7655 & 20.7423 & 53.9247 & 72.0434 & 93.8959 & 58.9204 \\ \hline
\end{tabular}%
}
\caption{Maximization of heterogeneous quadratic forms: Performance comparison of BOOOM and the semidefinite programming (SDP) relaxation across two problem dimensions $(p,d) \in \{(20,10),(50,40)\}$ under three matrix structures (Random, Toeplitz, and Block Diagonal). The table reports the median and interquartile range (IQR) of the achieved objective values across $10$ Monte Carlo replicates. Larger objective values indicate better optimization performance. Bold entries highlight the best-performing method for each 
configuration. Corresponding boxplots summarizing the distribution of results across replicates are shown in Figure~\ref{fig:HQFmaximization}.}
\label{tab:HQFmaximization}
\end{table}

\begin{table}[!htbp]
\centering
\resizebox{0.8\columnwidth}{!}{%
\begin{tabular}{c|c|cc|cc|cc|cc}
\hline
\multirow{2}{*}{Rank} & \multicolumn{1}{c|}{\multirow{2}{*}{Methods}} & \multicolumn{2}{c|}{$n=50, p=10$} & \multicolumn{2}{c|}{$n=70, p=20$} & \multicolumn{2}{c|}{$n=100, p=50$} & \multicolumn{2}{c}{$n=100, p=100$} \\ \cline{3-10} 
 & \multicolumn{1}{c|}{} & \multicolumn{1}{l}{Median} & \multicolumn{1}{l|}{IQR} & \multicolumn{1}{l}{Median} & \multicolumn{1}{l|}{IQR} & \multicolumn{1}{l}{Median} & \multicolumn{1}{l|}{IQR} & \multicolumn{1}{l}{Median} & \multicolumn{1}{l}{IQR} \\ \hline
 & BOOOM & \textbf{0.1895} & 0.0385 & \textbf{0.1244} & 0.0232 & \textbf{0.0612} & 0.0065 & \textbf{0.0385} & 0.0037 \\
$d=5$ & AccAlt projection & \textbf{0.3295} & 0.1380 & \textbf{0.2849} & 0.0580 & \textbf{0.2233} & 0.0723 & \textbf{0.1746} & 0.0561 \\
 & GoDec+ & 0.6386 & 0.3758 & 0.7270 & 0.1905 & 0.5773 & 0.2235 & 0.7075 & 0.5586 \\
 & LRSD TNNSR & 1.2822 & 1.2537 & 0.7119 & 0.4260 & 0.7211 & 1.0122 & 0.6673 & 0.5159 \\ \hline
 & BOOOM & \textbf{0.6267} & 0.3850 & \textbf{0.2182} & 0.0373 & \textbf{0.1024} & 0.0059 & \textbf{0.1196} & 0.0649 \\
$d=10$ & AccAlt projection & \textbf{0.3564} & 0.1291 & \textbf{0.3904} & 0.1355 & \textbf{0.3238} & 0.0571 & \textbf{0.3015} & 0.0860 \\
 & GoDec+ & 0.6515 & 0.4076 & 0.8652 & 0.3018 & 0.8142 & 0.2322 & 0.9016 & 0.5044 \\
\multicolumn{1}{l|}{} & LRSD TNNSR & 1.3058 & 1.0008 & 0.8668 & 0.5693 & 0.8125 & 0.8964 & 0.8704 & 0.3522 \\ \hline
\end{tabular}%
}
\caption{Low-rank and sparse matrix decomposition: Performance comparison of BOOOM, BOOOM-parallel, AccAlt projection, GoDec+, and LRSD-TNNSR for recovering the low-rank component $L$ from synthetic data matrices $X \in \mathbb{R}^{n \times p}$. The table reports the median and interquartile range (IQR) of the mean absolute error (MAE) between the estimated low-rank matrix $\hat{L}$ and the true matrix $L$ across Monte Carlo replicates. Results are shown for rank settings $d \in \{5,10\}$ under four matrix dimensions $(n,p) \in \{(50,10),(70,20),(100,50),(100,100)\}$. Lower MAE indicates more accurate recovery of the low-rank structure. Bold entries highlight the top two best-performing methods within each configuration. Corresponding boxplots summarizing the distribution of results across replicates are shown in Figure~\ref{fig: low-rank decomposition}.}
\label{tab:Low_rank_decomposition}
\end{table}

\begin{table}[]
\centering





\resizebox{\columnwidth}{!}{%
\begin{tabular}{c|c|cc|cc||c|c|cc|cc}
\hline
\multirow{2}{*}{Dimension} & \multirow{2}{*}{Methods} & \multicolumn{2}{c|}{Objective} & \multicolumn{2}{c||}{Amari} & \multirow{2}{*}{Dimension} & \multirow{2}{*}{Methods} & \multicolumn{2}{c|}{Objective} & \multicolumn{2}{c}{Amari} \\ \cline{3-6} \cline{9-12}
 & & Median & IQR & Median & IQR & & & Median & IQR & Median & IQR \\ \hline

\multirow{4}{*}{$n = 50, p = 20$} 
 & BOOOM   & \textbf{8.253} & 0.0320 & 0.357 & 0.0151 & \multirow{4}{*}{$n = 125, p = 50$} 
 & BOOOM   & \textbf{19.585} & 0.4106 & 0.317 & 0.0061 \\
 & FastICA & \textbf{7.520} & 0.0303 & 0.345 & 0.0278 &
 & FastICA & \textbf{18.716} & 0.0639 & 0.315 & 0.0078 \\
 & RunICA  & 7.444          & 0.1276 & \textbf{0.329} & 0.0214 &
 & RunICA  & 15.266          & 0.1914 & \textbf{0.279} & 0.0136 \\
 & Picard  & 5.795          & 0.1046 & \textbf{0.274} & 0.0139 &
 & Picard  & 13.966          & 0.0497 & \textbf{0.271} & 0.0010 \\ \hline

\multirow{4}{*}{$n = 200, p = 20$} 
 & BOOOM   & \textbf{7.882} & 0.0245 & 0.373 & 0.0145 & \multirow{4}{*}{$n = 500, p = 50$} 
 & BOOOM   & \textbf{18.967} & 0.0771 & 0.324 & 0.0043 \\
 & FastICA & \textbf{7.398} & 0.0362 & 0.346 & 0.0123 &
 & FastICA & \textbf{18.540} & 0.1204 & 0.311 & 0.0095 \\
 & RunICA  & 7.310          & 0.0657 & \textbf{0.329} & 0.0132 &
 & RunICA  & 18.272          & 0.2169 & \textbf{0.296} & 0.0075 \\
 & Picard  & 6.584          & 0.0691 & \textbf{0.196} & 0.0145 &
 & Picard  & 16.174          & 0.1982 & \textbf{0.163} & 0.0098 \\ \hline

\end{tabular}}
\caption{Independent Component Analysis: Performance comparison of BOOOM with three widely used ICA algorithms (FastICA, RunICA, and Picard) across four dimensional configurations 
$(n,p) \in \{(50,20), (200,20), (125,50), (500,50)\}$. For each setting, the table reports the median and interquartile range (IQR) over 10 Monte Carlo replicates for two evaluation metrics: the optimization objective value and the Amari distance measuring source recovery accuracy. Lower values indicate better performance for both metrics. Bold entries highlight the top two best-performing methods within each configuration. Corresponding boxplots summarizing the distribution of results across replicates are shown in Figure~\ref{fig:ica_results}.}
\label{tab:ica_results}
\end{table}

\begin{table}[]
\centering
\resizebox{0.8\columnwidth}{!}{%
\begin{tabular}{c|c|cc||c|c|cc}
\hline
Dimension & Methods & Median & IQR & Dimension & Methods & Median & IQR \\ \hline

\multirow{2}{*}{$n=30, p =5$} 
 & BOOOM & -0.0266 & 0.0023 & \multirow{2}{*}{$n=80, p =20$} 
 & BOOOM & \textbf{-0.0827} & 0.0045 \\
 & rotatefactors & -0.0266 & 0.0023 & 
 & rotatefactors & -0.0825 & 0.0053 \\ \hline

\multirow{2}{*}{$n=60, p =5$} 
 & BOOOM & -0.0062 & 0.0002 & \multirow{2}{*}{$n=150, p =20$} 
 & BOOOM & -0.0203 & 0.0017 \\
 & rotatefactors & -0.0062 & 0.0002 & 
 & rotatefactors & -0.0203 & 0.0017 \\ \hline

\multirow{2}{*}{$n=50, p =10$} 
 & BOOOM & -0.0435 & 0.0099 & \multirow{2}{*}{$n=120, p =30$} 
 & BOOOM & \textbf{-0.0824} & 0.0066 \\
 & rotatefactors & -0.0435 & 0.0099 & 
 & rotatefactors & -0.0818 & 0.0063 \\ \hline

\multirow{2}{*}{$n=100, p =10$} 
 & BOOOM & -0.0103 & 0.0007 & \multirow{2}{*}{$n=120, p =30$} 
 & BOOOM & -0.0270 & 0.0015 \\
 & rotatefactors & -0.0103 & 0.0007 & 
 & rotatefactors & -0.0270 & 0.0015 \\ \hline

\end{tabular}}
\caption{Varimax factor rotation: Performance comparison of 
BOOOM and the classical MATLAB implementation \texttt{rotatefactors} 
across eight dimensional configurations 
$(n,p) \in \{(30,5),(60,5),(50,10),(100,10),(80,20),(150,20),(120,30),(200,30)\}$. 
The table reports the median and interquartile range (IQR) of the minimized 
objective value $-V(R)$ over $10$ Monte Carlo replicates. Lower values 
indicate better optimization performance. Bold entries highlight the unique
best-performing method within each configuration. Corresponding boxplots 
summarizing the distribution of results across replicates are shown in 
Figure~\ref{fig:varimax_results}.}
\label{tab:varimax_results}
\end{table}

\begin{table}[]
\centering
\resizebox{0.45\columnwidth}{!}{%
\begin{tabular}{c|c|cc|cc}
\hline
\multirow{2}{*}{$p$} & \multirow{2}{*}{Methods} & \multicolumn{2}{c|}{$m=5$} & \multicolumn{2}{c}{$m=10$} \\ \cline{3-6} 
 &  & Median & IQR & Median & IQR \\ \hline
\multirow{4}{*}{20} & BOOOM & \textbf{7.3909} & 1.0087 & \textbf{17.1263} & 0.6438 \\
 & JacobiAJD & 9.5176 & 0.9737 & 20.5208 & 1.6747 \\
 & RiemGD & 7.4094 & 1.0365 & \textbf{17.1263} & 0.6495 \\
 & RTR & \textbf{7.3909} & 1.0087 & \textbf{17.1263} & 0.6438 \\ \hline
\multirow{4}{*}{50} & BOOOM & \textbf{44.8938} & 0.9084 & \textbf{108.7274} & 2.9948 \\
 & JacobiAJD & 55.8221 & 1.5184 & 124.5165 & 4.2047 \\
 & RiemGD & 45.3430 & 0.6774 & 109.1435 & 3.1041 \\
 & RiemTR & \textbf{44.7387} & 1.0154 & \textbf{108.5678} & 3.0796 \\ \hline
\end{tabular}}
\caption{Orthogonal joint diagonalization: Performance comparison 
of BOOOM, the classical Jacobi AJD algorithm, Riemannian gradient descent 
(RiemGD), and the Riemannian trust-region method (RiemTR) across four problem 
configurations $(p,m) \in \{(20,5),(20,10),(50,5),(50,10)\}$, where $p$ 
denotes the matrix dimension and $m$ the number of jointly diagonalized 
matrices. The table reports the median and interquartile range (IQR) of the 
achieved objective value in \eqref{eqn:AJD} across $10$ Monte Carlo 
replicates. Lower objective values correspond to better minimization of the 
off-diagonal criterion. Bold entries highlight the top two best-performing methods 
within each configuration. Corresponding boxplots summarizing the distribution 
of results across replicates are shown in Figure~\ref{fig:AJD_results}.}
\label{tab:AJD_results}
\end{table}

\begin{table}[]
\centering
\resizebox{0.6\columnwidth}{!}{%
\begin{tabular}{c|c|cc|cc}
\hline
\multirow{2}{*}{Dimension} & \multirow{2}{*}{Methods} & \multicolumn{2}{c|}{\begin{tabular}[c]{@{}c@{}}Objective\\ gap\end{tabular}} & \multicolumn{2}{c}{\begin{tabular}[c]{@{}c@{}}KKT\\ residual\end{tabular}} \\ \cline{3-6} 
 &  & Median & IQR & Median & IQR \\ \hline
\multirow{2}{*}{$p=20$} & BOOOM & \textbf{$<$ 0.0001} & 0.0001 & \textbf{0.0038} & 0.0027 \\
 & RCG & 0.6268 & 0.3295 & 0.6003 & 0.2462 \\ \hline
\multirow{2}{*}{$p=50$} & BOOOM & \textbf{$<$ 0.0001} & $<$ 0.0001 & \textbf{0.0075} & 0.0015 \\
 & RCG & 0.8889 & 0.1998 & 0.8219 & 0.2799 \\ \hline
\multirow{2}{*}{$p=80$} & BOOOM & \textbf{0.0001} & 0.0001 & \textbf{0.0099} & 0.0023 \\
 & RCG & 1.0091 & 0.2142 & 0.9502 & 0.4419 \\ \hline
\multirow{2}{*}{$p=100$} & BOOOM & \textbf{0.0001} & 0.0001 & \textbf{0.0114} & 0.0027 \\
 & RCG & 1.0138 & 0.2936 & 1.0909 & 0.3937 \\ \hline
\end{tabular}}
\caption{Reduced Kohn--Sham Rayleigh--Ritz optimization: Performance 
comparison of BOOOM and the Riemannian conjugate gradient method (RCG) across 
four reduced subspace dimensions $p \in \{20,50,80,100\}$ with $d=2$ extracted 
eigenvectors. The table reports the median and interquartile range (IQR) over 
$10$ Monte Carlo replicates for two evaluation metrics: the objective value 
gap relative to the optimal Rayleigh--Ritz solution obtained from the exact 
eigenvalue decomposition of $H_{\mathrm{red}}$, and the Karush--Kuhn--Tucker 
(KKT) residual measuring violation of the first-order optimality conditions. 
Lower values indicate better solution accuracy and optimality. Bold entries 
highlight the best-performing method within each configuration. Corresponding 
boxplots summarizing the distribution of results across replicates are shown 
in Figure~\ref{fig:ks_results}.}
\label{tab:ks_results}
\end{table}




\vskip 0.2in
\clearpage
\bibliographystyle{plainnat}
\bibliography{sample}

\begin{thebibliography}{80}
\providecommand{\natexlab}[1]{#1}
\providecommand{\url}[1]{\texttt{#1}}
\expandafter\ifx\csname urlstyle\endcsname\relax
  \providecommand{\doi}[1]{doi: #1}\else
  \providecommand{\doi}{doi: \begingroup \urlstyle{rm}\Url}\fi

\bibitem[Ablin et~al.(2018)Ablin, Cardoso, and Gramfort]{ablin2018faster}
P.~Ablin, J.~Cardoso, and A.~Gramfort.
\newblock Faster ica under orthogonal constraint.
\newblock \emph{IEEE International Conference on Acoustics, Speech and Signal Processing (ICASSP)}, page 4464–4468, 2018.

\bibitem[Absil et~al.(2008)Absil, Mahony, and Sepulchre]{AbsilMahonySepulchre2008}
P.~Absil, R.~Mahony, and R.~Sepulchre.
\newblock \emph{Optimization Algorithms on Matrix Manifolds}.
\newblock Princeton University Press, Princeton, NJ, 2008.

\bibitem[Amari et~al.(1995)Amari, Cichocki, and Yang]{Amari1995}
S.~Amari, A.~Cichocki, and H.~Yang.
\newblock A new learning algorithm for blind signal separation.
\newblock \emph{Advances in Neural Information Processing Systems}, 8:\penalty0 757--763, 1995.

\bibitem[{Athalye et al.}(2018)]{athalye2018synthesizing}
{Athalye et al.}
\newblock Synthesizing robust adversarial examples.
\newblock In \emph{International conference on machine learning}, pages 284--293. PMLR, 2018.

\bibitem[{Avram et al.}(2025)]{avram2025metabolomic}
{Avram et al.}
\newblock Metabolomic exploration of colorectal cancer through amino acids and acylcarnitines profiling of serum samples.
\newblock \emph{Cancers}, 17\penalty0 (3):\penalty0 427, 2025.

\bibitem[Beck(2017)]{Beck2017}
A.~Beck.
\newblock \emph{First-Order Methods in Optimization}, volume~25 of \emph{MOS-SIAM Series on Optimization}.
\newblock Society for Industrial and Applied Mathematics, Philadelphia, PA, 2017.

\bibitem[{Bednarz-Misa et al.}(2020)]{bednarz2020arginine}
{Bednarz-Misa et al.}
\newblock L-arginine/no pathway metabolites in colorectal cancer: relevance as disease biomarkers and predictors of adverse clinical outcomes following surgery.
\newblock \emph{Journal of Clinical Medicine}, 9\penalty0 (6):\penalty0 1782, 2020.

\bibitem[Bell and Sejnowski(1995)]{bell1995information}
A.~Bell and T.~Sejnowski.
\newblock An information-maximization approach to blind separation and blind deconvolution.
\newblock \emph{Neural Computation}, 7\penalty0 (6):\penalty0 1129--1159, 1995.

\bibitem[{Berger-Tal et al.}(2014)]{BergerTal2014}
{Berger-Tal et al.}
\newblock The exploration-exploitation dilemma: A multidisciplinary framework.
\newblock \emph{PLoS ONE}, 9\penalty0 (4):\penalty0 e95693, 2014.

\bibitem[Bernasocchi and Mostoslavsky(2024)]{bernasocchi2024subcellular}
T.~Bernasocchi and R.~Mostoslavsky.
\newblock Subcellular one carbon metabolism in cancer, aging and epigenetics.
\newblock \emph{Frontiers in Epigenetics and Epigenomics}, 2:\penalty0 1451971, 2024.

\bibitem[Bethke(1980)]{Bethke1980}
A.~Bethke.
\newblock Genetic algorithms as function optimizers, 1980.

\bibitem[Billingsley(1995)]{billingsley1995probability}
P.~Billingsley.
\newblock \emph{Probability and Measure}.
\newblock Wiley Series in Probability and Mathematical Statistics. John Wiley \& Sons, Hoboken, NJ, 3rd edition, 1995.

\bibitem[Boumal(2023)]{Boumal2023}
N.~Boumal.
\newblock \emph{An Introduction to Optimization on Smooth Manifolds}.
\newblock Cambridge University Press, Cambridge, 2023.

\bibitem[{Boumal et al.}(2014)]{boumal2014manopt}
{Boumal et al.}
\newblock Manopt, a {Matlab} toolbox for optimization on manifolds.
\newblock \emph{Journal of Machine Learning Research}, 15\penalty0 (42):\penalty0 1455--1459, 2014.

\bibitem[{Bryniarski et al.}(2021)]{bryniarski2021evading}
{Bryniarski et al.}
\newblock Evading adversarial example detection defenses with orthogonal projected gradient descent.
\newblock \emph{arXiv preprint arXiv:2106.15023}, 2021.

\bibitem[Burer and Monteiro(2003)]{burer2003nonlinear}
S.~Burer and R.~Monteiro.
\newblock A nonlinear programming algorithm for solving semidefinite programs via low-rank factorization.
\newblock \emph{Mathematical Programming}, 95:\penalty0 329--357, 2003.

\bibitem[Cai et~al.(2019)Cai, Cai, and Wei]{cai2019accelerated}
H.~Cai, J.~Cai, and K.~Wei.
\newblock Accelerated alternating projections for robust principal component analysis.
\newblock \emph{Journal of Machine Learning Research}, 20\penalty0 (20):\penalty0 1--33, 2019.

\bibitem[{Candes et al.}(2011)]{candes2011robust}
{Candes et al.}
\newblock Robust principal component analysis?
\newblock \emph{Journal of the ACM}, 58\penalty0 (3):\penalty0 1--37, 2011.

\bibitem[Cardoso(1998)]{cardoso1998blind}
J.~Cardoso.
\newblock Blind signal separation: statistical principles.
\newblock \emph{Proceedings of the IEEE}, 86\penalty0 (10):\penalty0 2009--2025, 1998.

\bibitem[Cardoso and Souloumiac(1996)]{cardoso1996jacobi}
J.~Cardoso and A.~Souloumiac.
\newblock Jacobi angles for simultaneous diagonalization.
\newblock \emph{SIAM Journal on Matrix Analysis and Applications}, 17\penalty0 (1):\penalty0 161--164, 1996.

\bibitem[{Coker et al.}(2022)]{coker2022altered}
{Coker et al.}
\newblock Altered gut metabolites and microbiota interactions are implicated in colorectal carcinogenesis and can be non-invasive diagnostic biomarkers.
\newblock \emph{Microbiome}, 10\penalty0 (1):\penalty0 35, 2022.

\bibitem[Collette and Siarry(2003)]{collette2003multiobjective}
Y.~Collette and P.~Siarry.
\newblock \emph{Multiobjective Optimization: Principles and Case Studies}.
\newblock Springer-Verlag, Berlin, Heidelberg, 2003.
\newblock \doi{10.1007/978-3-662-08883-8}.

\bibitem[Conn et~al.(2009)Conn, Scheinberg, and Vicente]{Conn2009}
A.~Conn, K.~Scheinberg, and L.~Vicente.
\newblock Introduction to derivative-free optimization.
\newblock \emph{Mathematics without boundaries: Surveys in interdisciplinary research, MOS-SIAM Series on Optimization, SIAM}, 2009.

\bibitem[Das(2021)]{Das2021}
P.~Das.
\newblock Recursive modified pattern search on high-dimensional simplex : A blackbox optimization technique.
\newblock \emph{The Indian Journal of Statistics - Sankhya B}, 83:\penalty0 440--483, 2021.

\bibitem[Das(2023)]{Das2023RMPSH}
P.~Das.
\newblock Black-box optimization on hyper-rectangle using recursive modified pattern search and application to {ROC}-based classification problem.
\newblock \emph{Sankhya B}, 85:\penalty0 365--404, 2023.
\newblock \doi{10.1007/s13571-023-00312-w}.

\bibitem[Das and Ghosal(2017)]{Das2017b}
P.~Das and S.~Ghosal.
\newblock Analyzing ozone concentration by bayesian spatio‐temporal quantile regression.
\newblock \emph{Environmetrics}, 28\penalty0 (4):\penalty0 e2443, 2017.

\bibitem[{Das et al.}(2022)]{Das2022}
{Das et al.}
\newblock Estimating the optimal linear combination of predictors using spherically constrained optimization.
\newblock \emph{BMC Bioinformatics}, 23\penalty0 (Suppl 3):\penalty0 436, 2022.

\bibitem[{Das et al.}(2023{\natexlab{a}})]{Das2023}
{Das et al.}
\newblock Utilizing biologic disease-modifying anti-rheumatic treatment sequences to subphenotype rheumatoid arthritis.
\newblock \emph{Arthritis Research and Therapy}, 25\penalty0 (1):\penalty0 1--7, 2023{\natexlab{a}}.

\bibitem[{Das et al.}(2023{\natexlab{b}})]{Das2023MsiCOR}
{Das et al.}
\newblock Clustering sequence data with mixture {Markov} chains with covariates using multiple simplex constrained optimization routine {(MSiCOR)}.
\newblock \emph{Journal of Computational and Graphical Statistics}, 33\penalty0 (2):\penalty0 379--392, 2023{\natexlab{b}}.

\bibitem[Edelman et~al.(1998)Edelman, Arias, and Smith]{Edelman1998}
A.~Edelman, T.~Arias, and S.~Smith.
\newblock The geometry of algorithms with orthogonality constraints.
\newblock \emph{SIAM Journal on Matrix Analysis and Applications}, 20\penalty0 (2):\penalty0 303--353, 1998.

\bibitem[{Fei et al.}(2022)]{fei2022vit}
{Fei et al.}
\newblock O-vit: Orthogonal vision transformer.
\newblock \emph{arXiv preprint arXiv:2201.12133}, 2022.

\bibitem[{Feng et al.}(2019)]{feng2019supervised}
{Feng et al.}
\newblock Supervised discriminative sparse {PCA} for com-characteristic gene selection and tumor classification on multiview biological data.
\newblock \emph{IEEE Transactions on Neural Networks and Learning Systems}, 30\penalty0 (10):\penalty0 2926--2937, 2019.

\bibitem[Fermi and Metropolis(1952)]{Fermi1952}
E.~Fermi and N.~Metropolis.
\newblock Numerical solution of a minimum problem. los alamos unclassified report la–1492.
\newblock \emph{Los Alamos National Laboratory, Los Alamos, USA}, 1952.

\bibitem[Fraser(1957)]{Fraser1957}
A.~Fraser.
\newblock Simulation of genetic systems by automatic digital computers.
\newblock \emph{Australian Journal of Biological Sciences}, 10:\penalty0 484--491, 1957.

\bibitem[Geris(2012)]{Geris2012}
L.~Geris.
\newblock \emph{Computational Modeling in Tissue Engineering}.
\newblock Springer, 2012.

\bibitem[Gilman et~al.(2025)Gilman, Burer, and Balzano]{gilman2025semidefinite}
K.~Gilman, S.~Burer, and L.~Balzano.
\newblock A semidefinite relaxation for sums of heterogeneous quadratic forms on the {Stiefel} manifold.
\newblock \emph{SIAM Journal on Matrix Analysis and Applications}, 46\penalty0 (2):\penalty0 1091--1116, 2025.

\bibitem[Goedecker(1999)]{goedecker1999linear}
S.~Goedecker.
\newblock Linear scaling electronic structure methods.
\newblock \emph{Reviews of Modern Physics}, 71:\penalty0 1085--1123, 1999.

\bibitem[{Guo et al.}(2017)]{guo2017godec+}
{Guo et al.}
\newblock {GoDec}+: Fast and robust low-rank matrix decomposition based on maximum correntropy.
\newblock \emph{IEEE Transactions on Neural Networks and Learning Systems}, 29\penalty0 (6):\penalty0 2323--2336, 2017.

\bibitem[{Hashim et al.}(2020)]{hashim2020global}
{Hashim et al.}
\newblock Global metabolomics profiling of colorectal cancer in malaysian patients.
\newblock \emph{BioImpacts: BI}, 11\penalty0 (1):\penalty0 33, 2020.

\bibitem[{He et al.}(2016)]{he2016deep}
{He et al.}
\newblock Deep residual learning for image recognition.
\newblock In \emph{Proceedings of the IEEE conference on computer vision and pattern recognition}, pages 770--778, 2016.

\bibitem[{Hong et al.}(2021)]{hong2021heppcat}
{Hong et al.}
\newblock {HePPCAT}: Probabilistic {PCA} for data with heteroscedastic noise.
\newblock \emph{IEEE Transactions on Signal Processing}, 69:\penalty0 4819--4834, 2021.

\bibitem[{Hu et al.}(2020)]{hu2020moving}
{Hu et al.}
\newblock Moving object detection based on non-convex rpca with segmentation constraint.
\newblock \emph{IEEE Access}, 8:\penalty0 41026--41036, 2020.

\bibitem[{Huang et al.}(2018)]{huang2018orthogonal}
{Huang et al.}
\newblock Orthogonal weight normalization: Solution to optimization over multiple dependent stiefel manifolds in deep neural networks.
\newblock \emph{AAAI Technical Track: Machine Learning}, 32\penalty0 (1), 2018.

\bibitem[Hurwitz(1963)]{Hurwitz1963}
A.~Hurwitz.
\newblock Ueber die erzeugung der invarianten durch integration.
\newblock In \emph{Mathematische Werke}, pages 546--564. Springer, 1963.

\bibitem[Hyvarinen(1999)]{hyvarinen1999fast}
A.~Hyvarinen.
\newblock Fast and robust fixed-point algorithms for independent component analysis.
\newblock \emph{IEEE Transactions on Neural Networks}, 10\penalty0 (3):\penalty0 626--634, 1999.

\bibitem[Hyvarinen et~al.(2001)Hyvarinen, Karhunen, and Oja]{hyvarinen2000independent}
A.~Hyvarinen, J.~Karhunen, and E.~Oja.
\newblock \emph{Independent Component Analysis}.
\newblock Wiley, 2001.

\bibitem[Jennrich(2001)]{jennrich2001simple}
R.~Jennrich.
\newblock A simple general procedure for orthogonal rotation.
\newblock \emph{Psychometrika}, 66:\penalty0 289--306, 2001.

\bibitem[{Jiang et al.}(2022)]{jiang2022givens}
{Jiang et al.}
\newblock Givens coordinate descent methods for rotation matrix learning in trainable embedding indexes.
\newblock In \emph{Proceedings of the International Conference on Learning Representations (ICLR)}, 2022.

\bibitem[{Jiang et al.}(2026)]{jiang2026diversifying}
{Jiang et al.}
\newblock Diversifying counterattacks: Orthogonal exploration for robust cllp inference.
\newblock In \emph{Proceedings of the AAAI Conference on Artificial Intelligence}, volume~40, pages 5359--5368, 2026.

\bibitem[{Jiao et al.}(2022)]{Jiao2022kssolv}
{Jiao et al.}
\newblock {KSSOLV} 2.0: An efficient matlab toolbox for solving the kohn-sham equations with plane-wave basis set.
\newblock \emph{Computer Physics Communications}, 279:\penalty0 108424, 2022.

\bibitem[Jolliffe(2002)]{jolliffe2002principal}
I.~Jolliffe.
\newblock \emph{Principal Component Analysis}.
\newblock Springer, 2002.

\bibitem[Kaiser(1958)]{kairser1958varimax}
H.~Kaiser.
\newblock The varimax criterion for analytic rotation in factor analysis.
\newblock \emph{Psychometrika}, 23:\penalty0 187--200, 1958.

\bibitem[Kim et~al.(2026)Kim, Xia, and Das]{kim2026smart}
B.~Kim, Z.~Xia, and P.~Das.
\newblock {SMART-MC}: Sparse matrix estimation with covariate-based transitions in {Markov} chain modeling of multiple sclerosis disease modifying therapies.
\newblock \emph{Journal of the American Statistical Association}, 121\penalty0 (553):\penalty0 85–99, 2026.

\bibitem[Kirkpatrick et~al.(1983)Kirkpatrick, Gelatt, and Vecchi]{kirkpatrick1983optimization}
S.~Kirkpatrick, C.~Gelatt, and M.~Vecchi.
\newblock Optimization by simulated annealing.
\newblock \emph{Science}, 220\penalty0 (4598):\penalty0 671--680, 1983.

\bibitem[Kohn and Sham(1965)]{kohn1965self}
W.~Kohn and L.~Sham.
\newblock Self-consistent equations including exchange and correlation effects.
\newblock \emph{Physical Review}, 140:\penalty0 A1133--A1138, 1965.

\bibitem[Kolda et~al.(2003)Kolda, Lewis, and Torczon]{Kolda2003}
G.~Kolda, R.~Lewis, and V.~Torczon.
\newblock Optimization by direct search: New perspectives on some classical and modern methods.
\newblock \emph{SIAM Review}, 45\penalty0 (3):\penalty0 385--482, 2003.
\newblock \doi{10.1137/S003614450242889}.

\bibitem[Lee(2018)]{Lee2018}
J.~Lee.
\newblock \emph{Introduction to Riemannian Manifolds}, volume 176 of \emph{Graduate Texts in Mathematics}.
\newblock Springer, Cham, 2 edition, 2018.

\bibitem[{Li et al.}(2019)]{li2019orthogonal}
{Li et al.}
\newblock Orthogonal deep neural networks.
\newblock \emph{IEEE transactions on pattern analysis and machine intelligence}, 43\penalty0 (4):\penalty0 1352--1368, 2019.

\bibitem[{Li et al.}(2022)]{li2022metabolomic}
{Li et al.}
\newblock Metabolomic comparison of patients with colorectal cancer at different anticancer treatment stages.
\newblock \emph{Frontiers in Oncology}, 11:\penalty0 574318, 2022.

\bibitem[{Lin et al.}(2012)]{lin2012adaptive}
{Lin et al.}
\newblock Adaptive local basis set for kohn--sham density functional theory in a discontinuous galerkin framework: Total energy calculation.
\newblock \emph{Journal of Computational Physics}, 231\penalty0 (4):\penalty0 2140--2154, 2012.

\bibitem[Luo et~al.(2025)Luo, Wang, and Ren]{luo2025direct}
K.~Luo, T.~Wang, and X.~Ren.
\newblock Direct minimization on the complex stiefel manifold in kohn-sham density functional theory for finite and extended systems.
\newblock \emph{Computer Physics Communications}, 312:\penalty0 109596, 2025.

\bibitem[Nesterov(2004)]{Nesterov2004}
Y.~Nesterov.
\newblock \emph{Introductory Lectures on Convex Optimization: A Basic Course}, volume~87 of \emph{Applied Optimization}.
\newblock Kluwer Academic Publishers, Boston, MA, 2004.

\bibitem[{Netrapalli et al.}(2014)]{netrapalli2014nonconvex}
{Netrapalli et al.}
\newblock Non-convex robust pca.
\newblock In \emph{Advances in Neural Information Processing Systems (NeurIPS)}, volume~27, 2014.

\bibitem[Nielsen and Chuang(2010)]{nielsen2010quantum}
M.~Nielsen and I.~Chuang.
\newblock \emph{Quantum computation and quantum information}.
\newblock Cambridge university press, 2010.

\bibitem[Pham(2001)]{pham2001joint}
D.~Pham.
\newblock Joint approximate diagonalization of positive definite hermitian matrices.
\newblock \emph{SIAM Journal on Matrix Analysis and Applications}, 22\penalty0 (4):\penalty0 1136--1152, 2001.

\bibitem[Rios and Sahinidis(2013)]{Rios2013}
L.~Rios and N.~Sahinidis.
\newblock Derivative-free optimization: a review of algorithms and comparison of software implementations.
\newblock \emph{Journal of Global Optimization}, 56\penalty0 (3):\penalty0 1247--1293, 2013.

\bibitem[Roy et~al.(2024)Roy, Basu, and Ghosh]{roy2024robust}
S.~Roy, A.~Basu, and A.~Ghosh.
\newblock Robust principal component analysis using density power divergence.
\newblock \emph{Journal of Machine Learning Research}, 25\penalty0 (324):\penalty0 1--40, 2024.

\bibitem[Ruder(2016)]{ruder2016overview}
S.~Ruder.
\newblock An overview of gradient descent optimization algorithms.
\newblock \emph{arXiv preprint arXiv:1609.04747}, 2016.

\bibitem[Saad(2011)]{saad2011numerical}
Y.~Saad.
\newblock \emph{Numerical Methods for Large Eigenvalue Problems}.
\newblock SIAM, 2nd edition, 2011.

\bibitem[{Shi et al.}(2020)]{shi2020supervised}
{Shi et al.}
\newblock Supervised discriminative sparse {PCA} with adaptive neighbors for dimensionality reduction.
\newblock In \emph{Proceedings of the International Joint Conference on Neural Networks (IJCNN)}, pages 1--8. IEEE, 2020.

\bibitem[{Shi et al.}(2022)]{shi2022potential}
{Shi et al.}
\newblock Potential roles of serum {ATPase} and {AMPase} in predicting diagnosis of colorectal cancer patients.
\newblock \emph{Bioengineered}, 13\penalty0 (6):\penalty0 14204--14214, 2022.

\bibitem[Surjanovic and Bingham(2013)]{surjanovic2013virtual}
S.~Surjanovic and D.~Bingham.
\newblock Virtual library of simulation experiments: Test functions and datasets, 2013.
\newblock [Online]. Available: \url{https://www.sfu.ca/~ssurjano/optimization.html}.

\bibitem[Tan and Ghosal(2020)]{tan2020}
Q.~Tan and S.~Ghosal.
\newblock Bayesian quantile regression in differential equation models.
\newblock 339:\penalty0 321--334, 2020.

\bibitem[{Tevini et al.}(2022)]{tevini2022changing}
{Tevini et al.}
\newblock Changing metabolic patterns along the colorectal adenoma--carcinoma sequence.
\newblock \emph{Journal of Clinical Medicine}, 11\penalty0 (3):\penalty0 721, 2022.

\bibitem[Torczon(1997)]{torczon1997convergence}
V.~Torczon.
\newblock On the convergence of pattern search algorithms.
\newblock \emph{SIAM Journal on Optimization}, 7\penalty0 (1):\penalty0 1--25, 1997.

\bibitem[{Udo et al.}(2020)]{udo2020urinary}
{Udo et al.}
\newblock Urinary charged metabolite profiling of colorectal cancer using capillary electrophoresis-mass spectrometry.
\newblock \emph{Scientific Reports}, 10\penalty0 (1):\penalty0 21057, 2020.

\bibitem[{Vorontsov et al.}(2017)]{vorontsov2017}
{Vorontsov et al.}
\newblock On orthogonality and learning recurrent networks with long term dependencies.
\newblock In \emph{International Conference on Machine Learning (ICML)}, 2017.

\bibitem[{Xue et al.}(2019)]{xue2019low}
{Xue et al.}
\newblock Low-rank and sparse matrix decomposition via the truncated nuclear norm and a sparse regularizer.
\newblock \emph{The Visual Computer}, 35\penalty0 (11):\penalty0 1549--1566, 2019.

\bibitem[{Yang et al.}(2009)]{Yang2009kssolv}
{Yang et al.}
\newblock {KSSOLV}-a matlab toolbox for solving the kohn--sham equations.
\newblock \emph{ACM Transactions on Mathematical Software}, 36\penalty0 (2), 2009.

\bibitem[{Zhang et al.}(2023)]{Zhang2023}
{Zhang et al.}
\newblock Methods to balance the exploration and exploitation in differential evolution from different scales: A survey.
\newblock \emph{Neurocomputing}, 561:\penalty0 126899, 2023.

\end{thebibliography}

\end{document}